\documentclass[11pt,reqno]{amsart}

\usepackage{amsmath,amssymb,amsthm}
\usepackage{bm}
\usepackage{natbib}
\usepackage{graphicx}
\usepackage{multirow}
\usepackage{here}

\makeatletter

\@addtoreset{equation}{section}
\makeatletter

\newtheorem{theorem}{Theorem}[section]

\newtheorem{lemma}{Lemma}[section]
\newtheorem{proposition}{Proposition}[section]
\newtheorem{assumption}{Assumption}[section]
\theoremstyle{definition}

\newtheorem{remark}{Remark}[section]
\newtheorem{example}{Example}[section]

\newcommand{\R}{\mathbb{R}}
\newcommand{\mF}{\mathcal{F}}
\newcommand{\mG}{\mathcal{G}}

\newcommand{\mS}{\mathcal{S}}

\newcommand{\Ep}{\mathrm{E}}
\renewcommand{\Pr}{\mathrm{P}}
\renewcommand{\tilde}{\widetilde}
\renewcommand{\hat}{\widehat}

\DeclareMathOperator{\Var}{Var}
\DeclareMathOperator{\Cov}{Cov}

\DeclareMathOperator{\E}{E}

\begin{document}

\title[]{Bootstrap confidence bands for spectral estimation of L\'{e}vy densities under high-frequency observations}
\thanks{K. Kato is supported by Grant-in-Aid for Scientific Research (C) (15K03392) from the JSPS. D. Kurisu is supported by Grant-in-Aid for JSPS Research Fellow (16J06454) from the JSPS. We would like to thank Yuya Sasaki for kindly sharing the codes.}
\author[K. Kato]{Kengo Kato}
\author[D. Kurisu]{Daisuke Kurisu}

\date{First version: May 1, 2017. This version: \today}

\address[K. Kato]{
Graduate School of Economics, University of Tokyo\\
7-3-1 Hongo, Bunkyo-ku, Tokyo 113-0033, Japan.
}
\email{kkato@e.u-tokyo.ac.jp}

\address[D. Kurisu]{
Graduate School of Economics, University of Tokyo\\
7-3-1 Hongo, Bunkyo-ku, Tokyo 113-0033, Japan.
}
\email{dkurisu.mathstat@gmail.com}

\begin{abstract}
This paper develops bootstrap methods to construct uniform confidence bands for nonparametric spectral estimation of L\'{e}vy densities under high-frequency observations.
We assume that we observe $n$ discrete observations at frequency $1/\Delta  > 0$, and work with the high-frequency setup where $\Delta = \Delta_{n} \to 0$ and $n\Delta \to \infty$ as $n \to \infty$. 
We employ a spectral (or Fourier-based) estimator of the L\'{e}vy density, and develop novel implementations of  Gaussian multiplier (or wild) and empirical (or Efron's) bootstraps to construct confidence bands for the spectral estimator on a compact set that does not intersect the origin. 
We provide conditions under which the proposed confidence bands are asymptotically valid. 
Our confidence bands are shown to be asymptotically valid for a wide class  of L\'{e}vy processes.
We also develop a practical method for bandwidth selection, and conduct simulation studies to investigate the finite sample performance of the proposed confidence bands.
\end{abstract}


\maketitle

\section{Introduction}

In the financial economics literature, it has been argued that the presence of jumps plays an important role in the dynamics of financial  data such as asset returns, interest rates,  currencies, and so on \citep[cf.][]{CoTa04,Ai04,Jo04,AiJa14}. For example, \cite{Jo04} studies the dynamics of interest rate movements and  argues that the presence of jumps contributes to capturing non-normalities of increment distributions that are consistent with empirical data but diffusion models can not capture. 
A \textit{L\'{e}vy process} is a fundamental class of continuous-time stochastic processes allowing for jumps; we refer to \cite{Sa99} and \cite{Be96} as standard references on L\'{e}vy processes. From the L\'{e}vy-It\^{o} decomposition \citep[][Theorem 19.2]{Sa99}, a L\'{e}vy process is decomposed into the sum of drift, Brownian, and jump components, and the distribution of the L\'{e}vy process is completely determined by the three parameters, namely, the drift, the diffusion coefficient, and the \textit{L\'{e}vy measure}. The L\'{e}vy measure  controls the jump dynamics of the L\'{e}vy process, and therefore, inference on the L\'{e}vy measure is of particular interest. 
In this paper, we assume that the L\'{e}vy density has a Lebesgue density (L\'{e}vy density), and study inference on the L\'{e}vy density from \textit{high-frequency} observations. High-frequency data -- data collected for every minute, second, or even microsecond -- have become available due to the advancement of information technologies, and the analysis of high-frequency financial data has attracted a great deal of attentions in the financial econometrics  literature; see, e.g., \cite{AiJa14}.


To be precise, we work with the following setting. 
Let $L=(L_{t})_{t \geq 0}$ be a L\'evy process, i.e., $L$ is a stochastic process starting at $0$ with stationary independent increments and c\`{a}dl\`{a}g sample paths.
From the L\'{e}vy-Khinchin representation \citep[][Theorem 8.1]{Sa99}, $L_{t}$ has characteristic function $\varphi_{t}(u) = \Ep[e^{iuL_{t}}] = e^{t\psi (u)}, u \in \R$, where $i=\sqrt{-1}$, and 
\[
\psi (u) =- \frac{u^{2}\sigma^{2}}{2} + iu\gamma+ \int_{\R} (e^{iux} - 1- iux1_{[-1,1]} (x) ) \nu(dx).
\]
The triplet $(\sigma^{2}, \gamma, \nu)$, called the L\'{e}vy triplet, completely characterizes the distribution of the L\'{e}vy process $L$ \citep[cf.][Theorem 7.10]{Sa99}. Specifically,  $\sigma^{2} \geq 0$ is the diffusion coefficient, $\gamma \in \R$ is the drift,  and $\nu$ is the L\'{e}vy measure, i.e., a Borel measure on $\R$ such that 
\[
\int_{\R} (1 \wedge x^{2}) \nu(dx) < \infty \quad \text{and} \quad \nu (\{ 0 \})=0.
\]
For any (Borel) set $A \subset \R$, $\nu (A)$ coincides with the expected number of jumps falling in $A \setminus \{ 0 \}$ in the unit time: 
\[
\nu (A) = \Ep \left [ \sum_{0 < t \leq 1} 1(L_{t} - L_{t-} \in A \setminus \{ 0 \}) \right ],
\]
where $L_{t-} = \lim_{s \uparrow t} L_{s}$ (recall that $L$ has at most countably many jumps on $(0,t]$ for any $t > 0$).
In this paper, we assume that the L\'{e}vy measure has Lebesgue density $\rho$, called the L\'{e}vy density, i.e., $\nu (dx) = \rho (x) dx$. 
Furthermore, we assume that we observe discrete observations $L_{j\Delta}, \ j =1,\dots,n$
at frequency $1/\Delta > 0$, and work with the high-frequency setup where $\Delta = \Delta_{n} \to 0$ and $n\Delta \to \infty$ as $n \to \infty$. Since we are interested in estimation of the L\'{e}vy measure (or more precisely its Lebesgue density), we require $n\Delta \to \infty$. Heuristically, this can be understood from the observation that, within any fixed time interval, say the unit time, the L\'{e}vy process $(L_{t})_{t \in [0,1]}$ has only finitely many jumps that fall in a local neighborhood not containing the origin, so that even if the whole path $(L_{t})_{t \in [0,1]}$ could be observed, there are only finitely many data that can be used to estimate the L\'{e}vy measure at the local neighborhood. Concretely, we have in mind that the unit time is one day, and if we have 6.5 trading hours in a day and take 5 minutes as a time span, then $\Delta = 5/(6.5 \times 60) = 1/78 \approx 0.013$; each year has around 252 business days, and so we have $78 \times 252 = 19656$ observations per a year.

Under this setup, the goal of this paper is to develop bootstrap methods to construct confidence bands for the L\'{e}vy density. Since the L\'{e}vy density can blow up around the origin, we focus on confidence bands on a compact set that does not intersect the origin. We employ a spectral (or Fourier-based) estimator of the L\'{e}vy density, and develop novel implementations of Gaussian multiplier and empirical bootstraps to construct confidence bands for the spectral estimator. We provide conditions under which the proposed confidence bands are asymptotically valid. 
Notably, our confidence bands are shown to be asymptotically valid for a wide class of L\'{e}vy processes, including compound Poisson processes, (Variance-) Gamma processes, Inverse Gaussian processes, tempered stable processes, and Normal Inverse Gaussian processes with or without Brownian components.\footnote{For tempered stable processes, however, the stability index has is at most $1$ for a technical reason.}  
Confidence bands provide a simple graphical description of the accuracy of a nonparametric curve estimator, thereby quantifying uncertainties of the estimator simultaneously over (in most cases continuum of) designs points, which is of practical importance in statistical analysis. 
Despite extensive studies on consistent estimation of the L\'{e}vy density, however, research on  confidence intervals or bands for the L\'{e}vy density is relatively scarce -- see a literature review below. In particular, to the best of our knowledge, this is the first paper to derive bootstrap-based confidence bands for the L\'{e}vy density. In addition to the theoretical results, we also develop a practical method for bandwidth selection, inspired by \cite{BiDuHoMu07}, and conduct simulation studies to investigate the finite sample performance of the proposed confidence bands.

The literature on nonparametric estimation of L\'{e}vy measures or densities is broad. Recent contributions include \cite{Sh06}, \cite{Fi09}, \cite{CoGe09,CoGe10,CoGe15}, \cite{KaRe10}, \cite{Du13}, and \cite{BeLa15} under the high-frequency setup (i.e., $\Delta = \Delta_{n} \to 0$ as $n \to \infty$), and \cite{EsGuSp07}, \cite{NeRe09}, \cite{Gu09}, \cite{ChDeHa10}, \cite{CoGe10}, \cite{KaRe10}, \cite{Be11}, \cite{Gu12}, \cite{Ka14}, \cite{Tr15}, and \cite{BeRe15} under the low-frequency setup (i.e., $\Delta > 0$ is fixed). 
\cite{Jovava05} study nonparametric estimation of the L\'{e}vy measure for a L\'{e}vy driven  Ornstein-Uhlenbeck process under high and low frequency observations.
\cite{NiRe12} and \cite{NiReSoTr16} study estimation of  distribution functions such as $\int_{-\infty}^{\cdot} (1 \wedge x^{2}) \nu (dx)$, and prove Donsker-type functional limit theorems for distributional function estimators under low- and high-frequency setups, respectively. We also refer to \cite{BuVe13}, \cite{Ve14}, \cite{BuHoVeDe17}, and \cite{HoVe17} for inference on L\'{e}vy measures. 
However, none of these papers  studies confidence bands for L\'{e}vy densities.  

To the best of our knowledge, \cite{Fi11} and its follow-up paper \cite{KoPa16} are the only references that derive uniform confidence bands for L\'{e}vy densities. They work with the high-frequency setup, but employ sieve (or projection) estimators based on the observation that $\Delta^{-1} \Pr (L_{\Delta} \geq x) \approx \nu([x,\infty))$ for $x > 0$ \citep[see also][]{Fi09}, which  are substantially different from our spectral estimator. 
So, first of all, their results do not cover ours. Similarly to \cite{Sm50} and \cite{BiRo73}, \cite{Fi11} proves that the supremum deviation of the sieve estimator, suitably normalized, converges in distribution to a Gumbel distribution, by using the Koml\'{o}s-Major-Tusn\'{a}dy (KMT)  approximation of the empirical distribution function by Brownian bridges \citep{KoMaTu75}, combined with extreme value theory. \cite{Fi11} uses the Gumbel approximation to construct analytic confidence bands for the L\'{e}vy density, but does not study bootstrap-based confidence bands. However, since the convergence of normal extremes is known to be slow \citep{Ha91}, in standard nonparametric density and regression function estimation, it is often recommended to use versions of bootstraps to construct confidence bands, instead of relying on Gumbel approximations \citep[cf.][]{NePo98,Clva03,ChChKa14a}.\footnote{For the trigonometric basis, \cite{KoPa16} develop an analytical method based on higher oder expansions to improve on the Gumbel approximation; see their Theorem 3.7.} This paper contributes to the literature on nonparametric inference for L\'{e}vy processes by developing bootstrap confidence bands for the first time in the L\'{e}vy density estimation case. Furthermore, spectral-type estimators are among the most commonly used methods for estimation of the L\'{e}vy density \cite[cf.][]{CoGe15, BeRe15}, and developing inference methods for them is practically important.

From a technical point of view, the proofs of the main theorems build on non-trivial applications of the intermediate Gaussian and bootstrap approximation theorems developed in \cite{ChChKa14a,ChChKa14b,ChChKa16}. The analysis of the present paper has some connections to those of \cite{KaSa16,KaSa17} that study confidence bands for deconvolution and nonparametric errors-in-variables regression, respectively. 
However, the high-frequency setup in L\'{e}vy process estimation has different probabilistic structures than the i.i.d. setup in standard nonparametric estimation. 
For example, the increment distribution $P_{\Delta}$ (i.e., the distribution of $L_{\Delta}$) need not be continuous and may have a discrete component (which is in contrast to the standard density estimation case); $P_{\Delta}$ is indexed by $\Delta$ with $\Delta = \Delta_{n} \to 0$ as $n \to \infty$, and degenerates to the point mass at the origin; and the interplay between $\Delta$ and $n$ has to be taken care of. 
In particular, providing low-level regularity conditions for validity of bootstrap confidence bands in the L\'{e}vy density estimation case is far from trivial and requires substantial work. See the discussion after Assumption \ref{as: assumption 1} and Section \ref{sec: condition (ii)}.

In this paper, we assume that data do not contain microstructure noises. We have in mind that the time span $\Delta$ is small but not too small -- say 5 minutes if the unit time is one day. For such cases, \cite{AiXu17} argue that the effect of microstructure noise is small. 

The rest of the paper is organized as follows. In Section \ref{sec: spectral estimation}, we define a spectral estimator for the L\'{e}vy density, and in Section \ref{sec: confidence band} we describe our bootstrap methods to construct confidence bands for the spectral estimator. We consider two bootstrap methods, namely, Gaussian multiplier and empirical bootstraps. In Section \ref{sec: main results}, we present theorems that establish asymptotic validity of the proposed confidence bands. 
In Section \ref{sec: examples}, we provide concrete examples of L\'{e}vy processes that satisfy our regularity conditions. In Section \ref{sec: simulation results},  we propose a practical method for bandwidth selection, and study its finite sample performance via numerical simulations. 
Section \ref{sec: conclusion} concludes. All the proofs are deferred to Appendix.

\subsection{Notation}
We will obey the following notation. For any non-empty set $T$ and any (complex-valued) function $f$ on $T$, let $\| f \|_{T} = \sup_{t \in T}|f(t)|$. Let $\ell^{\infty}(T)$ denote the (real) Banach space of all bounded real-valued functions on $T$ equipped with the sup-norm $\| \cdot \|_{T}$. The Fourier transform of an integrable function $f$ on $\R$ is defined as 
\[
\varphi_{f}(u) = \int_{\R} e^{iux} f(x) dx, \ u \in \R.
\]
For any $x \in \R$, let $\delta_{x}$ denote the Dirac measure at $x$. 
For any $a,b \in \R$, let $a \vee b = \max \{ a,b \}$ and $a \wedge b = \min \{ a,b \}$. For $a \in \R$ and $b > 0$, we use the shorthand notation $[a \pm b] = [a-b,a+b]$. 
For any non-empty set $A$ in $\R$ and any $\varepsilon > 0$, let $A^{\varepsilon} = \{ x \in \R : d(x,A) \leq \varepsilon \}$ where $d(x,A) = \inf_{y \in A} |x-y|$. For any positive sequences $a_{n},b_{n}$, we write $a_{n} \lesssim b_{n}$ if there is a positive constant $C > 0$ independent of $n$ such that $a_{n} \leq Cb_{n}$ for all $n$, $a_{n} \sim b_{n}$ if $a_{n} \lesssim b_{n}$ and $b_{n} \lesssim a_{n}$, and $a_{n} \ll b_{n}$ if $a_{n}/b_{n} \to 0$ as $n \to \infty$. 


\section{Spectral estimation of L\'{e}vy density}
\label{sec: spectral estimation}

We first describe a spectral estimation method for  L\'{e}vy densities. 
Let $Y_{n,j} = L_{j\Delta} - L_{(j-1)\Delta}, j=1,\dots,n$ be increments of discrete observations of the L\'{e}vy process $L$, and observe that $Y_{n,j}, j=1,\dots,n$ are i.i.d. whose common characteristic function is  $\varphi_{\Delta}(u)= \Ep[e^{iuL_{\Delta}}] = e^{\Delta \psi(u)}$.
In this paper, we assume that 
\begin{equation}
\int_{\R} x^{2} \rho(x) dx < \infty,
\label{eq: finite variance}
\end{equation}
which is equivalent to assuming that $\E[ L_{1}^{2} ] < \infty$ (and $\Ep[ L_{t}^{2} ] < \infty$ for all $t > 0$; see \cite{Sa99}, Corollary 25.8). Condition (\ref{eq: finite variance}) ensures that the integral $\int_{\R} (e^{iux} - 1-iux) \rho (x)dx$
is well-defined, so that the characteristic exponent $\psi(u)$ admits an alternative representation: 
\[
\psi(u)  = - \frac{u^{2}\sigma^{2}}{2} + iu\gamma_{c}+ \int_{\R} (e^{iux} - 1- iux ) \rho (x) dx,
\]
where $\gamma_{c} = i^{-1}(e^{\psi (u)})'|_{u=0} = \Ep[ L_{1} ]$. 
Furthermore, under Condition (\ref{eq: finite variance}), differentiating $\psi (u)$ twice, we arrive at the key identity
\[
\psi''(u) = {\varphi''_{\Delta}(u)\varphi_{\Delta}(u) - (\varphi'_{\Delta}(u))^{2} \over \Delta \varphi_{\Delta}^{2}(u)}  = -\sigma^{2}- \int_{\R}e^{iux}x^{2}\rho(x)dx.
\]
Therefore, applying the Fourier inversion, we conclude that
\begin{align}
x^{2}\rho(x) &= {1 \over 2\pi}\int_{\R}e^{-iux}\left (-\psi''(u) -\sigma^{2}\right )du \notag \\
&=  {1 \over 2\pi }\int_{\R}e^{-iux}\left ( {(\varphi'_{\Delta}(u))^{2} - \varphi''_{\Delta}(u)\varphi_{\Delta}(u) \over \Delta \varphi^{2}_{\Delta}(u)} - \sigma^{2}\right ) du,
\label{eq: Fourier}
\end{align}
where the Fourier inversion should be interpreted in the distributional sense if the integral is not well-defined. 
This expression leads to a method to estimate $\rho$.

First, we estimate $\varphi_{\Delta}^{(k)}(u), k=0,1,2$ by $\hat{\varphi}_{\Delta}^{(k)}(u)$, where 
\[
\hat{\varphi}_{\Delta}(u) = \frac{1}{n} \sum_{j=1}^{n} e^{iuY_{n,j}}, \ u \in \R
\]
is the empirical characteristic function ($\varphi_{\Delta}^{(k)}$ denotes the $k$-th derivative of $\varphi_{\Delta}$ with $\varphi_{\Delta}^{(0)} = \varphi_{\Delta}$). Let $W: \R \to \R$ be an integrable function (kernel) such that $\int_{\R} W(x) dx= 1$ and its Fourier transform $\varphi_{W}$ is supported in $[-1,1]$ (i.e., $\varphi_{W}(u) = 0$ for all $|u| > 1$). Then the spectral estimator of $\rho$ is defined by
\begin{equation}
\hat{\rho}(x) = {1 \over 2\pi  x^{2}}\int_{\mathbb{R}}e^{-iux}\left ({(\hat{\varphi}'_{\Delta}(u))^{2} - \hat{\varphi}''_{\Delta}(u)\hat{\varphi}_{\Delta}(u) \over \Delta \hat{\varphi}^{2}_{\Delta}(u)} - \hat{\sigma}^{2}\right ) \varphi_{W}(uh)du \label{eq: spectral estimator}
\end{equation}
for $x \neq 0$, where $h = h_{n}$ is a sequence of positive numbers (bandwidths) such  that $h_{n} \to 0$ as $n \to \infty$, and $\hat{\sigma}^{2}$ is a pilot estimator of $\sigma^{2}$.

Some comments on the spectral estimator $\hat{\rho}$ are in order. 
 First, as long as $h \gtrsim \Delta^{1/2}$, $\inf_{|u| \leq h^{-1}} | \hat{\varphi}_{\Delta}(u)| \gtrsim 1-o_{\Pr}(1)$, so that $\hat{\rho}(x)$ for $x \neq 0$ is well-defined with probability approaching one (see Lemmas \ref{lem: lower bound on chf} and \ref{lem: uniform convergence}).
Second, the function $\hat{\rho}$ is real-valued.
Third, noting that $\hat{\psi}:= \Delta^{-1} \log \hat{\varphi}_{\Delta}$ is well-defined on $[-h^{-1},h^{-1}]$ (with probability approaching one) as the \textit{distinguished logarithm}  \citep[][Theorem 7.6.2]{Ch01}, we see that the spectral estimator $\hat{\rho}$ can  be alternatively expressed as 
\[
\hat{\rho}(x) = \frac{1}{2\pi x^{2}} \int_{\R} e^{-iux} \left ( -\hat{\psi}''(u) - \hat{\sigma}^{2} \right ) \varphi_{W}(uh) du
\]
for $x \neq 0$. Finally, in this  paper, we are interested in estimating $\rho$ on a compact set $I$ away from the origin (e.g. $[a,b] \cup [c,d]$ for $a<b<0<c<d$), and therefore, as long as $|W(x)|$ decays sufficiently fast as $|x| \to \infty$, we may take $\hat{\sigma}^{2}=0$ in theory. See the discussion after Assumption \ref{as: assumption 1} below. However, in our simulation studies, we found that, in case of $\sigma > 0$, using a proper estimator for $\hat{\sigma}^{2}$ improves on the empirical performance of the estimator $\hat{\rho}$ and the inference methods, especially if the set $I$ is close to the origin. Therefore, we recommend to plug-in a proper estimator of $\sigma^{2}$. There are several existing estimators for $\sigma^{2}$; see Example \ref{ex: volatility estimator} below.

Our spectral estimator (\ref{eq: spectral estimator}) is considered and studied in \cite{Be11} and \cite{Gu12} under the low-frequency setup. \cite{NiReSoTr16} use the spectral estimator (\ref{eq: spectral estimator}) to construct estimators for distribution functions such as $\int_{-\infty}^{\cdot} (1 \wedge x^{2}) \nu (dx)$, and prove Donsker-type functional limit theorems for distribution function estimators under the high-frequency setup. There are versions of spectral-type estimators of $\rho$ similar to but different from ours (\ref{eq: spectral estimator}). For example, in case of $\sigma=0$, \cite{CoGe11} consider simplified versions of the estimator (\ref{eq: spectral estimator}) by replacing $\hat{\varphi}_{\Delta}(u)$ with 1 and/or $\hat{\varphi}_{\Delta}'(u)$ with $0$. Such simplifications produce additional biases that depend on $\Delta$ (but not on smoothness of $x^{2}\rho$); since the problem of bias is already  delicate  in construction of confidence bands in standard nonparametric estimation \citep[cf.][Section 5.7]{Wa06}, producing additional biases is not favorable to our goal from both theoretical and practical view points. Hence, in this paper, we  focus on the current spectral estimator (\ref{eq: spectral estimator}).
It is worth pointing out that the identification (\ref{eq: Fourier}) of the L\'{e}vy density $\rho$ holds  without relying on the assumption that $\Delta \to 0$, and therefore the deterministic bias of our spectral estimator $\hat{\rho}$ does not depend on $\Delta$; see the discussion after Assumption \ref{as: assumption 1} below. 

Furthermore, under relatively mild conditions, our spectral estimator (\ref{eq: spectral estimator}) is consistent under the weighted  sup-norm on $\R$, $\| f \|_{w,\infty} = \sup_{x \in \R} |x^{2}f(x)|$ (see Appendix \ref{sec: uniform convergence}), and thereby is able to capture the shape of $\rho$ \textit{globally} (i.e., uniformly over $\R \setminus \{ 0 \}$), which we believe is an attractive feature of the spectral estimator $\hat{\rho}$. 
  

%

\begin{example}[Examples of estimators for $\sigma^{2}$]
\label{ex: volatility estimator}

There are several consistent estimators for $\sigma^{2}$ available in the literature on high-frequency data analysis for continuous-time stochastic processes. We provide a couple of examples here. The first example is the truncated realized volatility (TRV)  estimator proposed in \cite{M(2001)} :
\begin{equation}
\hat{\sigma}^2_{TRV} = {1 \over n\Delta}\sum_{j=1}^{n}Y_{n,j}^{2}1_{\{|Y_{n,j}| \leq \alpha_{0}\Delta^{\theta_{0}} \}}, \label{eq: TRV}
\end{equation}
where $\alpha_{0}>0$ and $\theta_{0} \in (0,1/2)$. The second example is the power variation (PV) estimator proposed in \cite{BS(2004)}:
\[
\hat{\sigma}^2_{PV} (\alpha)= \left({1 \over n\Delta^{\alpha/2} m_{\alpha}}\sum_{j=1}^{n}|Y_{n,j}|^{\alpha}\right)^{2/\alpha},
\]
where $\alpha \in (0,2)$ and $m_{\alpha} = 2^{\alpha/2}\Gamma ((\alpha+1)/2)/\sqrt{\pi}$ is the $\alpha$-th absolute moment of $N(0,1)$. \cite{JaRe14} study the optimal rate of convergence for estimating $\sigma^2$ in the minimax sense and propose the following estimator (modified to our setup): 
\[
\hat{\sigma}^2_{JR} = -{2 \over \Delta u_{n}^{2}}(\log|\hat{\varphi}_{\Delta}(u_{n})|)1_{\{\hat{\varphi}_{\Delta}(u_{n}) \neq 0\}},
\]
where $u_{n} \propto \sqrt{(\log n)/\Delta}$ is a deterministic sequence. 
In our simulation studies, we use $\hat{\sigma}^{2}_{TRV}$ as an estimator of $\sigma^2$. 

Strictly speaking, the references cited above study the asymptotic properties of the estimators under a different high-frequency setup that $\Delta \to 0$ as $n \to \infty$ but $n \Delta$ is fixed.
For the asymptotic properties of the PV and JR estimators under our setup, we refer to \citet[][Proposition 5.3]{CoGe11} \citep[see also][]{AiJa07} and \citet[][Proposition 13]{NiReSoTr16}, respectively. 
For the sake of completeness, we summarize the asymptotic properties of the TRV estimator in the following lemma.
See Appendix \ref{sec: appendix c} for the proof.

\begin{lemma}
\label{lem: TRV}
Suppose that the L\'{e}vy measure satisfies $\int_{[-1,1]} |x|^{\alpha} \nu (dx) < \infty$ for some $\alpha \in (0,2)$, and if $\alpha \in [1,2)$, then assume in addition that  $\int_{[-1,1]^{c}}|x|\nu(dx)<\infty$. Then 
$| \hat{\sigma}^{2}_{TRV} - \sigma^{2}| = O_{\Pr}(n^{-1/2} + \Delta^{(2-\alpha)\theta_{0}})$. 
\end{lemma}
\end{example}

\subsection{Comparison with direct kernel estimation: preliminary simulations}

Alternatively to spectral-type estimation methods, exploiting the assumption that $\Delta \to 0$ as $n \to \infty$, we can estimate the L\'{e}vy density $\rho(x)$ for $x \neq 0$ by applying directly kernel density estimation to increments $Y_{n,1},\dots,Y_{n,n}$, i.e., 
\begin{equation}
\hat{\rho}^{direct}(x) = \frac{1}{n\Delta h} \sum_{j=1}^{n} W^{direct}((x-Y_{n,j})/h), \label{eq: direct kernel estimator}
\end{equation}
where $W^{direct}: \R \to \R$ is a compactly supported smooth kernel function. In fact, for a given fixed $x \neq 0$, if $\rho$ is Lipschitz continuous in a neighborhood of $x$, then from Lemma B.1 and Proposition 2.1 in \cite{Fi11}, we can show that
\[
\Ep[ \hat{\rho}^{direct}(x)] = \underbrace{\frac{1}{h} \int_{\R} W^{direct}((x-y)/h) \rho(y) dy}_{=\rho (x) + O(h)} + O(\Delta/h),
\]
so that $\hat{\rho}^{direct}$ is consistent for $\rho$ at $x$ under appropriate regularity conditions. Actually, the direct kernel estimator (\ref{eq: direct kernel estimator}) is mentioned in \citet[][Section 4.1]{Fi11a}, although the detailed properties of (\ref{eq: direct kernel estimator}) are not studied there. 
From a theoretical point of view, it is rather easier to develop inference methods for $\hat{\rho}^{direct}$ than the spectral estimator (\ref{eq: spectral estimator}) under the high-frequency setup, since the former is of simpler form than the latter.
So, one might be tempted to wonder why we bother to use a more complicated estimator $\hat{\rho}$.  

It turns out that, however, in the finite sample, the direct kernel estimate (\ref{eq: direct kernel estimator}) tends to have (much) larger biases, especially near the origin, than the spectral estimate (\ref{eq: spectral estimator}). 
Figures \ref{fig:comparison1} and \ref{fig:comparison2} compare realizations of direct kernel estimates with Epanechnikov kernel and spectral estimates with a flap-top kernel for a jump-diffusion process $L_{t}=B_{t}+J_{t}$ where $B=(B_{t})_{t \geq 0}$ is a standard Brownian motion and $J=(J_{t})_{t \geq 0}$ is a compound Poisson process (independent of $B$) with intensity $\lambda =10$ and jump size distribution $N(0,v^{2})$ ($v=0.5$), and for a Gamma process with parameter $(c_{+},\lambda_{+})=(0.2,1)$ (i.e, $L_{\Delta}$ has Gamma distribution with shape parameter $c_{+}\Delta$ and scale parameter $1$). Here we set $(n,\Delta)=(50000, 0.01)$. Preliminary calibrations show that $h=\sqrt{\Delta}=0.1$ works well for for both estimates, and we set this bandwidth value to generate these figures. The flap-top kernel used for the spectral estimate is defined by the inverse Fourier transform of equation (\ref{eq: flap-top}) ahead with $b=1$ and $c=0.05$. Furthermore, for the spectral estimate, we plug-in  the TRV estimate $\hat{\sigma}^{2}_{TRV}$ with $\alpha_{0}=3$ and $\theta_{0} =  0.48$ for the jump-diffusion case, and 
set $\hat{\sigma}^{2} = 0$ for the Gamma process case (using the TRV estimate for the Gamma process case produced almost same simulation results).

\begin{figure}
  \begin{center}
    \begin{tabular}{cc}

      \begin{minipage}{0.5\hsize}
        \begin{center}
          \includegraphics[clip, width=6cm]{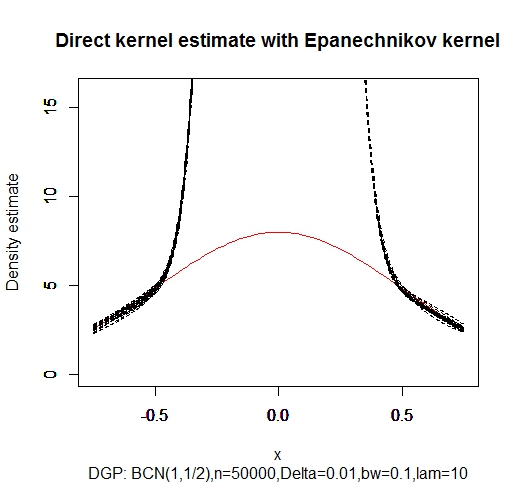}
        \end{center}
      \end{minipage}

      \begin{minipage}{0.5\hsize}
        \begin{center}
          \includegraphics[clip, width=6cm]{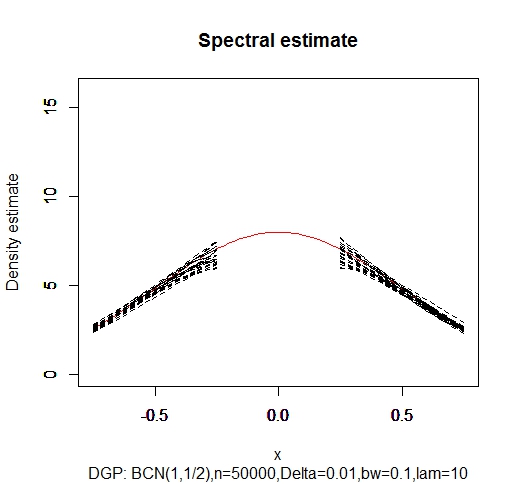}
        \end{center}
      \end{minipage}

    \end{tabular}
    \caption{{\small Plots of 25 realizations of direct kernel estimates with Epanechnikov kernel (left) and spectral estimates with a flap-top kernel (right) on $[-0.75,-0.25] \cup [0.25,0.75]$ for $L_{t} =  B_{t} + J_{t}$ where $B_{t}$ is a standard Brownian motion and $J_{t}$ is a compound Poisson process with intensity $\lambda = 10$ and jump size distribution $N(0,v^{2})$ ($v=0.5$). The bandwidth value is $h=\sqrt{\Delta} = 0.1$. The red lines correspond to the true L\'{e}vy density $\rho (x) = \lambda  e^{-x^{2}/(2v^{2})}/\sqrt{2\pi v^{2}}$.  \label{fig:comparison1}}}
  \end{center}
\end{figure}

\begin{figure}
  \begin{center}
    \begin{tabular}{cc}

      \begin{minipage}{0.5\hsize}
        \begin{center}
          \includegraphics[clip, width=6cm]{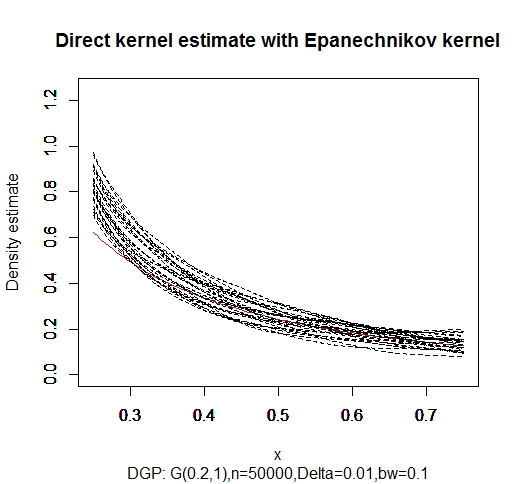}
        \end{center}
      \end{minipage}

      \begin{minipage}{0.5\hsize}
        \begin{center}
          \includegraphics[clip, width=6cm]{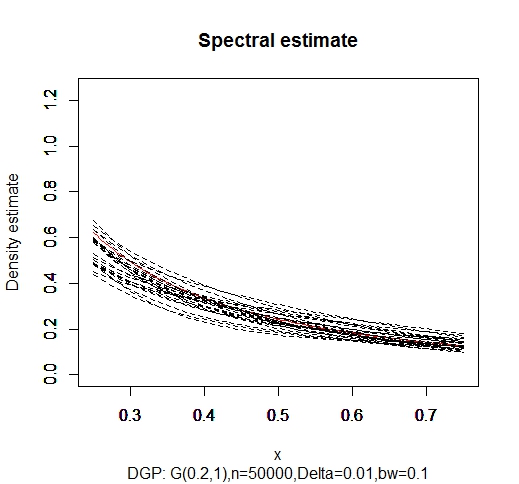}
        \end{center}
      \end{minipage}

    \end{tabular}
    \caption{{\small Plots of 25 realizations of direct kernel estimates with Epanechnikov kernel (left) and spectral estimates with a flap-top kernel (right) on $[0.25,0.75]$ where $L_{t}$ is a Gamma process with parameter $(c_{+},\lambda_{+}) = (0.2,1)$, i.e., $L_{\Delta}$ has Gamma distribution with shape parameter $c_{+}\Delta$ and scale parameter $\lambda_{+}$.  The bandwidth value is $h=\sqrt{\Delta} = 0.1$. The red lines correspond to the true L\'{e}vy density $\rho (x) = c_{+} x^{-1} e^{-\lambda_{+} x}1_{(0,\infty)}(x)$.  \label{fig:comparison2}}}
  \end{center}
\end{figure}

These figures show that the direct kernel estimate  has large biases especially near the origin, whereas the spectral estimate performs reasonably well on the entire set for each case ($[-0.75,-0.25] \cup [0.25,0.75]$ for the jump-diffusion case, and $[0.25,0.75]$ for the Gamma process case). In particular, the direct kernel estimate exhibits unreasonable behaviors for $|x| < 0.5$ in the jump-diffusion case. Intuitively, such unreasonable behaviors can be understood from the observation that the increment distribution $P_{\Delta}$ (i.e., the distribution of $L_{\Delta}$) of the jump-diffusion process is the convolution of $N(0,\Delta v^{2})$ with the compound Poisson distribution that has a point mass at the origin, and therefore, the density of $P_{\Delta}$ has a sharp peak around the origin. 
Since the direct kernel estimate is estimating the density of $P_{\Delta}$ scaled by $\Delta^{-1}$, it tends to return unreasonably large values near the origin (we also note that, since $\Pr (|N(0,v^{2})| \geq 0.5) \approx 0.3$ for $v=0.5$, the interval $(-0.5,0.5)$ is not a ``small'' neighborhood of the origin). 
These preliminary simulation results motivate us to study inference methods for the spectral estimator (\ref{eq: spectral estimator}) rather than the direct kernel estimator (\ref{eq: direct kernel estimator}). 

\section{Construction of confidence bands}
\label{sec: confidence band}
We consider to construct confidence bands for $\rho$ on a compact set  $I$ in $\R \setminus \{ 0 \}$. For example, $I=[a,b]$ for $0 < a < b \ \text{or} \ a < b < 0$; or $I=[a,b] \cup [c,d]$ for $a < b < 0 < c < d$.
The set $I$ may be a singleton, i.e., $I=\{ x_{0} \}$ for $x_{0} \neq 0$, although we are primarily interested in the above two cases.

Under the regularity conditions stated below, we will show that $x^{2}(\hat{\rho}(x) - \rho(x))$ can be approximated by 
\begin{equation}
\frac{-1}{2\pi \Delta}\int_{\R}e^{-iux} \frac{(\hat{\varphi}_{\Delta} - \varphi_{\Delta})''}{\varphi_{\Delta}}(u) \varphi_{W}(uh) du \label{eq: main term}
\end{equation}
uniformly in $x \in I$. By a change of variables, we may rewrite the term (\ref{eq: main term}) as
\begin{equation}
\frac{1}{n\Delta h} \sum_{j=1}^{n}\{ Y_{n,j}^{2}K_{n}((x-Y_{n,j})/h) - \Ep[Y_{n,1}^{2}K_{n}((x-Y_{n,1})/h) ] \}, \label{eq: intermediate process}
\end{equation}
where $K_{n}$ is the function defined by
\[
K_{n}(x) = \frac{1}{2\pi} \int_{\R} e^{-iux} \frac{\varphi_{W}(u)}{\varphi_{\Delta}(u/h)} du, \ x \in \R. 
\]
Note that $K_{n}$ is well-defined and real-valued. Define
\[
s_{n}^{2} (x) = \Var (Y_{n,1}^{2}K_{n}((x-Y_{n,1})/h)), \ s_{n}(x) = \sqrt{s_{n}^{2}(x)}, \ x \in I, 
\]
and consider the process
\[
Z_{n}(x) = \frac{1}{s_{n}(x)\sqrt{n}} \sum_{j=1}^{n} \{ Y_{n,j}^{2}K_{n}((x-Y_{n,j})/h) - \Ep[Y_{n,1}^{2}K_{n}((x-Y_{n,1})/h) ] \}, \ x \in I.
\]
Under the regularity conditions stated below, $\inf_{x \in I} s_{n}^{2}(x) \gtrsim \Delta h$ for sufficiently large $n$. 
Furthermore, we will show that there exists a tight Gaussian random variable $Z_{n}^{G}$ in $\ell^{\infty}(I)$ with mean zero and the same covariance function as $Z_{n}$, and such that the distribution of $\| Z_{n}^{G} \|_{I} = \sup_{x \in I} |Z_{n}^{G}(x)|$ asymptotically approximates that of $\| Z_{n} \|_{I}$ in the sense that
\[
\sup_{z \in \R}\left |  \Pr \{ \| Z_{n} \|_{I} \leq z \} - \Pr \{ \| Z_{n}^{G} \|_{I} \leq z \} \right | \to 0
\]
as $n \to \infty$, which in turn yields that 
\begin{equation}
\sup_{z \in \R}  \left | \Pr \{ \| \sqrt{n}\Delta h x^{2}(\hat{\rho} - \rho)/s_{n} \|_{I} \leq z \} - \Pr \{ \| Z_{n}^{G} \|_{I} \leq z \} \right | \to 0.
\label{eq: Gaussian approximation 0}
\end{equation}
Hence, construction of confidence bands for $\rho$ on $I$ reduces to approximating or estimating quantiles of $\| Z_{n}^{G} \|_{I}$. In fact, let 
\begin{align*}
c_{n}^{G}(1-\tau) &= \text{$(1-\tau)$-quantile of $\|Z_{n}^{G}\|_{I}$} \\
&:= \inf \{ z \in \R : \Pr \{ \| Z_{n}^{G} \|_{I} \leq z \} \geq 1-\tau \}
\end{align*}
for $\tau \in (0,1)$, and consider a confidence band of the form
\[
\hat{\mathcal{C}}_{1-\tau} (x) = \left [ \hat{\rho}(x) \pm \frac{s_{n}(x)}{x^{2}\sqrt{n} h \Delta} c_{n}^{G}(1-\tau) \right ], \ x \in I.
\]
Since $\rho (x) \in \hat{\mathcal{C}}_{1-\tau} (x) \ \forall x \in I \Leftrightarrow \| \sqrt{n}\Delta h x^{2} (\hat{\rho} - \rho)/s_{n} \|_{I} \leq c_{n}^{G}(1-\tau)$, 
it is seen that
\[
\Pr \{ \rho (x) \in \hat{\mathcal{C}}_{1-\tau}(x) \ \forall x \in I \} = \Pr \{ \| Z_{n}^{G} \|_{I} \leq c_{n}^{G}(1-\tau) \} + o(1) \geq 1-\tau +o(1),
\]
so that $\hat{\mathcal{C}}_{1-\tau}$ will be a valid confidence band for $\rho$ on $I$ with level approximately $1-\tau$, provided that (\ref{eq: Gaussian approximation 0}) holds. 

Now, we shall  estimate the quantile $c_{n}^{G}(1-\tau)$, in addition to the variance function $s_{n}^{2}(x)$. 
The latter can be estimated from 
\[
\hat{s}_{n}^{2}(x) = \frac{1}{n}\sum_{j=1}^{n}Y_{n,j}^{4}\hat{K}_{n}^{2}((x-Y_{n,j})/h) - \left \{ \frac{1}{n}\sum_{j=1}^{n}Y_{n,j}^{2}\hat{K}_{n}((x-Y_{n,j})/h) \right \}^{2}, \ x \in I,
\]
where $\hat{K}_{n}$ is the function defined by
\[
\hat{K}_{n}(x) = \frac{1}{2\pi} \int_{\R} e^{-iux} \frac{\varphi_{W}(u)}{\hat{\varphi}_{\Delta}(u/h)} du, \ x \in \R.
\]
Note that as long as $h \gtrsim \Delta^{1/2}$, $\inf_{|u| \leq h^{-1}} | \hat{\varphi}_{\Delta}(u) | \gtrsim 1-o_{\Pr}(1)$, so that $\hat{K}_{n}$ is well-defined with probability approaching one. 
To estimate the quantile $c_{n}^{G}(1-\tau)$, we consider two bootstrap methods. The one is the Gaussian multiplier (or wild) bootstrap, and the other is the empirical (or Efron's) bootstrap. 

\subsection*{Gaussian multiplier bootstrap} Generate $\xi_{1},\dots, \xi_{n} \sim N(0,1)$ i.i.d., 
independent of the data $\mathcal{D}_{n} = \{ Y_{n,j} \}_{j=1}^{n}$, and construct the multiplier process 
\[
\hat{Z}_{n}^{MB}(x) = \frac{1}{\hat{s}_{n}(x)\sqrt{n}} \sum_{j=1}^{n} \xi_{j} \left \{ Y_{n,j}^{2} \hat{K}_{n}((x-Y_{n,j})/h) - n^{-1} {\textstyle \sum}_{j'=1}^{n} Y_{n,j'}^{2} \hat{K}_{n}((x-Y_{n,j'}/h) \right \}, \ x \in I,
\]
where $\hat{s}_{n}(x) = \sqrt{\hat{s}_{n}^{2}(x)}$. Note that under the regularity conditions stated below, $\inf_{x \in I} \hat{s}_{n}^{2}(x) \gtrsim (1-o_{\Pr}(1)) \Delta h$, so that $\hat{Z}_{n}^{MB}$ is well-defined with probability approaching one. Conditionally on the data $\mathcal{D}_{n}$, $\hat{Z}_{n}^{MB}$ is a Gaussian process whose covariance function ``estimates'' that of $Z_{n}^{G}$. Hence, we estimate $c_{n}^{G}(1-\tau)$ by 
\[
\hat{c}_{n}^{MB}(1-\tau) = \text{conditional $(1-\tau)$-quantile of $\| \hat{Z}_{n}^{MB} \|_{I}$ given $\mathcal{D}_{n}$},
\]
which can be computed via simulations. The resulting confidence band is 
\[
\hat{\mathcal{C}}_{1-\tau}^{MB}(x) = \left [ \hat{\rho}(x) \pm \frac{\hat{s}_{n}(x)}{x^{2}\sqrt{n} h \Delta} \hat{c}_{n}^{MB}(1-\tau) \right ], \ x \in I.
\]

\subsection*{Empirical bootstrap} Next, we consider the empirical bootstrap. Let $P_{n,\Delta} = n^{-1} \sum_{j=1}^{n} \delta_{Y_{n,j}}$ denote the empirical distribution. Conditionally on the data, generate $Y_{n,1}^{b}, \dots, Y_{n,n}^{b} \sim P_{n,\Delta}$ i.i.d., 
and construct the bootstrap process 
\[
\hat{Z}_{n}^{EB}(x) = \frac{1}{\hat{s}_{n}(x)\sqrt{n}} \sum_{j=1}^{n} \left \{ (Y_{n,j}^{b})^{2} \hat{K}_{n}((x-Y_{n,j}^{b})/h) - n^{-1} {\textstyle \sum}_{j'=1}^{n} Y_{n,j'}^{2} \hat{K}_{n}((x-Y_{n,j'})/h) \right \}, \ x \in I.
\]
We estimate $c_{n}^{G}(1-\tau)$ by 
\[
\hat{c}_{n}^{EB}(1-\tau) = \text{conditional $(1-\tau)$-quantile of $\| \hat{Z}_{n}^{EB} \|_{I}$ given $\mathcal{D}_{n}$}.
\]
The resulting confidence band is 
\[
\hat{\mathcal{C}}_{1-\tau}^{EB}(x) = \left [ \hat{\rho}(x) \pm \frac{\hat{s}_{n}(x)}{x^{2}\sqrt{n} h \Delta} \hat{c}_{n}^{EB}(1-\tau) \right ], \ x \in I.
\]

\begin{remark}[Scaling by $1/x^{2}$]
One might think that, because of the scaling by $1/x^{2}$, our confidence bands would be too wide if $x$ is close to the origin. However, heuristically, the standard deviation function $s_{n}(x)$ would scale  like  $x^{2} \sqrt{\Delta h \rho(x)}$ for $x \neq 0$,  so the scaling factor $1/x^{2}$ would be canceled out and $s_{n}(x)/(x^{2} \sqrt{n} \Delta h)$ would scale like $\sqrt{\rho(x)/(n\Delta h)}$. 
To see this, assuming that $\rho$ has finite fourth moment (which ensures that $P_{\Delta}$ has finite fourth moment), observe that since $y^{4}P_{\Delta}(dy)/\Delta \to y^{4}\rho (y) dy$ weakly as finite measures\footnote{To  see this, observe that, for each $u \in \R$, $\varphi^{(4)}_{\Delta}(u) = \Delta \psi^{(4)}(u) \varphi_{\Delta}(u) + o(\Delta) = \Delta \int_{\R} e^{iuy} y^{4} \rho(y) dy + o(\Delta)$, so that 
\[
\int_{\R} e^{iuy} \frac{y^{4}P_{\Delta}(dy)}{\Delta} = \frac{\varphi_{\Delta}^{(4)}(u)}{\Delta} \to \int_{\R} e^{iuy} y^{4} \rho(y) dy, 
\]
which implies that $y^{4}P_{\Delta}(dy)/\Delta \to y^{4} \rho(y) dy$ weakly.}, we have that, heuristically,
\begin{align*}
s_{n}^{2}(x) &\approx \Ep[ Y_{n,1}^{4} K^{2}_{n}((x-Y_{n,1})/h)] \\
&\approx \Delta \int_{\R} y^{4} K_{n}^{2}((x-y)/h) \rho(y)dy \\
&\approx \Delta h x^{4} \rho(x) \int_{\R} K_{n}^{2}(y)dy.
\end{align*}
Of course, these approximations are only heuristic, but the discussion so far at least provides a partial explanation for that the scaling factor $1/x^{2}$ would not too much inflate the width of our bands. See also figures in Section \ref{sec: simulation results}.

\end{remark}

\section{Main results}
\label{sec: main results}

\subsection{Validity of bootstrap confidence bands}

In this section, we prove validity of the proposed confidence bands $\hat{\mathcal{C}}_{1-\tau}^{MB}$ and $\hat{\mathcal{C}}_{1-\tau}^{EB}$. 
To this end, we make the following assumption. 
Recall that a function $f: \R \to \R$ is said to be $\alpha$-H\"{o}lder continuous for $\alpha \in (0,1]$ if
\[
\sup_{x,y \in \R, x \neq y} \frac{|f(x) - f(y)|}{|x-y|^{\alpha}} < \infty.
\]
Let $P_{\Delta}$ denote the distribution of $L_{\Delta} =Y_{n,1}$, and for any $b \in \R$, let $P_{\Delta,b}$ denote the distribution of $L_{\Delta} - b\Delta$. Let $I$ be a compact set in $\R \setminus \{ 0 \}$. 
\begin{assumption}
\label{as: assumption 1}
We assume the following conditions. 
\begin{enumerate}
\item[(i)] $\int_{\R}x^{4} \rho(x) dx < \infty$. 
\item[(ii)] There exists $b \in \R$ such that the measure $y^{4}P_{\Delta,b}(dy)$ has Lebesgue density $g_{\Delta,b}$ such that  $\| g_{\Delta,b} \|_{\R} \lesssim \Delta$. Furthermore, 
$\inf_{y \in I^{\varepsilon_{0}}} g_{\Delta,b} (y)\gtrsim \Delta$ for some sufficiently small $\varepsilon_{0} > 0$ such that $0 \not \in I^{\varepsilon_{0}}$. 
\item[(iii)] Let $r > 0$, and let $p$ be the integer such that $p < r \leq p+1$. The function $x^{2} \rho$ is $p$-times differentiable, and $(x^{2}\rho)^{(p)}$ is $(r-p)$-H\"{o}lder continuous.
\item[(iv)] Let $W: \R \to \R$ be an integrable function (kernel) such that 
\begin{equation}
\label{eq: kernel}
\begin{cases}
&\int_{\R} W(x) dx = 1, \ \int_{\R} |x|^{p+1} |W(x)| dx < \infty, \\
&\int_{\R} x^{\ell} W(x) dx=0, \ell=1,\dots,p, \\
&\text{$\varphi_{W}$ is four-times continuously differentiable and}  \ \varphi_{W}(u) = 0 \ \forall |u| > 1,
\end{cases}
\end{equation}
where $\varphi_{W}$ is the Fourier transform of $W$. 
\item[(v)] $h \gtrsim \Delta^{1/2}, \ n^{1/2-\delta} \sqrt{\Delta h} \to \infty$ for some $\delta \in (0,\frac{r}{2r+3})$, and $h^{r}\sqrt{n \Delta h\log n} \to 0$. 
\item[(vi)] Let $\hat{\sigma}^{2}$ be an estimator of $\sigma^{2}$ such that $|\hat{\sigma}^{2} - \sigma^{2} | \cdot \| h^{-1} W(\cdot/h) \|_{I} = o_{\Pr}\{ (n\Delta h \log n)^{-1/2} \}$.
\end{enumerate}
\end{assumption}

Condition (i) is a moment condition and is equivalent to finiteness of the fourth moment of $L_{1}$ (and $L_{t}$ for all $t > 0$; see \cite{Sa99}, Corollary 25.8). Condition (i) excludes, e.g., $\alpha$-stable processes for $\alpha \in (0,2)$, but it allows for $\rho$ not to be integrable (i.e., $\nu (\R) = \infty$ is allowed). 
Condition (ii) is a high-level condition and will be discussed in detail in the next subsection. 
However, at this point, we would like to remark that Condition (ii) is satisfied by a wide class of L\'{e}vy processes. Importantly, Condition (ii) allows the distribution $P_{\Delta}$ to have a discrete component. For example, if $L_{t} = bt + J_{t}$ where $J = (J_{t})_{t \geq 0}$ is a compound Poisson process (with absolutely continuous jump size distribution), then $P_{\Delta}$ has a point mass at $b\Delta$ and $P_{\Delta,b} = P_{\Delta,b} (\{ 0 \}) \delta_{0} + P_{\Delta,b}^{ac}$ where $P_{\Delta,b}^{ac}$ is absolutely continuous. In this case, $P_{\Delta}$ itself is not absolutely continuous, but $y^{4}P_{\Delta,b} = y^{4}P_{\Delta,b}^{ac}$ is absolutely continuous. 

Condition (iii) is concerned with smoothness of the scaled L\'{e}vy density $x^{2}\rho$. Condition (iii) allows the L\'{e}vy density to have a ``cusp'' at the origin. For example, a Gamma process  has L\'{e}vy density $\rho(x) = \alpha x^{-1}e^{-\beta x}1_{(0,\infty)}(x)$ for some $\alpha,\beta > 0$; in this case, the L\'{e}vy density itself $\rho$ is discontinuous (at the origin), but the scaled version $x^{2}\rho$ is globally Lipschitz continuous. 
Condition (iv) is concerned with the kernel function $W$. We assume that $W$ is a $(p+1)$-th order kernel, but allow for the possibility that $\int_{\R} x^{p+1}W(x)dx=0$. We will provide examples of kernel functions satisfying Condition (iv) in Remark \ref{rem: kernel} below.
It is worth mentioning that since the Fourier transform of $W$ has compact support, the support of the kernel function $W$ itself is necessarily unbounded (which is a consequence of the Paley-Wiener theorem; see \cite{StWe71}, Theorem 4.1), and we will use global regularity of the scaled L\'{e}vy density $x^{2}\rho$ to suitably bound the deterministic bias, despite that we focus on constructing confidence bands on a compact set that does not intersect the origin. It could be possible to replace Condition (iii) by a ``local'' smoothness condition on $x^{2}\rho$, but we shall keep current Condition (iii) for the simplicity of the exposition.

Condition (v) is concerned with the bandwidth and the time span $\Delta$. The condition  $h \gtrsim \Delta^{1/2}$ ensures that $\inf_{|u| \leq h^{-1}}  | \varphi_{\Delta} (u)| \gtrsim 1$ (see Lemma \ref{lem: lower bound on chf}). 
The condition  $h^{r} \sqrt{n\Delta h \log n} \to 0$ is an ``undersmoothing'' condition. 
Inspection of the proof of Theorem \ref{thm: Gaussian approximation} shows that, without the condition $h^{r} \sqrt{n\Delta h \log n} \to 0$, we have that
\[
\| \hat{\rho} - \rho \|_{I} = O_{\Pr}\{ (n\Delta h)^{-1/2}\sqrt{\log n} \} + O(h^{r}),
\]
where the  $O(h^{r})$ term comes from the deterministic bias. The right hand side is optimized by choosing $h  \sim \left ( \frac{\log n}{n \Delta} \right )^{-1/(2r+1)}$, and the optimal rate for $\| \hat{\rho} - \rho \|_{I}$ is $\left ( \frac{\log n}{n \Delta} \right )^{-r/(2r+1)}$. 
For our confidence bands to be valid, however, we have to choose bandwidths of smaller order  (by $\log n$ factors) than the optimal one for estimation under the sup-norm, so that the bias term is negligible relative to the ``variance'' or stochastic term. Undersmoothing bandwidths are commonly used in construction of confidence bands. See Section 5.7 in \cite{Wa06} for related discussions.
For example, if we choose $h \sim (n\Delta)^{-1/(2r+1)}(\log n)^{-1}$, then Condition (v) reduces to 
\begin{equation}
n^{(2r+1)\delta/r- 1} (\log n)^{(2r+1)/(2r)} \ll \Delta \lesssim n^{-1/(r+3/2)} (\log n)^{-(2r+1)/(r+3/2)},
\label{eq: condition on Delta}
\end{equation}
where the condition $\delta \in (0,\frac{r}{2r+3})$ ensures that  $(2r+1)\delta/r- 1 < -1/(r+3/2)$.

Condition (vi) is concerned with  the pilot estimator of $\sigma^{2}$. 
Since the set $I$ is away from the origin, if $|W(x)| = O(|x|^{-r-1})$ as $|x| \to \infty$, then $\| h^{-1} W(\cdot/h) \|_{I} = O(h^{r}) = o\{ (n\Delta h \log n)^{-1/2} \}$, so that the pilot estimator $\hat{\sigma}^{2}$ need not be even consistent (e.g. we may take $\hat{\sigma}^{2} = 0$). 
Note that as long as $\int_{\R} |x|^{p+2} |W(x)|dx < \infty$, the Fourier transform $\varphi_{W}$ is $(p+2)$-times continuously differentiable, so that $|W(x)| = o(|x|^{-p-2}) = o(|x|^{-r-1})$ as $|x| \to \infty$ (however, as noted before, in our simulations studies, we found that, when $\sigma > 0$, using a proper estimator for $\hat{\sigma}^{2}$ improves upon the empirical performance of the estimator $\hat{\rho}$ and the confidence bands). 


The following theorem  derives a Gaussian approximation result. Recall that a Gaussian process $\{ Z(x) : x \in I \}$ is a tight random variable in $\ell^{\infty}(I)$ if and only if $I$ is totally bounded for the intrinsic pseudo-metric $d_{Z}(x,y) = \sqrt{\Ep[ (Z(x)-Z(y))^{2} ]}$ for $x,y \in I$, and $Z$ has sample paths almost surely uniformly $d_{Z}$-continuous \citep[cf.][p.41]{vaWe96}. In that case, we say that $Z$ is a tight Gaussian random variable in $\ell^{\infty}(I)$. 

\begin{theorem}[Gaussian approximation]
\label{thm: Gaussian approximation}
Under Assumption \ref{as: assumption 1}, for each sufficiently large $n$, there exists a tight Gaussian random variable $Z_{n}^{G}$ in $\ell^{\infty}(I)$ with mean zero and  covariance function
\begin{align*}
\Cov \left (Z_{n}^{G}(x),Z_{n}^{G}(y) \right ) &= \frac{1}{s_{n}(x)s_{n}(y)} \Bigg \{ \int_{\R} K_{n}((x-w)/h)K_{n}((y-w)/h) w^{4} P_{\Delta}(dw) \\
&\quad - \left ( \int_{\R} K_{n}((x-w)/h) w^{2}P_{\Delta}(dw) \right ) \left ( \int_{\R} K_{n}((y-w)/h) w^{2}P_{\Delta}(dw) \right ) \Bigg \}
\end{align*}
for $x,y \in I$, and such that as $n \to \infty$, 
\[
\sup_{z \in \R}\left | \Pr \{ \| \sqrt{n}\Delta h x^{2}(\hat{\rho} - \rho)/s_{n} \|_{I} \leq z \}- \Pr \{ \| Z_{n}^{G} \|_{I} \leq z \} \right |\to 0. 
\]
\end{theorem}

 The distribution of the Gaussian process $Z_{n}^{G}$ that appears in Theorem \ref{thm: Gaussian approximation} changes with $n$, and so the approximation is only an ``intermediate'' one. However, such an intermediate Gaussian approximation is sufficient to prove validity of  bootstraps \citep[cf.][]{ChChKa14b}. Building on Theorem \ref{thm: Gaussian approximation}, the following theorem formally establishes asymptotic validity of the two bootstrap confidence bands. 

\begin{theorem}[Validity of bootstrap confidence bands]
\label{thm: bootstrap}
Under Assumption \ref{as: assumption 1}, for either $B \in \{ MB, EB \}$, we have 
\[
\Pr \{ \rho (x) \in \hat{\mathcal{C}}_{1-\tau}^{B}(x) \ \forall x \in I \} \to 1-\tau
\]
as $n \to \infty$. Furthermore, the supremum width of the band $\hat{\mathcal{C}}_{1-\tau}^{B}$ is $O_{\Pr}\{ (n\Delta h)^{-1/2} \sqrt{\log n} \}$.  
\end{theorem}

\begin{remark}
For example, if we choose $h \sim (n\Delta)^{-1/(2r+1)}(\log n)^{-1}$, then provided that (\ref{eq: condition on Delta}) is satisfied, the supremum width of the band $\hat{\mathcal{C}}_{1-\tau}^{B}$ is $O_{\Pr}\{ (n\Delta)^{-r/(2r+1)} \log n \}$. 
\end{remark}

The proofs of Theorems \ref{thm: Gaussian approximation} and \ref{thm: bootstrap} build on non-trivial applications of the intermediate Gaussian and bootstrap approximation theorems developed in \cite{ChChKa14a,ChChKa14b,ChChKa16}. 
We would like to point out here that there are several non-trivial steps in proving Theorems \ref{thm: Gaussian approximation} and \ref{thm: bootstrap}. For example, we will require to show that $\inf_{x \in I} s_{n}^{2}(x) \gtrsim \Delta h$, but since the increment distribution $P_{\Delta}$ may have a discrete component and degenerates to $\delta_{0}$ as $\Delta \to 0$, and $K_{n}$ changes with $n$ and has unbounded support, lower bounding the variance function $s_{n}^{2}(x)$ is non-trivial. Second, a crucial fact in the proofs of Theorems \ref{thm: Gaussian approximation} and \ref{thm: bootstrap} is that the function class
\begin{equation}
\left \{ y \mapsto y^{2}K_{n}((x-y)/h) : x \in I \right \} 
\label{eq: function class}
\end{equation}
is a Vapnik-Chervonenkis (VC) type class. In view of Lemma 1 in \cite{KaSa16}, it is not difficult to verify that the function class (\ref{eq: function class}) is VC type for an envelope function of the form $y \mapsto \text{const}.\times y^{2}$; however, using this envelope function will require more restrictive moment conditions on $\rho$ (we will require at least finite eighth moment of $\rho$) and additional conditions on the  smoothness level $r$ depending on the moments conditions on $\rho$. In fact, although it is not apparent, it turns out that, under our assumption, the function $y \mapsto y^{2}K_{n}((x-y)/h)$ is bounded uniformly in $n$ and $x \in I$. So, we will verify that the function class (\ref{eq: function class}) is VC type for a constant envelope function, which requires a different and non-trivial idea; cf. Lemma \ref{lem: VC type}. 


\begin{remark}[Examples of kernel functions]
\label{rem: kernel}
Construction of a kernel function satisfying Condition (iv) is typically done by specifying its Fourier transform $\varphi_{W}$. 
Let $\varpi: \R \to \R$ be a function that is even (i.e., $\varpi (u) = \varpi (-u)$), supported in $[-1,1]$, and $(4 \vee (p+3))$-times continuously differentiable, and such that
\[
\varpi^{(\ell)} (0)
=
\begin{cases}
1 & \ell=0 \\
0 & \ell=1,\dots,p
\end{cases}
.
\]
Then the function $W(x) :=  \frac{1}{2\pi}\int_{\R} e^{-iux} \varpi(u) du$ is real-valued, $|W(x)| = o(|x|^{-p-3})$ as $|x| \to \infty$ (which follows from changes of variables), so that $(1 \vee |x|^{p+1})W$ is integrable, 
and 
\[
\int_{\R}x^{\ell}W(x) dx = i^{-\ell}\varpi^{(\ell)}(0) = 
\begin{cases}
1 & \ell=0 \\
0 & \ell=1,\dots,p
\end{cases}
.
\]
Here, since $W$ is even, if $p$ is even, we also have $\int_{\R} x^{p+1} W(x) dx = 0$. Examples of $\varpi$ include: $\varpi (u) = (1-u^{2})^{k}1_{[-1,1]}(u)$ for $k \geq 5 \vee (p+4)$, and 
\begin{equation}
\varpi(u) =
\begin{cases}
1 &\text{if } \vert u \vert \leq c\\
\exp\left\{ \frac{-b \exp(-b/(\vert u \vert - c)^2)}{(\vert u \vert - 1)^2} \right\} &\text{if } c < \vert u \vert < 1\\
0 &\text{if } 1 \leq \vert u \vert 
\end{cases}
\label{eq: flap-top}
\end{equation}
for $0 < c < 1$ and $b > 0$. For the latter case, $\varpi$ is infinitely differentiable with $\varpi^{(\ell)}(0) = 0$ for all $\ell \geq 1$, so that its inverse Fourier transform $W$, called a flap-top kernel, is of infinite order, i.e., $\int_{\R} x^{\ell} W(x) dx=0$ for all integers $\ell \geq 1$ \citep[cf.][]{McPo04}. 
\end{remark}

\subsection{Discussions on Condition (ii) in Assumption \ref{as: assumption 1}}
\label{sec: condition (ii)}

In this subsection, we provide primitive regularity conditions that guarantee Condition (ii) in Assumption \ref{as: assumption 1}. We make the following assumption. 

\begin{assumption}
\label{as: assumption 2}
Assume that $x^{2}\rho$ is continuous; $x^{4} \rho \in \ell^{\infty}(\R)$; and $\rho$ is positive on $I^{\varepsilon_{1}}$ for some sufficiently small $\varepsilon_{1} > 0$ such that $0 \notin I^{\varepsilon_{1}}$. Furthermore, there exists $b \in \R$ such that the signed measure $y P_{\Delta,b}(dy)$ has a Lebesgue density bounded (in absolute value) by $C \log (1/\Delta)$ for all sufficiently small  $\Delta > 0$ for some constant $C > 0$. 
\end{assumption}

 Assumption 4.2 ensures Condition (ii) in Assumption \ref{as: assumption 1} to hold. 

\begin{proposition}
\label{prop: condition (ii)}
Condition (ii) in Assumption \ref{as: assumption 1} is satisfied  under Assumption \ref{as: assumption 2}. 
\end{proposition}

\begin{remark}
\label{rem: condition (ii)}
The first part of Assumption \ref{as: assumption 2} is not restrictive. Recall that we are assuming finiteness of the fourth moment of $\rho$  in Assumption \ref{as: assumption 1}, so that the requirement that $x^{4} \rho \in \ell^{\infty}(\R)$ appears to be innocuous. 
Proposition 16 in \cite{NiReSoTr16} provides primitive regularity conditions that ensure that $yP_{\Delta}$ has a density bounded uniformly in $\Delta$; see Assumption 15 in \cite{NiReSoTr16} (we allow for the density of $y P_{\Delta,b} (dy)$ to grow like $\log (1/\Delta)$ to cover cases where $\sigma=0$ and the L\'{e}vy density behaves like $x^{-2}$ near the origin; see below). 
In particular, Assumption 15 in \cite{NiReSoTr16} covers many of basic examples of L\'{e}vy processes.

For example, consider the following two simple cases:
\begin{quote}
(a) $\sigma > 0$; \ or \ (b) $\sigma=0$ and $\| x \rho \|_{\R} < \infty$.
\end{quote}
 ($\| x \rho \|_{\R} < \infty$ together with the assumption that $\int_{\R} x^{2} \rho (x) dx < \infty$ ensure that $\int_{\R} |x|\rho(x) < \infty$.) 

For Case (a),  in view of $\varphi'_{\Delta}(u) = i \int_{\R} e^{iuy} yP_{\Delta}(dy)$ together with the fact that $\varphi_{\Delta}'$ is integrable, a Lebesgue density of $yP_{\Delta}$ exists and is given by 
\[
\frac{1}{2\pi} \int_{\R} e^{-iuy} i^{-1} \varphi_{\Delta}'(u) du.
\]
Since $|\varphi_{\Delta}'(u)| \leq \text{const.}\, \Delta (1+|u|) e^{-\Delta \sigma^{2}u^{2}/2}$, we see that 
\[
\int_{\R} |\varphi_{\Delta}'(u)| du \lesssim \Delta \int_{\R} (1+|u|) e^{-\Delta \sigma^{2}u^{2}/2} du  \lesssim 1,
\]
which shows that the density of $yP_{\Delta}$ is bounded uniformly in $\Delta$. 

For Case (b), observe that $\psi (u) = iub + \int_{\R} (e^{iux} - 1) \rho(x) dx$ with $b = \gamma_{c} - \int_{\R} x\rho(x) dx$, and 
\[
\int_{\R} e^{iuy} y P_{\Delta,b}(dy)= i^{-1}(\Ep[ e^{iu (L_{\Delta} - b\Delta)} ])' = \Delta\left  (\int_{\R} e^{iuy} y\rho(y) dy \right ) \Ep[ e^{iu (L_{\Delta} - b\Delta)} ].
\]
Applying the Fourier inversion, we see that $yP_{\Delta,b}$ has a Lebesgue density
\[
\Delta \int_{\R} (y-w)\rho (y-w) P_{\Delta,b}(dw),
\]
which is bounded (in absolute value) by $\Delta \| x \rho \|_{\R} P_{\Delta,b}(\R) \lesssim \Delta$. 
Hence, in these two cases, $yP_{\Delta,b}$ has a Lebesgue density bounded uniformly in $\Delta$ for some $b \in \R$. For other more complicated cases, we refer to Proposition 16 in \cite{NiReSoTr16}. 
\end{remark}

For the symmetric tempered stable process with stability index $\alpha=1$ and the Normal Inverse Gaussian process discussed in the next section, whose L\'{e}vy densities behave like $x^{-2}$ near the origin, Proposition 16 in \cite{NiReSoTr16} appears not to be  directly applicable. 
To cover those cases, we present the following lemma. 
\begin{lemma}
\label{lem: infinite variation}
Suppose that $\sigma=0$ and the L\'{e}vy density $\rho$ satisfies that for some constants $C > c > 0$,
\[
c \leq \frac{1}{\varepsilon} \int_{|x| \leq \varepsilon} x^{2}\rho(x) dx \leq C \quad \text{and} \quad  \int_{|x| > \varepsilon} |x| \rho(x) dx \leq C  \left ( 1+\log (1/\varepsilon) \right )
\]
 for all $\varepsilon \in (0,1)$. 
Then the signed measure $y P_{\Delta}(dy)$ has a Lebesgue density bounded (in absolute value) by $C' \log (1/\Delta)$ for all 
sufficiently small  $\Delta > 0$ for some constant $C' > 0$. 
\end{lemma}


\section{Examples}
\label{sec: examples}

In this section, we provide some examples of L\'{e}vy processes that satisfy Conditions (i)-(iii) in Assumption \ref{as: assumption 1}. For detailed properties of L\'{e}vy processes discussed below, we refer to \cite{CoTa04}. 
The first four examples are purely non-Gaussian L\'{e}vy processes (i.e., $\sigma =0$), and we allow them to have drift terms. 


\begin{example}[\textbf{Compound Poisson process}] 
\label{ex: compound poisson}
A compound Poisson process with drift is a stochastic process of the form  $L_{t} = bt + \sum_{k=1}^{N_{t}}X_{k}$, where $N=(N_{t})_{t \geq 0}$ is a Poisson process with constant intensity $\lambda>0$ and $\{ X_{k} \}_{k=1}^{\infty}$ is a sequence of i.i.d. random variables with common distribution $F_{X}$ (jump size distribution) independent of  $N$. We assume that $F_{X}$ has Lebesgue density $f_{X}$.  In this case, the characteristic exponent is $\psi (u)= iub + \int_{\R} (e^{iux}-1) \lambda f_{X}(x)dx$, and so the L\'{e}vy density is $\rho = \lambda f_{X}$. Conditions (i) and (iii) can be directly translated to conditions on $f_{X}$. From Proposition \ref{prop: condition (ii)} and Remark \ref{rem: condition (ii)},  Condition (ii) is satisfied if $x^{2}f_{X}$ is continuous, $(|x| \vee x^{4})f_{X} \in \ell^{\infty}(\R)$, and $f_{X}$ is positive on $I^{\varepsilon_{1}}$ for some $\varepsilon_{1} > 0$ such that $0 \not \in I^{\varepsilon_{1}}$.

For example, the jump-part of the Merton model \citep{Me76} is a compound Poisson process with jump size distribution $N(c,v^{2})$ for some $c \in \R$ and $v^{2} > 0$, for which it is not difficult to verify that Conditions (i)-(iii) are satisfied with arbitrary large $r > 0$, and any compact set $I$ in $\R \setminus \{ 0 \}$. 
The jump-part of the Kou model \citep{Ko02} is a compound Poisson process with  jump size density  
\[
f_{X}(x) = p \lambda_{+} e^{-\lambda_{+} x}1_{(0,\infty)}(x) + (1-p) \lambda_{-} e^{-\lambda_{-} |x|}1_{(-\infty,0)}(x)
\]
for some $p \in [0,1]$ and $\lambda_{+}, \lambda_{-} > 0$.
Let  $I$  be any compact set in $\R \setminus \{ 0 \}$,  $(0,\infty)$, and $(-\infty,0)$ if $0 < p < 1, p=1$, and $p=0$, respectively. Then,   it is not difficult to verify that Conditions (i)-(iii) are satisfied with $r=3$ if $p=1/2$ and $\lambda_{+}=\lambda_{-}$, and $r=2$ otherwise. 

A compound Poisson process is a process of finite activity, i.e., has only finitely many jumps on any bounded time interval. 
\end{example}

%

\begin{example}[\textbf{Tempered stable process}]
\label{ex: tempered stable}
A tempered stable process with index $0 \leq \alpha < 2$ is a  L\'evy process with L\'evy density
\[
\rho(x) = {c_{+} \over x^{1+\alpha}}e^{-\lambda_{+}x}1_{(0,\infty)}(x) +{c_{-} \over |x|^{1+\alpha}}e^{-\lambda_{-}|x|}1_{(-\infty,0)}(x),
\]
where $c_{+}, c_{-} \geq 0, c_{+} + c_{-} > 0$, and $\lambda_{+}, \lambda_{-}>0$.  We assume that the stability index is restricted to $0 \leq \alpha < 1$.
Let $I$ be any compact set in $\R \setminus \{ 0 \}$, $(0,\infty)$, and $(-\infty,0)$ if $c_{+} c_{-} \neq 0$, $c_{-} = 0$, and $c_{+}=0$, respectively.
It is clear that Conditions (i) and (iii) are satisfied with $r=2$ if $\alpha=0, c_{+}=c_{-}$, and $\lambda_{+}=\lambda_{-}$, and $r=1-\alpha$ otherwise. 
For example, to see that $x^{2} \rho$ is $(1-\alpha)$-H\"{o}lder continuous in the latter case, it suffices to verify that $x^{1-\alpha} e^{-x}$ is $(1-\alpha)$-H\"{o}lder continuous on $[0,\infty)$, which can be verified as follows: for any $y > x \geq 0$, 
\[
|x^{1-\alpha} e^{-x} - y^{1-\alpha}e^{-y}| \leq \underbrace{x^{1-\alpha} e^{-x}}_{\text{bounded}} (1-e^{-(y-x)}) + \underbrace{e^{-y}}_{\leq 1} \underbrace{(y^{1-\alpha} - x^{1-\alpha})}_{\leq (y-x)^{1-\alpha}} \leq \text{const}. \times (y-x)^{1-\alpha},
\]
where we have used that $1-e^{-(y-x)} \leq (1-e^{-(y-x)})^{1-\alpha}   \leq (y-x)^{1-\alpha}$.

If $\alpha = 0$, then since $x\rho$ is bounded, in view of Remark \ref{rem: condition (ii)}, $y P_{\Delta,b}$ with $b=\gamma_{c} - \int_{\R } x \rho (x) dx$ has a Lebesgue density bounded uniformly in $\Delta$. 
If $0 <\alpha < 1$, then Assumption 15 Case (iii) in \cite{NiReSoTr16} is satisfied for $L_{t} - bt$ with $b=\gamma_{c} - \int_{\R} x\rho (x) dx$, and hence Proposition 16 in \cite{NiReSoTr16} yields that $yP_{\Delta,b}$ has a Lebesgue density bounded uniformly in $\Delta$. Therefore, by Proposition \ref{prop: condition (ii)}, Condition (ii) is satisfied in either case.

The tempered stable process includes Gamma, Inverse Gaussian, and Variance Gamma processes as  special cases. A Gamma process corresponds to the case with $\alpha=0$ and $c_{-}=0$; an Inverse Gaussian process corresponds to the case with $\alpha=1/2$ and $c_{-}=0$; and a Variance Gamma process corresponds to the case with $\alpha=0$ and $c_{+}=c_{-} > 0$. 

The tempered stable process is a process of infinite activity, i.e., has infinitely many jumps on any bounded time interval. 
\end{example}

The L\'{e}vy density in each of Examples of  \ref{ex: compound poisson} and \ref{ex: tempered stable} has finite first moment, and therefore sample paths of the process have finite variation on any bounded time interval \cite[][Theorem 21.9]{Sa99}. On the other hand, the following two examples have infinite variation on any bounded time interval. 

\begin{example}[\textbf{Symmetric tempered stable process with $\alpha=1$}]
In the previous example, consider the case where $\alpha=1, c_{+}=c_{-} = c > 0$, and $\lambda_{+} = \lambda_{-} = \lambda > 0$, i.e., 
\[
\rho(x) = \frac{c}{x^{2}} e^{-\lambda |x|}, \ x \neq 0.
\]
In this case, Condition (i) is trivially  satisfied, and since $x^{2} \rho(x)$ extends to a Lipschitz continuous function on $\R$ as $x^{2} \rho(x) = c e^{-\lambda |x|}, x \in \R$, Condition (iii) is satisfied with $r = 1$. It is not difficult to verify that the assumption in Lemma \ref{lem: infinite variation} is satisfied, and 
therefore, by Proposition \ref{prop: condition (ii)}, Condition (ii) is satisfied for any compact set $I$ in $\R \setminus \{ 0 \}$. 
\end{example}

\begin{example}[\textbf{Normal Inverse Gaussian process}]
\label{ex: NIG}
A Normal Inverse Gaussian (NIG) process is a purely non-Gaussian L\'{e}vy process with L\'{e}vy density 
\[
\rho (x) = \frac{\delta \alpha e^{\beta x}}{\pi |x|}\mathbb{K}_{1}(\alpha |x|), \ x \neq 0
\]
where $\mathbb{K}_{1}$ is the modified Bessel function of the second kind with order $1$, and has integral representation 
\[
\mathbb{K}_{1}(z) = \frac{1}{2} \int_{0}^{\infty}e^{-\frac{z}{2}(t + t^{-1})}dt, \ z > 0.
\]
The parameters $\alpha, \beta, \delta$ are restricted such that $0 \leq | \beta | < \alpha$ and $\delta > 0$. Since $\mathbb{K}_{1}(z)$ decays like $z^{-1/2} e^{-z}$ as $z \to +\infty$, Condition (i) is satisfied. By a change of variables, we have that 
\[
\rho (x) = \frac{\delta e^{\beta x}}{\pi x^{2}} \int_{0}^{\infty} e^{-t - \frac{\alpha^{2}x^{2}}{4t}} dt, \ x \neq 0.
\]
Since the integral on the right hand side is well-defined for $x=0$,  $x^{2}\rho(x)$ extends to a continuous function on $\R$ as 
\begin{equation}
\rho_{\sharp}(x) := x^{2}\rho(x) = \frac{\delta e^{\beta x}}{\pi}\int_{0}^{\infty} e^{-t - \frac{\alpha^{2}x^{2}}{4t}} dt, \ x \in \R.
\label{eq: NIG density}
\end{equation}
Furthermore, it is not difficult to verify that the assumption in Lemma \ref{lem: infinite variation} is satisfied, and 
therefore, by Proposition \ref{prop: condition (ii)}, Condition (ii) is satisfied for any compact set $I$ in $\R \setminus \{ 0 \}$.

Finally, it appears to be difficult to directly verify Condition (iii) to the NIG process, but inspection of the proofs of Theorems \ref{thm: Gaussian approximation} and \ref{thm: bootstrap} shows that Condition (iii) is used only to bound the deterministic bias $\| \rho_{\sharp} * (h^{-1}W(\cdot/h)) - \rho_{\sharp} \|_{I}$. Fortunately, for the NIG process, it is possible to directly bound the deterministic bias $\| \rho_{\sharp} * (h^{-1}W(\cdot/h)) - \rho_{\sharp} \|_{I} \lesssim h^{r}$ for any $r \in (0,1)$; see below. Therefore, the conclusions of Theorems \ref{thm: Gaussian approximation} and \ref{thm: bootstrap} hold true for the NIG process with any $r \in (0,1)$, provided that other technical conditions (Conditions (iv)-(vi)) are satisfied. 

\begin{lemma}
\label{lem: NIG bias}
Let $\rho_{\sharp}$ be as in (\ref{eq: NIG density}), and let $W: \R \to \R$ be an integrable function such that $\int_{\R} W(x) dx = 1$ and $\int_{\R} x^{2}|W(x)| dx < \infty$.
Then for any $r \in (0,1)$ and any nonempty compact set $I$ in $\R$, we have that $\| \rho_{\sharp} * (h^{-1}W(\cdot/h)) - \rho_{\sharp} \|_{I} \lesssim h^{r}$,
where $*$ denotes the convolution. 
\end{lemma}

\end{example}

\begin{example}[\textbf{Brownian motion + purely non-Gaussian L\'{e}vy process}]
Let $L_{t} = \sigma B_{t} + J_{t}$, where $\sigma > 0$,  $B = (B_{t})_{t \geq 0}$ is a standard Brownian motion, and $J = (J_{t})_{t \geq 0}$ is a purely non-Gaussian L\'{e}vy process independent of $B$ (with drift). We assume that $J$ is one of purely non-Gaussian L\'{e}vy processes in Examples \ref{ex: compound poisson}--\ref{ex: NIG}. 
For the compound Poisson process case, we assume that $x^{4} f_{X} \in \ell^{\infty}(\R)$, and Conditions (i) and (iii) are satisfied for $\rho = \lambda f_{X}$.
In view of Remark \ref{rem: condition (ii)}, it is clear that Conditions (i)-(iii) are satisfied  with  $r$ given in Examples \ref{ex: compound poisson}--\ref{ex: NIG}, as long as $I$ is properly chosen. 
\end{example}

\section{Simulation results}
\label{sec: simulation results}

\subsection{Simulation framework}

In this section, we present simulation results to evaluate the finite-sample performance of the proposed confidence bands. 
We consider the following data generating processes.

\begin{enumerate}
\item \textbf{Brownian motion + compound Poisson process}. Let $L_{t} = \sigma B_{t} + J_{t}$,
where $\sigma \geq 0, B = (B_{t})_{t \geq 0}$ is a standard Brownian motion, and $J_{t} = \sum_{k=1}^{N_{t}}X_{k}$ is a compound Poisson process (see Example 5.1). The Poisson process $N$ has intensity $\lambda > 0$. We set $\lambda = 4$ or $10$. We consider two types of jump size distributions. For the first case, $f_{X}$ is the density of the normal distribution $N(0,v^{2})$ and the L\'{e}vy density is  $\rho(x) =  \lambda e^{-x^2/(2v^2)}/\sqrt{2\pi v^2}$. We denote this case by BCN($\sigma, v$). For the second case, $f_{X}$ is the density of the Laplace distribution with location $0$ and scale $v > 0$, and the L\'{e}vy density is  $\rho(x) =  \lambda e^{-|x|/v}/(2v)$. We denote the latter case by BCL($\sigma,v$). 
For BCN($\sigma,v$) and BCL($\sigma,v$), we take $I=[-0.75,-0.25] \cup [0.25,0.75]$.

If $\sigma = 0$, then $L$ reduces to the compound Poisson process $J$, for which we set $\hat{\sigma}^{2} = 0$. 
In case of  $\sigma>0$, we estimate $\sigma^2$ by the TRV estimator $\hat{\sigma}^{2}_{TRV}$ with $\alpha_{0} = 3$ and $\theta_{0} = 0.48$ (see Example \ref{ex: volatility estimator}). 
We also examined the performance of the confidence bands with estimated $\sigma^{2}$ in case of $\sigma=0$, but the simulation results are almost identical to those under $\hat{\sigma}^{2} = 0$. Hence, we only report the simulation results with $\hat{\sigma}^{2} =0$ in case of $\sigma= 0$. The same comment applies to the Gamma process case.

\item \textbf{Gamma process}. A Gamma process is a pure jump L\'evy process with L\'evy density $\rho(x) =  c_{+}x^{-1}e^{- \lambda_{+} x}1_{(0,\infty)}(x)$ (see Example 5.2). We denote this case by G($c_{+},\lambda_{+}$). The increment distribution $P_{\Delta}$ is the Gamma distribution with shape parameter $c_{+} \Delta$ and scale parameter $1/\lambda_{+}$. For the Gamma process case, we take $I=[0.25,0.75]$. 
\end{enumerate}

We use the kernel function $W$ whose Fourier transform $\varphi_{W}$ is specified by (\ref{eq: flap-top}), where we choose $b=1$ and $c=0.05$. We consider the following configurations for the sample size $n$ and the time span $\Delta$: $n = 5 \times 10^{4}$ or $10^{5}$, and $\Delta = 0.01$ or $0.005$. Here $n \Delta$ ranges from $250$ to $1000$.

\begin{remark}
From a theoretical point of view, we do not have to estimate $\sigma^2$ even when $\sigma>0$ (see  the discussion on Condition (vi) in Assumption 4.1). However,  in case of $\sigma > 0$, we found that the estimate $\hat{\rho}$ with $\hat{\sigma}^{2} = 0$ tends to be less precise near the origin than that with $\hat{\sigma}^{2} = \hat{\sigma}^{2}_{TRV}$. So, from a practical point of view, we recommend to estimate $\sigma^2$ when implementing our methods. 
\end{remark}


\subsection{Bandwidth selection}
Now, we discuss bandwidth selection. We adapt an idea of \cite{BiDuHoMu07} on bandwidth selection in density deconvolution. From a theoretical point of view, for our confidence bands to work, we have to choose bandwidths that are of smaller order than the optimal rate for estimation under the the sup-norm loss. At the same time, choosing a too small bandwidth results in a too wide confidence band. 
Therefore, heuristically, we should choose a bandwidth ``slightly'' smaller than the optimal one that minimizes the $L^{\infty}$-distance $\| \hat{\rho} - \rho \|_{I}$. 

Let $\hat{\rho}_{h}$ denote the spectral estimate with bandwidth $h$. Figure \ref{fig:A1} depicts five realizations of the $L^{\infty}$-distance $\| \hat{\rho}_{h} - \rho \|_{I}$ with different bandwidth values for BCN(0,1/2) with $\lambda=4$ and G(0.2,1), both with $(n,\Delta) = (5\times 10^4, 0.01)$. It is observed that as $h$ increases, $\| \hat{\rho}_{h} - \rho \|_{I}$ is sharply decreasing for $h < h^{\star}$ (say), and for $h > h^{\star}$, $\| \hat{\rho}_{h} - \rho \|_{I}$ is slowly increasing. We aim to choose a bandwidth slightly smaller than $h^{\star}$. Of course, the problem is that the value of $\| \hat{\rho}_{h} - \rho \|_{I}$ is unknown to us. 
Now, Figure \ref{fig:A2} depicts five realizations of the $L^{\infty}$-distance between the estimates of $\rho$ with adjacent bandwidth values. To be precise, we prepare grids of bandwidths $h_{1} < \cdots < h_{J}$, and compute the $L^{\infty}$-distance $\| \hat{\rho}_{h_{j}} - \hat{\rho}_{h_{j-1}} \|_{I}$; Figure \ref{fig:A2} depicts those values with $h=h_{j} \ (j=2,\dots,J)$. It is observed that shape of $\| \hat{\rho}_{h_{j}} - \hat{\rho}_{h_{j-1}} \|_{I}$ partly ``mimics'' that of $\| \hat{\rho}_{h} - \rho \|_{I}$; in fact, $\| \hat{\rho}_{h_{j}} - \hat{\rho}_{h_{j-1}} \|_{I}$ is sharply decreasing for $h_{j} < h^{\star}$, but for $h_{j} > h^{\star}$, $\| \hat{\rho}_{h_{j}} - \hat{\rho}_{h_{j-1}} \|_{I}$ is almost flat. Our idea of bandwidth selection is as follows: starting from $j=2$, choose the first $j$ such that $\| \hat{\rho}_{h_{j}} - \hat{\rho}_{h_{j-1}} \|_{I}$ is below $\kappa \times \min \{ \| \hat{\rho}_{h_{k}} - \hat{\rho}_{h_{k-1}} \|_{I} : k=2,\dots,J \}$ for some $\kappa > 1$; our choice of the bandwidth is $h = h_{j}$. Heuristically, this rule would choose a bandwidth ``slightly'' smaller than $h^{\star}$ (as long as the threshold value $\kappa$ is reasonably chosen). Formally, we employ the following rule for bandwidth selection.


\begin{enumerate}
\item Set a pilot bandwidth $h^{P} = M\Delta^{1/2}$ for some $M > 1$, and make a list of candidate bandwidths $h_{j} = jh^{P}/J$ for $j=1,\dots,J$. 
\item Choose the smallest bandwidth $h_{j} \ (j \geq 2)$ such that the adjacent $L^{\infty}$-distance $\| \hat \rho_{h_j} - \hat \rho_{h_{j-1}}\|_I$ is smaller than $\kappa \times \min \{ \| \hat{\rho}_{h_{k}} - \hat{\rho}_{h_{k-1}} \|_{I} : k=2,\dots,J \}$ for some $\kappa > 1$
\end{enumerate}

In this simulation study, we choose $M=2, J=20$, and $\kappa= 20$. In practice, it is also recommended to make use of visual information on how $\| \hat{\rho}_{h_{j}} - \hat{\rho}_{h_{j-1}} \|_{I}$ behaves as $j$ increases when determining the bandwidth. 

\begin{remark}
We have also examined a version of the bandwidth selection rule with $\hat{\rho}_{h_{j}} - \hat{\rho}_{h_{j-1}}$ replaced by $x^{2}(\hat{\rho}_{h_{j}} - \hat{\rho}_{h_{j-1}})$, but found that the rule described above shows better performances in terms of coverage probabilities.  
So, we present simulation results with the above rule. 
\end{remark}

\begin{figure}[H]
  \begin{center}
    \begin{tabular}{cc}

      \begin{minipage}{0.5\hsize}
        \begin{center}
          \includegraphics[clip, width=6cm]{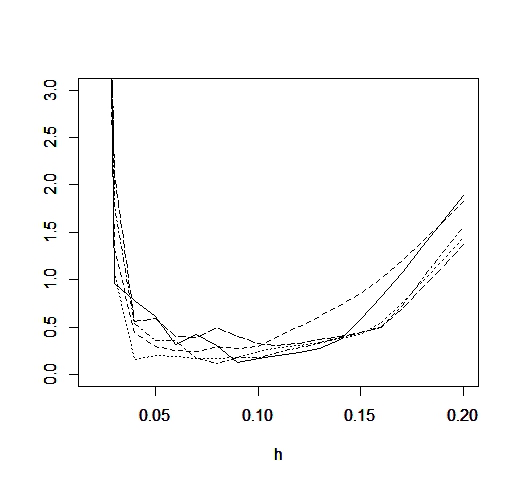}
        \end{center}
      \end{minipage}

      \begin{minipage}{0.5\hsize}
        \begin{center}
          \includegraphics[clip, width=6cm]{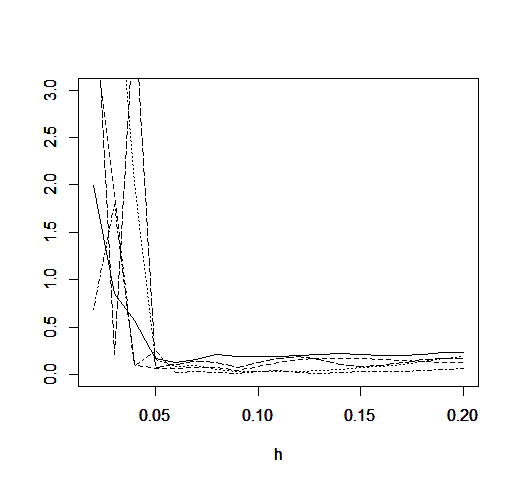}
        \end{center}
      \end{minipage}

    \end{tabular}
    \caption{$L^{\infty}$-distance between the true L\'{e}vy density $\rho$ and estimates $\hat{\rho}$ for different bandwidth values. The left figure corresponds to BCN(0,1/2) with $\lambda=4$, and the right figure corresponds to G(0.2,1), both with $(n,\Delta) = (5\times 10^4, 0.01)$. \label{fig:A1}}   
  \end{center}
\end{figure}

\begin{figure}[H]
  \begin{center}
    \begin{tabular}{cc}

      \begin{minipage}{0.5\hsize}
        \begin{center}
          \includegraphics[clip, width=6cm]{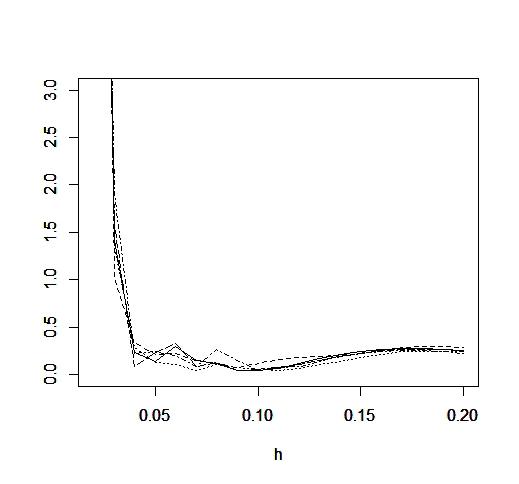}
        \end{center}
      \end{minipage}

      \begin{minipage}{0.5\hsize}
        \begin{center}
          \includegraphics[clip, width=6cm]{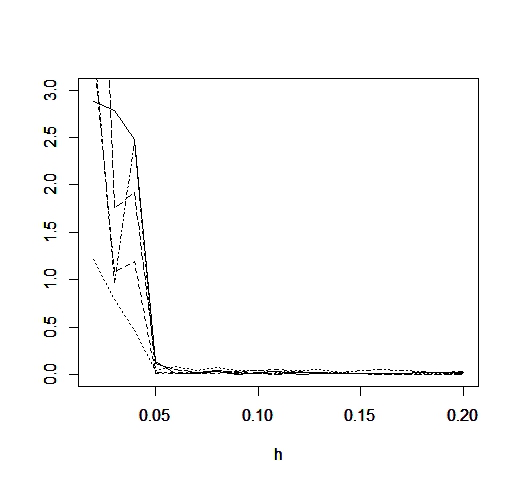}
        \end{center}
      \end{minipage}

    \end{tabular}
    \caption{$L^{\infty}$-distance between the estimates of $\rho$ with adjacent bandwidths for BCN(0,1/2) with $\lambda=4$ (left) and G(0.2,1) (right). \label{fig:A2}}    
  \end{center}
\end{figure}

\subsection{Simulation results}
In this simulation study, we focus on the multiplier bootstrap (MB) confidence band $\hat{\mathcal{C}}_{1-\tau}^{MB}$. 
We present simulated coverage probabilities of the MB confidence band together with simulated values of  the expected mean width of the band $\Ep \left [ \frac{1}{|I|} \int_{I}|\widehat{\mathcal{C}}_{1-\tau}^{MB}(x)|dx \right ]$, 
where $| A |$ denotes the Lebesgue measure for a measurable set $A \subset \R$. The number of Monte Carlo repetitions is $250$. To compute the critical value $\hat{c}_{n}^{MB}(1-\tau)$, we generate $1,500$ multiplier bootstrap replications for each run of the simulations. 

Tables \ref{table: T1} and \ref{table: T2} present simulation results for models BCN($\sigma,v$) and BCL($\sigma,v$) with $\sigma \in \{ 0,1 \}, v = 1/2$, and G($0.2,1$) under $(n,\Delta) = (5 \times 10^4, 0.005), (10^5, 0.005), (5 \times 10^4, 0.01)$, and $(10^5, 0.01)$. Comparing BCN($0,1/2$) with BCL($0,1/2$) and G(0.2,1), we find that BCN($0,1/2$) is apt to give more accurate coverage probabilities. This is partly due to the smoothness of L\'{e}vy densities. Since the normal density is smoother around the origin than those of Laplace and Gamma-L\'{e}vy densities, the estimate $\widehat{\rho}$ for BCN tends to be less biased than that for other cases.     
Figure \ref{fig:A3} depicts $90\%$ MB confidence bands for BCN($0,1/2$) (left), BCL($0,1/2$) (center), both with $\lambda=10$, and G(0.2,1) (right) with $(n,\Delta) = (5 \times 10^4, 0.01)$ (top row) and $(10^5, 0.01)$ (bottom row), based on one realization for each model. Figure \ref{fig:A4} depicts $90\%$ MB confidence bands for BCN($1,1/2$) and BCL($1,1/2$), both with $\lambda=10$. As seen from Figures 5 and 6, the width of the $90\%$ MB confidence band tends to increase near the origin when the Brownian component is present (i.e., $\sigma=1$). This partly comes from the difficulty of distinguishing small jumps from  fluctuations due to the Brownian component. 

Overall, the simulated coverage probabilities are reasonably close to the nominal coverage probabilities, although in some cases there are rooms for improvement. 
We also find that for every case, the expected mean width tends to decrease as $n\Delta$ increases, which is consistent with our theory. 
Notably, for BCN and BCL, the MB confidence bands exhibit similar performance for either case where the Brownian component is absent ($\sigma=0$) or present ($\sigma=1$). 

{\small
\begin{table}[H]
\begin{center}
\begin{tabular}{cccccccc}
\hline \hline
                                  &  &       & \multicolumn{5}{c}{Model} \\ \cline{4-8}
\multirow{2}{*}{\shortstack{Cov. Prob. \\$(1-\tau)$}} &  & &BCN(0,1/2) &BCN(0,1/2)  &BCL(0,1/2) &BCL(0,1/2)  & G(0.2,1)\\ 
                           & $\Delta$ & $n$              & $\lambda=4$ &$\lambda=10$ &$\lambda=4$ &$\lambda=10$  &\\ \hline

\multirow{8}{*}{0.90}  & 0.005 & $5\times 10^4$    &0.816       &0.824        &0.812        &0.824         &0.812\\
                          &          &                         &(1.150)     &(1.867)      &(0.818)      &(1.294)       &(0.317)\\
                          &          & $10^5$               &0.828       &0.836        &0.820       &0.808         &0.812\\
                          &          &                         &(0.816)     &(1.311)      &(0.573)      &(0.908)       &(0.213)\\ \cline{2-8}
                          &  0.01   & $5\times 10^4$   &0.824       &0.840        &0.820        &0.816         &0.820\\
                          &           &                        &(0.787)     &(1.285)      &(0.560)      &(0.922)       &(0.195)\\ 
                          &          &  $10^5$              &0.868       &0.856        &0.824        &0.796         &0.816\\
                          &         &                          &(0.545)     &(0.905)      &(0.399)      &(0.659)       &(0.131)\\  \hline

\multirow{8}{*}{0.95} & 0.005 & $5\times 10^4$    &0.908      &0.912        &0.908        &0.916         &0.908\\
                          &          &                         &(1.276)     &(2.071)      &(0.919)      &(1.453)       &(0.364)\\
                           &         & $10^5$               &0.912       &0.924        &0.908        &0.904         &0.920\\
                          &          &                         &(0.908)     &(1.455)      &(0.643)      &(1.019)       &(0.245)\\ \cline{2-8}
                           & 0.01   & $5\times 10^4$   &0.916       &0.928        &0.912        &0.916        &0.912\\
                          &          &                         &(0.876)     &(1.428)      &(0.631)      &(1.037)       &(0.226)\\ 
                           &         & $10^5$               &0.932       &0.936        &0.920        &0.876         &0.904\\
                          &           &                        &(0.607)     &(1.004)      &(0.449)      &(0.740)       &(0.153)\\  \hline 

\multirow{8}{*}{0.99} & 0.005 & $5\times 10^4$    &0.972      &0.976        &0.968        &0.984         &0.964\\
                          &          &                         &(1.532)     &(2.441)      &(1.110)      &(1.742)       &(0.454)\\
                           &         & $10^5$               &0.988       &0.984        &0.980        &0.976         &0.984\\
                          &          &                         &(1.090)     &(1.712)      &(0.771)      &(1.235)       &(0.301)\\ \cline{2-8}
                           & 0.01   & $5\times 10^4$   &0.972       &0.988        &0.984        &0.988        &0.980\\
                          &          &                         &(1.044)     &(1.695)      &(0.767)      &(1.265)       &(0.285)\\ 
                           &         & $10^5$               &0.984       &0.992        &0.992        &0.964         &0.988\\
                          &           &                        &(0.742)     &(1.184)      &(0.540)      &(0.892)       &(0.193)\\  \hline \hline
\end{tabular}
\caption{Empirical coverage probabilities of the MB confidence bands for BCN(0,1/2) and BCL(0,1/2) on $I = [-0.75, -0.25]\cup[0.25,0.75]$, and G(0.2,1) on $I = [0.25, 0.75]$, based on 250 Monte Carlo repetitions. 
Inside the parentheses are values of the expected mean width.\label{table: T1}}
\end{center}
\end{table}
}

{\small
\begin{table}[H]
\begin{center}
\begin{tabular}{ccccccc}
\hline \hline
                                  &  &       & \multicolumn{4}{c}{Model} \\ \cline{4-7}
\multirow{2}{*}{\shortstack{Cov. Prob. \\$(1-\tau)$}} &  & &BCN(1,1/2) &BCN(1,1/2)  &BCL(1,1/2) &BCL(1,1/2)  \\ 
                           & $\Delta$ & $n$              & $\lambda=4$ &$\lambda=10$ &$\lambda=4$ &$\lambda=10$  \\ \hline

\multirow{8}{*}{0.90}  & 0.005 & $5\times 10^4$    &0.804         &0.804          &0.832          &0.828           \\
                          &          &                         &(1.447)       &(2.425)        &(1.002)        &(1.536)         \\
                          &          & $10^5$               &0.808         &0.796          &0.820         &0.816        \\
                          &          &                         &(1.050)       &(1.756)        &(0.699)        &(1.070)       \\ \cline{2-7}
                          &  0.01   & $5\times 10^4$   &0.812         &0.808          &0.824          &0.804        \\
                          &           &                        &(1.113)       &(1.900)        &(0.870)        &(1.303)      \\ 
                          &          &  $10^5$              &0.824         &0.804          &0.816          &0.792      \\
                          &         &                          &(0.811)       &(1.409)        &(0.590)        &(0.946)     \\  \hline

\multirow{8}{*}{0.95} & 0.005 & $5\times 10^4$    &0904        &0.904          &0.912           &0.908   \\
                          &          &                         &(1.606)       &(2.692)        &(1.119)         &(1.720)       \\
                           &         & $10^5$               &0.908         &0.892         &0.916            &0.916      \\
                          &          &                         &(1.165)       &(1.945)        &(0.784)         &(1.199)      \\ \cline{2-7}
                           & 0.01   & $5\times 10^4$   &0.912         &0.908          &0.920           &0.904        \\
                          &          &                         &(1.241)       &(2.116)        &(0.975)         &(1.460)       \\ 
                           &         & $10^5$               &0.916         &0.896          &0.916           &0.896      \\
                          &           &                        &(0.901)       &(1.568)        &(0.661)         &(1.056)       \\  \hline 

\multirow{8}{*}{0.99} & 0.005 & $5\times 10^4$   &0.956        &0.968        &0.956         &0.988\\
                          &          &                         &(1.840)     &(3.113)      &(1.367)       &(2.067)\\
                           &         & $10^5$               &0.960       &0.972        &0.972        &0.980\\
                          &          &                         &(1.392)     &(2.301)      &(0.947)       &(1.460)\\ \cline{2-7}
                           & 0.01   & $5\times 10^4$   &0.972       &0.976        &0.984        &0.976\\
                          &          &                         &(1.478)     &(2.465)      &(1.149)       &(1.743)\\ 
                           &         & $10^5$               &0.976       &0.964        &0.964        &0.968\\
                          &           &                        &(1.059)     &(1.812)      &(0.798)       &(1.273)\\  \hline \hline
\end{tabular}
\caption{Empirical coverage probabilities of the MB confidence bands for BCN(1,1/2) and BCL(1,1/2) on $I = [-0.75, -0.25]\cup[0.25,0.75]$, based on 250 Monte Carlo repetitions. 
Inside the parentheses are values of the expected mean width.\label{table: T2}}
\end{center}
\end{table}
}

\begin{figure}[H]
  \begin{center}
    \begin{tabular}{ccc}

      \begin{minipage}{0.33\hsize}
        \begin{center}
          \includegraphics[clip, width=4.5cm]{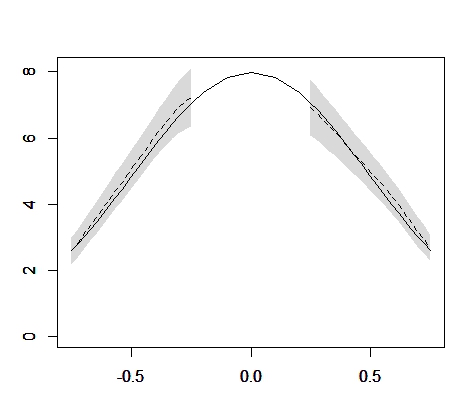}
        \end{center}
      \end{minipage}

      \begin{minipage}{0.33\hsize}
        \begin{center}
          \includegraphics[clip, width=4.5cm]{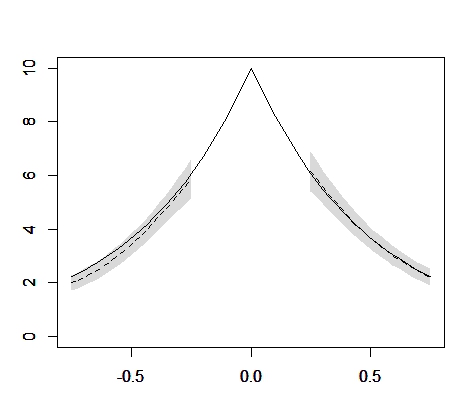}
        \end{center}
      \end{minipage}

      \begin{minipage}{0.33\hsize}
        \begin{center}
          \includegraphics[clip, width=4.5cm]{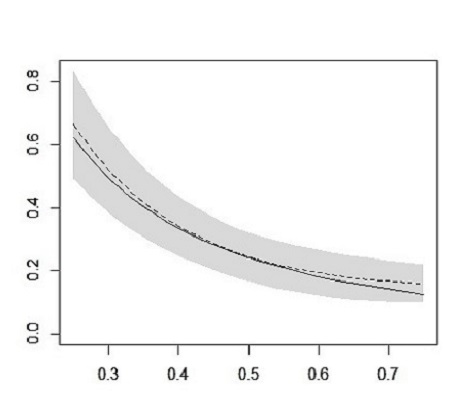}
        \end{center}
      \end{minipage}\\

      \begin{minipage}{0.33\hsize}
        \begin{center}
          \includegraphics[clip, width=4.5cm]{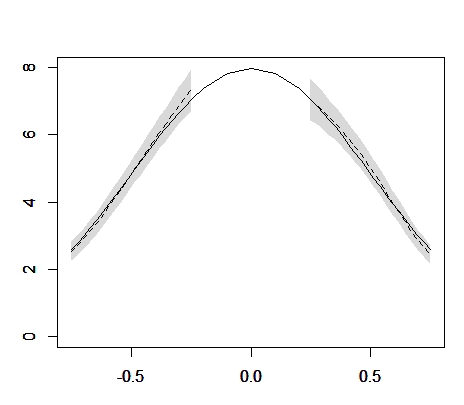}
        \end{center}
      \end{minipage}

      \begin{minipage}{0.33\hsize}
        \begin{center}
          \includegraphics[clip, width=4.5cm]{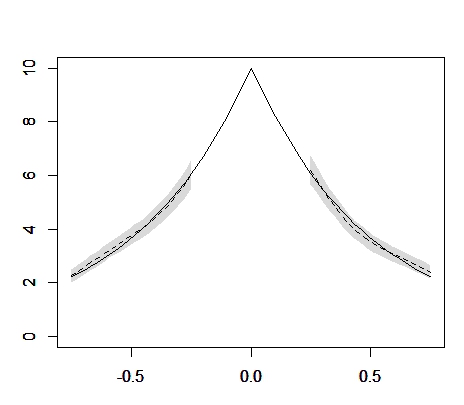}
        \end{center}
      \end{minipage}

      \begin{minipage}{0.33\hsize}
        \begin{center}
          \includegraphics[clip, width=4.5cm]{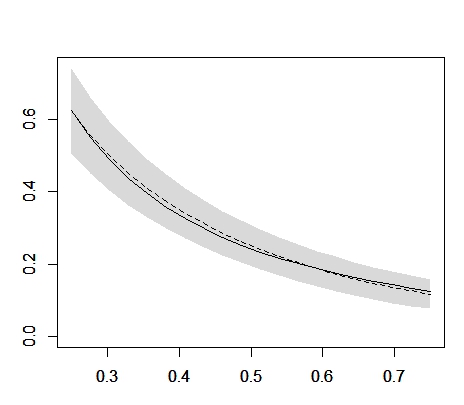}
        \end{center}
      \end{minipage}

    \end{tabular}
    \caption{Estimates of the L\'evy densities (dashed lines) for BCN(0,1/2) (left), BCL(0,1/2) (center), and G(0.2,1) (right), together with $90\%$ MB confidence bands (gray regions). 
The solid lines correspond to the true density functions.  $(n, \Delta) = (5\times 10^4, 0.01)$ (top row) and $(n, \Delta) = (10^5, 0.01)$ (bottom row). \label{fig:A3}}
 \end{center}
\end{figure}

\begin{figure}[H]
  \begin{center}
    \begin{tabular}{cc}

      \begin{minipage}{0.5\hsize}
        \begin{center}
          \includegraphics[clip, width=4.5cm]{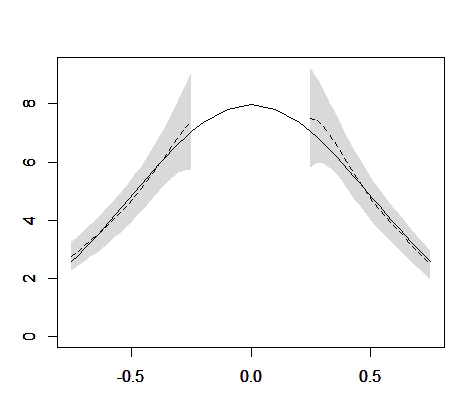}
        \end{center}
      \end{minipage}

      \begin{minipage}{0.5\hsize}
        \begin{center}
          \includegraphics[clip, width=4.5cm]{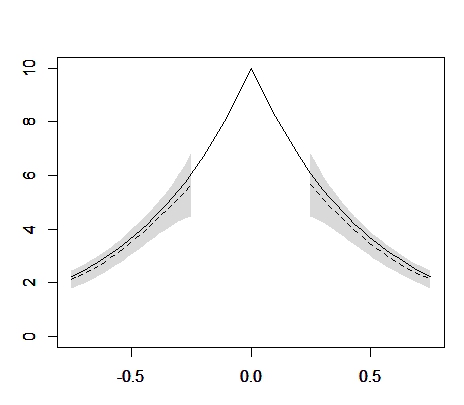}
        \end{center}
      \end{minipage}\\

      \begin{minipage}{0.5\hsize}
        \begin{center}
          \includegraphics[clip, width=4.5cm]{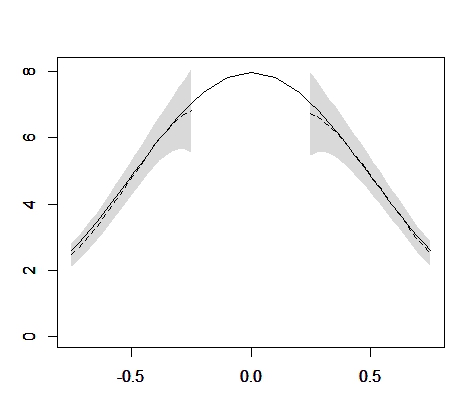}
        \end{center}
      \end{minipage}

      \begin{minipage}{0.5\hsize}
        \begin{center}
          \includegraphics[clip, width=4.5cm]{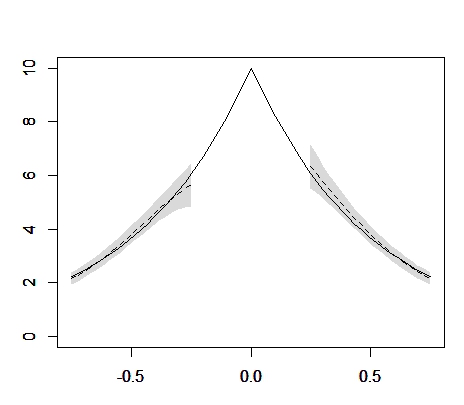}
        \end{center}
      \end{minipage}

    \end{tabular}
    \caption{Estimates of the L\'evy densities (dashed lines) for BCN(1,1/2) (left) and BCL(1,1/2) (right), together with $90\%$ MB confidence bands (gray regions). 
The solid lines correspond to the true density functions.  $(n, \Delta) = (5\times 10^4, 0.01)$ (top row) and $(n, \Delta) = (10^5, 0.01)$ (bottom row).\label{fig:A4}}
  \end{center}
\end{figure}

\section{Conclusion}
\label{sec: conclusion}

In this paper, we have developed bootstrap methods to construct uniform confidence bands for spectral estimators of L\'{e}vy densities from high-frequency observations. We have studied two bootstrap methods, namely, Gaussian multiplier and empirical bootstraps, and established asymptotic validity of the proposed confidence bands. 
Notably, the proposed confidence bands are shown to be valid for a wide class of L\'{e}vy processes.
We have also developed a practical method to choose a bandwidth.

\appendix 

\section{Proofs of Theorems \ref{thm: Gaussian approximation} and \ref{thm: bootstrap}}

In what follows, we always assume Assumption \ref{as: assumption 1}. The proofs rely on modern  empirical process theory. For a probability measure $Q$ on a measurable space $(S,\mS)$ and a  class of measurable functions $\mF$ on $S$ such that $\mF \subset L^{2}(Q)$, let $N(\mF,\| \cdot \|_{Q,2},\varepsilon)$ denote the $\varepsilon$-covering number for $\mF$ with respect to the $L^{2}(Q)$-seminorm  $\| \cdot \|_{Q,2}$.   See Section 2.1 in \cite{vaWe96} for details.
Let $\stackrel{d}{=}$ denote the equality in distribution. 

\subsection{Auxiliary lemmas}

We begin with proving some auxiliary lemmas that will be used to prove Theorems \ref{thm: Gaussian approximation} and \ref{thm: bootstrap}. We will freely use the following moment estimates for $L_{\Delta}=Y_{n,1}$. 

\begin{lemma}
\label{lem: moment}
We have 
\[
\Ep[ L_{\Delta}^{2} ] = \Delta \left (\sigma^{2}  +\int_{\R} x^{2}\rho(x) dx \right ) + \Delta^{2} \gamma_{c}^{2} \lesssim \Delta \quad \text{and} \quad  \Ep[ L_{\Delta}^{4} ] = \Delta \int_{\R} x^{4} \rho(x) dx + o(\Delta) \lesssim \Delta. 
\]
\end{lemma}

\begin{proof}
This follows from the observations that $\Ep[ L_{\Delta}^{2} ] =- \varphi_{\Delta}''(0)$ and $\Ep[ L_{\Delta}^{4} ] = \varphi_{\Delta}^{(4)}(0)$. 
\end{proof}
%

\begin{lemma}
\label{lem: lower bound on chf}
We have $\inf_{|u| \leq h^{-1}} |\varphi_{\Delta}(u)| \gtrsim1$.
\end{lemma}

\begin{proof}
Recall that $\varphi_{\Delta}(u) = e^{\Delta \psi(u)}$. From Taylor's theorem, $|e^{iux} - 1 - iux| \leq \frac{u^{2}x^{2}}{2}$
for any $u,x \in \R$, so that $|\psi (u)| \leq \frac{\sigma^{2}u^{2}}{2} + |\gamma_{c}| |u| + \frac{u^{2}}{2} \int_{\R} x^{2} \rho (x) dx \lesssim h^{-2} \lesssim \Delta^{-1}$
uniformly in $|u| \leq h^{-1}$. Therefore, we conclude that
\[
\inf_{|u| \leq h^{-1}}|\varphi_{\Delta}(u)| \geq e^{-\Delta \sup_{|u| \leq h^{-1}}|\psi (u)|} = e^{-O(1)} \gtrsim 1.
\]
This completes the proof. 
\end{proof}

\begin{lemma}
\label{lem: deconvolution kernel}
We have $\| (1+x^{2}+h^{2}x^{4})( | K_{n} | \vee | K_{n}' |)\|_{\R} \lesssim 1$. 
\end{lemma}

\begin{proof}
From the previous lemma, it is not difficult to verify that $\| K_{n} \|_{\R} \lesssim 1$ and $\| K_{n}' \|_{\R} \lesssim 1$. 
By changes of variables, observe that 
\[
K_{n}(x) = \frac{1}{2\pi x^{2}} \int_{\R} e^{-iux} \left ( \frac{\varphi_{W}(u)}{\varphi_{\Delta}(u/h)} \right )'' du = \frac{-1}{2\pi x^{4}} \int_{\R} e^{-iux} \left ( \frac{\varphi_{W}(u)}{\varphi_{\Delta}(u/h)} \right )^{(4)} du. 
\]
Since $\varphi_{W}$ is supported in $[-1,1]$, to show that $\| (x^{2}+h^{2}x^{4}) K_{n} \|_{\R} \lesssim 1$, it suffices to verify that 
\[
\left \| \left ( \frac{\varphi_{W}(\cdot)}{\varphi_{\Delta}(\cdot/h)} \right )'' \right \|_{[-1,1]} \lesssim 1 \quad \text{and} \quad \left \| \left ( \frac{\varphi_{W}(\cdot)}{\varphi_{\Delta}(\cdot/h)} \right )^{(4)} \right \|_{[-1,1]} \lesssim h^{-2}.
\]
To see this, observe that 
\begin{align*}
\psi'(u) &= -\sigma^{2} u + i\gamma_{c} + \int_{\R} (e^{iux} - 1) x\rho (x) dx, \
\psi''(u) = -\sigma^{2} - \int_{\R} e^{iux} x^{2} \rho (x) dx, \\
\psi'''(u) & = -i \int_{\R} e^{iux} x^{3} \rho(x) dx, \ \psi^{(4)}(u) = \int_{\R} e^{iux} x^{4} \rho(x) dx, \\
\varphi_{\Delta}'(u) &= \Delta \psi'(u) \varphi_{\Delta}(u), \ \varphi_{\Delta}''(u) = \Delta\{ \psi''(u) + \Delta (\psi'(u))^{2} \}\varphi_{\Delta}(u), \\ 
\varphi_{\Delta}'''(u) &= \Delta \{ \psi'''(u) + 3 \Delta \psi''(u) \psi'(u) + \Delta^{2} (\psi'(u))^{3} \} \varphi_{\Delta}(u), \\
\varphi_{\Delta}^{(4)}(u) &= \Delta \left [  \psi^{(4)}(u) + \Delta \{ 4\psi'''(u) \psi'(u) + 3 (\psi''(u))^{2} \} + 6 \Delta^{2} \psi''(u) (\psi'(u))^{2} + \Delta^{3} (\psi'(u))^{4} \right ] \varphi_{\Delta}(u).
\end{align*}
This yields that 
\begin{align*}
&\left \| \frac{\varphi_{\Delta}'}{\varphi_{\Delta}} \right \|_{[-h^{-1},h^{-1}]} \lesssim \Delta h^{-1}, \ \left \| \frac{\varphi_{\Delta}''}{\varphi_{\Delta}} \right \|_{[-h^{-1},h^{-1}]} \lesssim \Delta  \{ 1 + \Delta h^{-2} \} \lesssim \Delta, \\
&\left \| \frac{\varphi_{\Delta}'''}{\varphi_{\Delta}} \right \|_{[-h^{-1},h^{-1}]} \lesssim \Delta \{ 1 + \Delta h^{-1} + \Delta^{2}  h^{-3} \} \lesssim \Delta, \\
& \left \| \frac{\varphi_{\Delta}^{(4)}}{\varphi_{\Delta}}  \right \|_{[-h^{-1},h^{-1}]} \lesssim \Delta \{ 1 + \Delta (h^{-1} + 1) + \Delta^{2} h^{-2} + \Delta^{3} h^{-4}\} \lesssim \Delta, 
\end{align*}
where we have used that $h \gtrsim \Delta^{1/2}$. Next, observe the following identities
\begin{align*}
\left (\frac{f}{g} \right)'' &= \frac{f''}{g} - 2 \frac{f'}{g} \frac{g'}{g}  + \frac{f}{g} \left \{ - \frac{g''}{g}   + 2\left ( \frac{g'}{g} \right)^{2} \right \},\\
\left({f \over g}\right)^{(4)} &= {f^{(4)} \over g} - 4{f''' \over g}{g' \over g} + 6{f'' \over g}\left \{ -{ g'' \over g } + 2\left(g' \over  g\right)^{2}\right \} + 4{f' \over g} \left \{ -{g''' \over g} + 6{g' \over  g}{g'' \over  g} - 6\left(g' \over  g\right)^{3}\right \}\\
&\quad +{f \over g} \left \{ -{g^{(4)} \over g} + 8{g' \over  g}{g''' \over  g} + 6\left (g'' \over  g\right)^{2} -36\left(g' \over  g\right)^{2}{g'' \over  g} + 24\left(g' \over  g\right)^{4} \right \}.
\end{align*}
The second identity follows from the following (straightforward but tedious) calculations:
\begin{align*}
\left ( \frac{f}{g} \right )^{(4)} &={f^{(4)} \over g} + 4f''' \left (1 \over g \right)' + 6f''\left (1 \over g \right)'' + 4f'\left (1 \over g \right)'''+ f\left (1 \over g \right)^{(4)},\\
\left (1 \over g \right)' &= -{g' \over g^{2}} = -{1 \over g}{g' \over g}, \\
\left (1 \over g \right)'' &=  -{g'' \over g^{2}} + 2{(g')^{2} \over g^{3}} = {1 \over g}\left \{-{g'' \over g} + 2\left({g' \over g}\right)^{2}\right \}, \\
\left (1 \over g \right)''' &= -{g''' \over g^{2}} + 6{g'g'' \over g^{3}}- 6{(g')^{3} \over g^{4}} = {1 \over g}\left \{ -{g''' \over g}+ 6{g' \over g}{g'' \over g}- 6\left({g' \over g}\right)^{3}\right \},\\
\left (1 \over g \right)^{(4)} &= -{g^{(4)} \over g^{2}} + 8{g'g''' \over  g^{3}} + 6{(g'')^{2} \over  g^{3}} -36{(g')^{2}g'' \over  g^{4}} + 24{(g')^{4} \over  g^{5}} \\
& = {1 \over g}\left \{ -{g^{(4)} \over g}+ 8{g' \over  g}{g''' \over  g}+ 6\left(g'' \over  g\right)^{2} -36\left(g' \over  g\right)^{2}{g'' \over  g} + 24\left(g' \over  g\right)^{4} \right \}.
\end{align*}
Now, noting that $(\varphi_{\Delta}(u/h))^{(k)} = h^{-k} \varphi_{\Delta}^{(k)}(u/h)$ for $k=1,2,3,4$, 
we conclude that 
\begin{align*}
\left \| \left ( \frac{\varphi_{W}(\cdot)}{\varphi_{\Delta}(\cdot/h)} \right )'' \right \|_{[-1,1]} &\lesssim 1 + \Delta h^{-2}+ \{ \Delta h^{-2} + \Delta^{2} h^{-4} \} \lesssim 1, \\
\left \| \left ( \frac{\varphi_{W}(\cdot)}{\varphi_{\Delta}(\cdot/h)} \right )^{(4)} \right \|_{[-1,1]} &\lesssim 1 + \Delta h^{-2} + \{ \Delta h^{-2} + \Delta^{2} h^{-4} \} + \{ \Delta h^{-3} + \Delta^{2} h^{-4} + \Delta^{3} h^{-6} \} \\
&\quad + \{ \Delta h^{-4} + \Delta^{2} h^{-5} + \Delta^{2} h^{-4} + \Delta^{3} h^{-6} + \Delta^{4} h^{-8} \} \\
&\lesssim 1+ \Delta h^{-4} \lesssim h^{-2}.
\end{align*}
Likewise, we have that $\| (x^{2}+h^{2}x^{4}) K_{n}' \|_{\R} \lesssim 1$. This completes the proof. 
\end{proof}

\begin{lemma}
\label{lem: uniform convergence}
For $k=0,1,2$, we have that $\|(\hat{\varphi}_{\Delta} - \varphi_{\Delta})^{(k)}\|_{[-h^{-1},h^{-1}]} = O_{\Pr}\left \{ n^{-1/2}\Delta^{(k\wedge 1)/2}\log h^{-1} \right \}$.
\end{lemma}

\begin{proof}
This follows from Theorem 1 in \cite{KaRe10}, which shows that for the weight function $w(u) = (\log (e+|u|))^{-1}$,
\[
C_{k}:=\sup_{n} \Ep[ \| \sqrt{n} \Delta^{-(k \wedge 1)/2}(\hat{\varphi}_{\Delta} - \varphi_{\Delta})^{(k)} w \|_{\R}] < \infty
\]
for $k=0,1,2$ under our assumption. Since 
\[
\| \sqrt{n}\Delta^{-(k \wedge 1)/2}(\hat{\varphi}_{\Delta} - \varphi_{\Delta})^{(k)} w \|_{\R} \geq \sqrt{n}\Delta^{-(k \wedge 1)/2} \| (\hat{\varphi}_{\Delta} - \varphi_{\Delta})^{(k)} \|_{[-h^{-1},h^{-1}]} \inf_{|u| \leq h^{-1}} w(u),
\]
we conclude that 
\[
\Ep [\| (\hat{\varphi}_{\Delta} - \varphi_{\Delta})^{(k)} \|_{[-h^{-1},h^{-1}]}] \leq \frac{C_{k}\Delta^{(k \wedge 1)/2}}{\sqrt{n} \inf_{|u| \leq h^{-1}}w(u)} \lesssim n^{-1/2}\Delta^{(k \wedge 1)/2} \log h^{-1}.
\]
The desired result follows from Markov's inequality. 
\end{proof}

 Lemmas  \ref{lem: lower bound on chf} and \ref{lem: uniform convergence} imply that
\[
\inf_{|u| \leq h^{-1}} |\hat{\varphi}_{\Delta} (u)| \geq \inf_{|u| \leq h^{-1}} | \varphi_{\Delta}(u) | - o_{\Pr}(1) \gtrsim 1 - o_{\Pr}(1),
\]
so that with probability approaching one, $\inf_{|u| \leq h^{-1}} | \hat{\varphi}_{\Delta}(u)| > 0$. Hence, with probability approaching one, $\hat{\psi} := \Delta^{-1}\log \hat{\varphi}_{\Delta}$ is well-defined on $[-h^{-1},h^{-1}]$ as the distinguished logarithm \citep[][Theorem 7.6.2]{Ch01}. 

\begin{lemma}
\label{lem: linearization}
We have
\[
\left \|(\hat{\psi} - \psi)'' - \frac{(\hat{\varphi}_{\Delta} - \varphi_{\Delta})''}{\Delta \varphi_{\Delta}} \right \|_{[-h^{-1},h^{-1}]} =o_{\Pr} \{ h (n\Delta h \log n)^{-1/2} \}.
\]
\end{lemma}

\begin{proof}
The lemma essentially follows from the proof of Proposition 7 in \cite{NiReSoTr16}. For the sake of completeness, we provide a proof of the lemma. 
Rewrite $(\hat{\psi} - \psi)''$ as $ (\hat{\psi} - \psi)''=\Delta^{-1}(\log(\hat{\varphi}_{\Delta}/\varphi_{\Delta}))''$. 
Let $F(y) = \log(1+y), \eta = (\hat{\varphi}_{\Delta} - \varphi_{\Delta})/\varphi_{\Delta}$, and observe that for any $|u| \leq h^{-1}$, 
\[
(F \circ \eta)''(u) = F'(\eta(u))\eta''(u) + F''(\eta(u))\eta'(u)^{2} = \eta''(u) + F''(\theta \eta(u))\eta(u)\eta''(u) + F''(\eta(u))\eta'(u)^{2}
\]
for some $\theta \in [0,1]$. Since $F''$ is bounded in a neighborhood of the origin and $\| \eta  \|_{[-h^{-1},h^{-1}]} = o_{\Pr}(1)$ (which follows from Lemmas  \ref{lem: lower bound on chf} and \ref{lem: uniform convergence}), we have that 
\begin{align*}
&\left \|(\log(\hat{\varphi}_{\Delta}/\varphi_{\Delta}))'' - \left ( \frac{\hat{\varphi}_{\Delta} - \varphi_{\Delta}}{\varphi_{\Delta}} \right )'' \right \|_{[-h^{-1},h^{-1}]} \\
&\quad = O_{\Pr }\left ( \| \eta \|_{[-h^{-1}, h^{-1}]}\| \eta'' \|_{[-h^{-1}, h^{-1}]} + \| \eta' \|^{2}_{[-h^{-1}, h^{-1}]}\right).
\end{align*}
Next, we shall bound $\| \eta^{(k)}\|_{[-h^{-1}, h^{-1}]}$ for $k=0,1,2$. Since $\| \varphi'_{\Delta}/\varphi_{\Delta} \|_{[-h^{-1},h^{-1}]} \lesssim \Delta h^{-1} \lesssim \Delta^{1/2}$ and $\| \varphi_{\Delta}''/\varphi_{\Delta} \|_{[-h^{-1},h^{-1}]}\lesssim \Delta$, we have that 
\begin{align}
&\left \| \left ( \frac{1}{\varphi_{\Delta}} \right )' \right \|_{[-h^{-1},h^{-1}]} = \left \| \frac{\varphi'_{\Delta}}{\varphi_{\Delta}^{2}} \right \|_{[-h^{-1},h^{-1}]} \lesssim \Delta^{1/2}, \label{eq: derivative 1} \\ 
&\left \| \left ( \frac{1}{\varphi_{\Delta}} \right )'' \right \|_{[-h^{-1},h^{-1}]} =\left \|{\varphi''_{\Delta}\varphi_{\Delta} - 2(\varphi'_{\Delta})^{2} \over \varphi_{\Delta}^{3}} \right\|_{[-h^{-1},h^{-1}]} \lesssim \Delta. \label{eq: derivative 2}
\end{align}
In view of the identities 
\begin{align*}
&\eta' =\left ( \frac{1}{\varphi_{\Delta}} \right )'(\hat{\varphi}_{\Delta} - \varphi_{\Delta}) +\left ( \frac{1}{\varphi_{\Delta}} \right )(\hat{\varphi}_{\Delta} - \varphi_{\Delta})', \\
&\eta'' = \left ( \frac{1}{\varphi_{\Delta}} \right )''(\hat{\varphi}_{\Delta} - \varphi_{\Delta}) + 2\left ( \frac{1}{\varphi_{\Delta}} \right )'(\hat{\varphi}_{\Delta} - \varphi_{\Delta})' + \left ( \frac{1}{\varphi_{\Delta}} \right )(\hat{\varphi}_{\Delta} - \varphi_{\Delta})'', 
\end{align*}
we conclude from Lemma \ref{lem: uniform convergence} that $\| \eta^{(k)} \|_{[-h^{-1},h^{-1}]} = O_{\Pr} (n^{-1/2}\Delta^{(k \wedge 1)/2}\log n )$
for $k=0,1,2$, which yields that 
\[
 \| \eta \|_{[-h^{-1}, h^{-1}]}\| \eta''\|_{[-h^{-1}, h^{-1}]} \bigvee \| \eta' \|^{2}_{[-h^{-1}, h^{-1}]} = O_{\Pr}\{ n^{-1}\Delta^{1/2}(\log n)^{2} \}. 
\]

Finally, observe that 
\[
\left ( \frac{1}{\varphi_{\Delta}} \right ) (\hat{\varphi}_{\Delta} - \varphi_{\Delta})''=\eta'' -\left ( \frac{1}{\varphi_{\Delta}} \right ) ''(\hat{\varphi}_{\Delta} - \varphi_{\Delta}) - 2\left ( \frac{1}{\varphi_{\Delta}} \right ) '(\hat{\varphi}_{\Delta} - \varphi_{\Delta})'.
\]
From the bounds (\ref{eq: derivative 1}) and (\ref{eq: derivative 2}), together with Lemma \ref{lem: uniform convergence}, we have that
\[
\left \| \left ( \frac{1}{\varphi_{\Delta}} \right ) ''(\hat{\varphi}_{\Delta} - \varphi_{\Delta}) \right \|_{[-h^{-1},h^{-1}]} \bigvee \left \| \left ( \frac{1}{\varphi_{\Delta}} \right )'(\hat{\varphi}_{\Delta} - \varphi_{\Delta})' \right \|_{[-h^{-1},h^{-1}]}= O_{\Pr}(n^{-1/2}\Delta \log n ).
\]
Taking these together, we conclude that 
\begin{equation}
\left \|(\log(\hat{\varphi}_{\Delta}/\varphi_{\Delta}))''- \frac{(\hat{\varphi}_{\Delta} - \varphi_{\Delta})''}{\varphi_{\Delta}} \right \|_{[-h^{-1},h^{-1}]} =O_{\Pr} \{ n^{-1/2}\Delta \log n + n^{-1}\Delta^{1/2} (\log n)^{2} \}. 
\label{eq: linearization}
\end{equation}
We shall verify that the right hand side is $o_{\Pr} \{ \Delta h (n\Delta h \log n)^{-1/2} \}$. Observe that 
\[
\frac{n^{-1/2}\Delta \log n + n^{-1}\Delta^{1/2} (\log n)^{2}}{\Delta h (n\Delta h \log n)^{-1/2}} = \Delta^{1/2} h^{-1/2} (\log n)^{3/2} + n^{-1/2}  h^{-1/2} (\log n)^{5/2}. 
\]
Since $\sqrt{n h} \gg n^{\delta} \Delta^{-1/2}$, we have that $n^{-1/2} h^{-1/2} (\log n)^{5/2} \ll n^{-\delta}\Delta^{1/2} (\log n)^{5/2} \ll 1$. On the other hand, as $h \gtrsim \Delta^{1/2}$, $\Delta^{1/2} h^{-1/2} (\log n)^{3/2} \lesssim \Delta^{1/4} (\log n)^{3/2}$. 
Since $1 \gg h^{r} \sqrt{n \Delta h \log n} \gtrsim \Delta^{r/2 + 3/4} \sqrt{n \log n}$, we have that $\Delta \ll (n \log n)^{-1/(r+3/2)}$, which implies that $\Delta^{1/4} (\log n)^{3/2} \ll 1$. 
This completes the proof. 
\end{proof}

\begin{lemma}
\label{lem: variance lower bound}
Recall that $s_{n}^{2}(x) = \Var (Y_{n,1}^{2} K_{n}((x-Y_{n,1})/h))$. Then $\inf_{x \in I} s_{n}^{2}(x) \gtrsim \Delta h$
for sufficiently large $n$. 
\end{lemma}
\begin{proof}
Since $\| K_{n} \|_{\R} \lesssim 1$, we have that $\left ( \Ep[Y_{n,1}^{2}K_{n}((x-Y_{n,1})/h)] \right)^{2} \leq \| K_{n} \|_{\R}^{2} (\Ep[Y_{n,1}^{2}])^{2} \lesssim \Delta^{2} \ll \Delta h$.
Next, observe that  $(x+y)^{4} \geq x^{4}/16 - y^{4}$ for any $x,y \in \R$. Using this inequality and recalling  that $Y_{n,1} - b\Delta$ has distribution $P_{\Delta,b}$ with $P_{\Delta,b}(dy) = g_{\Delta,b} (y)dy$, we have that
\begin{align*}
&\Ep[Y_{n,1}^{4}K_{n}^{2}((x-Y_{n,1})/h)] = \Ep [ (Y_{n,1}-b\Delta + b\Delta)^4 K_{n}^{2}((x-(Y_{n,1}-b\Delta) - b\Delta)/h)] \\
&\quad \geq \frac{1}{16} \int_{\R} K_{n}^{2}((x-b\Delta-y)/h) g_{\Delta,b}(y) dy - b^{4}\Delta^{4}\| K_{n} \|_{\R}^{2} \\
&\quad = \frac{h}{16} \int_{\R}K^{2}_{n}(y)g_{\Delta,b} (x-b\Delta - yh) dy- b^{4}\Delta^{4}\| K_{n} \|_{\R}^{2}.
\end{align*}
Since $\Delta^{4} \ll \Delta h$,  it suffices to verify that 
\[
\inf_{x \in I^{|b|\Delta}} \int_{\R}K^{2}_{n}(y)g_{\Delta,b} (x-yh) dy \gtrsim \Delta
\]
for sufficiently large $n$. 
Since $|\varphi_{\Delta}| \leq 1$ and by Plancherel's theorem, we have that 
\[
\int_{\R} K_{n}^{2}(y) dy = \frac{1}{2\pi} \int_{\R} \left | \frac{\varphi_{W}(u)}{\varphi_{\Delta}(u/h)} \right |^{2} du
\geq \frac{1}{2\pi} \int_{\R} |\varphi_{W}(u)|^{2} du =: \underline{c}.
\]
From Lemma \ref{lem: deconvolution kernel}, we have that $\| (1+y^{2}) K_{n} \|_{\R} \lesssim 1$, so that for $m > 0$, 
\[
\int_{[-m,m]^{c}} K_{n}^{2}(y) dy \lesssim \int_{[-m,m]^{c}} y^{-4} dy
\]
up to a constant independent of $(n,m)$. The right hand side is approaching zero as $m \to \infty$, so that by taking $m$ sufficiently large, we have that 
\[
\int_{-m}^{m} K_{n}^{2}(y) dy \geq \int_{\R} K_{n}^{2}(y) dy -\frac{\underline{c}}{2} \geq \frac{\underline{c}}{2}
\]
for all $n$. Hence, for sufficiently large $n$ such that $|b|\Delta + m h \leq \varepsilon_{0}$, 
\[
\inf_{x \in I^{|b|\Delta}} \int_{\R} K_{n}^{2}(y) g_{\Delta,b}(x-yh) dy \geq \inf_{x \in I^{|b|\Delta}} \int_{-m}^{m} K_{n}^{2}(y) g_{\Delta,b}(x-yh) dy \geq  \frac{\underline{c}}{2} \inf_{x \in I^{\varepsilon_{0}}} g_{\Delta,b}(x) \gtrsim \Delta
\]
by Condition (ii) in Assumption \ref{as: assumption 1}. 
This completes the proof. 
\end{proof} 

Consider the function class
\[
\mF_{n} =\left \{ y \mapsto \frac{y^{2}}{s_{n}(x)} K_{n}\left ( \frac{x-y}{h} \right ): x \in I \right \}.
\]
Observe that 
\[
y^{2}K_{n}\left ( \frac{x-y}{h} \right ) =  h^{2} \left ( \frac{x-y}{h} \right )^{2} K_{n}\left ( \frac{x-y}{h} \right ) -2x h   \left ( \frac{x-y}{h} \right ) K_{n}\left ( \frac{x-y}{h} \right ) + x^{2} K_{n}\left ( \frac{x-y}{h} \right ). 
\]
Since $\| (1+y^{2}) K_{n} \|_{\R} \lesssim 1$ by Lemma \ref{lem: deconvolution kernel} and $I$ is compact, each of the three terms on the right hand side is bounded (as a function of $y$) uniformly in $n$ and $x \in I$. 
Choose  constants $D_{1},D_{2} > 0$ independent of $n$ such that $\| (1+y^{2}) K_{n}\|_{\R} \leq D_{1}$ and $\| 1/s_{n} \|_{I} \leq D_{2}/\sqrt{\Delta h}$ (cf. Lemma \ref{lem: variance lower bound}). Without loss of generality, we may assume that $h \leq 1$. Then functions in $\mF_{n}$ are bounded by
\[
(1+2\| x \|_{I} + \| x^{2} \|_{I}) D_{1}D_{2}/\sqrt{\Delta h} \lesssim 1/\sqrt{\Delta h}. 
\]
The next lemma provides a bound on the uniform covering number for the function class $\mF_{n}$. 
\begin{lemma}
\label{lem: VC type}
There exist constants $A, v > 0$ independent of $n$ such that 
\begin{equation}
\sup_{Q} N(\mF_{n},\| \cdot \|_{Q,2}, \varepsilon/\sqrt{\Delta h} ) \leq (A/\varepsilon)^{v}, \ 0 < \forall \varepsilon \leq 1, 
\label{eq: VC type}
\end{equation}
where $\sup_{Q}$ is taken over all finitely discrete distributions on $\R$.
\end{lemma}

The proof of this lemma relies on the following lemma.

\begin{lemma}[\cite{GiNi16}, Lemma 3.2.16]
\label{lem: BV}
Let $f: \R \to \R$ be a function of bounded variation, i.e., 
\[
\mathsf{TV}(f) :=\sup \left \{ \sum_{j=1}^{N} | f(x_{j}) - f(x_{j-1}) | : -\infty < x_{0} < \cdots < x_{N} <\infty, N=1,2,\dots \right \} < \infty,
\]
and consider the function class $\mF = \{ x \mapsto f(ax+b) : a, b \in \R \}$. Then there exist universal constants $A,v > 0$ such that 
\[
\sup_{Q} N(\mF, \| \cdot \|_{Q,2}, \varepsilon \mathsf{TV} (f)) \leq \left ( \frac{A}{\varepsilon} \right )^{v}, \ 0 < \forall \varepsilon \leq 1,
\]
where $\sup_{Q}$ is taken over all finitely discrete distributions on $\R$.
\end{lemma}

\begin{proof}[Proof of Lemma \ref{lem: VC type}]
Consider the auxiliary function classes
\[
\mathcal{G}_{n,\ell} = \left \{ y \mapsto h^{\ell} (ay + b)^{\ell}  K_{n}(ay + b) :  a,b \in \R \right \}, \ \ell=0,1,2. 
\]
Since $\mF_{n} \subset \{ \frac{g_{2} -2x g_{1} + x^{2} g_{0}}{s_{n}(x)} : g_{\ell} \in \mG_{n,\ell}, \ell=0,1,2; x \in I \}$, $I$ is compact, and $\| 1/s_{n} \|_{I} \lesssim 1/\sqrt{\Delta h}$,  the desired conclusion follows by verifying that there exist constants $A_{1},v_{1} > 0$ independent of $n$ such that 
\[
\sup_{Q} N(\mG_{n,\ell},\| \cdot \|_{Q,2}, \varepsilon ) \leq \left ( \frac{A_{1}}{\varepsilon} \right )^{v_{1}},  \ 0 < \forall \varepsilon \leq 1
\]
for $\ell=0,1,2$. To this end, in view of Lemma \ref{lem: BV}, it suffices to verify that, for $K_{n,\ell} (y) := y^{\ell} h^{\ell} K_{n}(y), \ell=0,1,2$, the total variations of $K_{n,\ell}$ are bounded in $n$, i.e., 
\[
\mathsf{TV} (K_{n,\ell}) = \int_{\R} | K_{n,\ell}'(y) | dy  \lesssim 1.
\]
This follows from observations that $K_{n,1}'(y) = h (K_{n}(y) + y K_{n}'(y)),  K_{n,2}'(y) = h^{2} (2y K_{n}(y) + y^{2}K_{n}'(y))$, 
and  $\| (1+y^{2} + h^{2}y^{4}) (|K_{n}| \vee |K_{n}'|) \|_{\R} \lesssim 1$ by Lemma \ref{lem: deconvolution kernel}.
\end{proof}

\begin{lemma}
\label{lem: bias}
Let $\rho_{\sharp}(x) = x^{2}\rho(x)$. Then $\| \rho_{\sharp}*(h^{-1}W(\cdot/h)) - \rho_{\sharp}  \|_{\R} \lesssim h^{r}$,
where $*$ denotes the convolution. 
\end{lemma}

\begin{proof}
Observe that by a change of variables, $[\rho_{\sharp}*(h^{-1}W(\cdot/h))] (x) - \rho_{\sharp}(x) = \int_{\R} \{ \rho_{\sharp}(x-yh) - \rho_{\sharp}(x) \} W(y) dy$.
If $p \geq 1$, then
by Taylor's theorem, for any $x,y \in \R$, 
\[
\rho_{\sharp}(x-yh) - \rho_{\sharp}(x) 
=
 \sum_{\ell=1}^{p-1} \frac{\rho_{\sharp}^{(\ell)}(x)}{\ell !} (-yh)^{\ell} + \frac{\rho_{\sharp}^{(p)}(x-\theta yh)}{p!} (-yh)^{p}
\]
for some $\theta \in [0,1]$, where $\sum_{\ell=1}^{0} = 0$ by convention.
Since $\rho_{\sharp}^{(p)}$ is $(r-p)$-H\"{o}lder continuous, we have that $H:= \sup_{x,y \in \R, x \neq y} \frac{|\rho_{\sharp}^{(p)} (x) - \rho_{\sharp}^{(p)}(y)|}{|x-y|^{r-p}} < \infty$.
Now, since $\int_{\R} y^{\ell} W(y) dy = 0$ for $\ell=1,\dots,p$, we have that for any $x \in \R$, 
\begin{align*}
&\left | \int_{\R} \{ \rho_{\sharp}(x-yh) - \rho_{\sharp}(x)  \} W(y) dy \right | =\left | \int_{\R} \left [ \{ \rho_{\sharp}(x-yh) - \rho_{\sharp}(x)  \} - \sum_{\ell=1}^{p} \frac{\rho_{\sharp}^{(\ell)}(x)}{\ell !} (-yh)^{\ell} \right ] W(y) dy \right | \\
&\quad \leq \frac{Hh^{r}}{p!} \int_{\R} |y|^{r} |W(y)| dy,
\end{align*}
where $0!=1$ by convention. 
This completes the proof. 
\end{proof}

\subsection{Proof of Theorem \ref{thm: Gaussian approximation}}

Observe that 
\begin{align}
x^{2}  (\hat{\rho}(x) -  \rho(x))&={1 \over 2\pi}\int_{\R}e^{-iux}\left ( -\psi''(u) - \sigma^{2} \right ) \varphi_{W}(uh)du - x^{2}\rho(x) \notag \\
&\quad + {1 \over 2\pi}\int_{\R}e^{-iux} (\psi(u) - \hat{\psi}(u))''\varphi_{W}(uh)du  + (\sigma^{2} - \hat{\sigma}^{2}){1 \over 2\pi}\int_{\R}e^{-iux}\varphi_{W}(uh)du \notag \\
&= \underbrace{{1 \over 2\pi}\int_{\R}e^{-iux}\left(\int_{\R}e^{iuy}y^{2}\rho(y)dy\right)\varphi_{W}(uh)du}_{=[(y^{2}\rho) * (h^{-1}W(\cdot/h))](x)}  -x^{2}\rho(x) \notag \\ 
&\quad + {1 \over 2\pi}\int_{\R}e^{-iux} (\psi(u) - \hat{\psi}(u))'' \varphi_{W}(uh)du \notag \\
&\quad  + (\sigma^{2} - \hat{\sigma}^{2})h^{-1}W(x/h) \notag \\
&=:I_{1,n} + II_{2,n} + III_{3,n}. \label{eq: decomposition}
\end{align}
For the first and third terms, we have that $\| I_{1,n} \|_{I} \lesssim h^{r} \ll  (n\Delta h\log n)^{-1/2}$ (by Lemma \ref{lem: bias}) and $\| III_{3,n} \|_{I} \ll (n\Delta h \log n)^{-1/2}$ under our assumption. For the second term $II_{2,n}$, Lemma \ref{lem: linearization} yields that
\begin{align*}
II_{2,n} &= \frac{-1}{2\pi \Delta} \int_{\R} e^{-iux} \frac{(\hat{\varphi}_{\Delta} - \varphi_{\Delta})''}{\varphi_{\Delta}} (u) \varphi_{W}(uh) du +o_{\Pr} \{  (n\Delta h \log n)^{-1/2} \}
\end{align*}
uniformly in $x \in I$, and observe that the first term on the right hand side can be expressed as 
\[
\frac{1}{n \Delta h} \sum_{j=1}^{n} \{ Y_{n,j}^{2}K_{n}((x-Y_{n,j})/h) - \Ep[Y_{n,1}^{2}K_{n}((x-Y_{n,1})/h) ] \}.
\]
Therefore, since $\inf_{x \in I} s_{n}(x) \gtrsim \sqrt{\Delta h}$, we have that 
\begin{equation}
\frac{\sqrt{n} \Delta h x^{2}(\hat{\rho}(x) - \rho(x))}{s_{n}(x)} = Z_{n}(x) + o_{\Pr} \{ (\log n)^{-1/2} \}  \label{eq: linearization}
\end{equation}
uniformly in $x \in I$. 

Now, we approximate $\| Z_{n} \|_{I}$ by the supremum in absolute value of a tight Gaussian random variable in $\ell^{\infty}(I)$ with mean zero and the same covariance function as $Z_{n}$. To this end, we shall employ Theorem 2.1 in \cite{ChChKa16}. Consider the empirical process
\[
\mathbb{G}_{n}(f) = \frac{1}{\sqrt{n}} \sum_{j=1}^{n}\{ f(Y_{n,j}) - \Ep[f(Y_{n,1})] \}, \ f \in \mF_{n}.
\]
The covering number bound (\ref{eq: VC type})  ensures the existence of a tight Gaussian random variable $\mathbb{U}_{n}$ in $\ell^{\infty} (\mF_{n})$ with mean zero and the same covariance function as $\{ \mathbb{G}_{n}(f) : f \in \mF_{n} \}$. Extend $\mathbb{G}_{n}$ linearly  to $\mF_{n}\cup (-\mF_{n}) = \{ f,-f : f \in \mF_{n} \}$, and observe that $\| \mathbb{G}_{n} \|_{\mF_{n}} = \sup_{f \in \mF_{n} \cup (-\mF_{n})} \mathbb{G}_{n}(f)$.
Note that from Theorem 3.7.28 in \cite{GiNi16}, $\mathbb{U}_{n}$ extends to the linear hull of $\mF_{n}$ in such a way that $\mathbb{U}_{n}$ has linear sample paths, so that $\| \mathbb{U}_{n} \|_{\mF_{n}} = \sup_{f \in \mF_{n} \cup (-\mF_{n})} \mathbb{U}_{n}(f)$, and in addition  $\mathbb{U}_{n}$  has uniformly continuous paths on the symmetric convex hull of $\mF_{n}$. It is not difficult to verify that the covering number of $\mF_{n} \cup (-\mF_{n})$ is at most twice that of $\mF_{n}$. In particular, $\{ \mathbb{U}_{n}(f) : f \in \mF_{n} \cup (-\mF_{n}) \}$ is a tight Gaussian random variable in $\ell^{\infty}(\mF_{n} \cup (-\mF_{n}))$ with mean zero and the same covariance function as $\{ \mathbb{G}_{n}(f) : f \in \mF_{n} \cup (-\mF_{n}) \}$. 

Next, since $Y_{n,1} - b\Delta$ has distribution $P_{\Delta,b}$ such that $y^{4}P_{\Delta,b}(dy) = g_{\Delta,b}(y)dy$ and $\| g_{\Delta,b} \|_{\R} \lesssim \Delta$,  we have that
\begin{align}
&\Ep[ Y_{n,1}^{4} K_{n}^{2}((x-Y_{n,1})/h) ] \leq 8\Ep[ (Y_{n,1}-b\Delta)^{4}  K_{n}^{2}((x-(Y_{n,1}-b\Delta)-b\Delta)/h) ] + 8b^{4} \Delta^{4}\| K_{n} \|_{\R}^{2} \notag \\
&\quad =8\int_{\R} K_{n}^{2}((x-b\Delta - y)/h)g_{\Delta,b}(y) dy + 8b^{4} \Delta^{4}\| K_{n} \|_{\R}^{2} \notag \\
&\quad= 8h \int_{\R} K_{n}^{2}(y) g_{\Delta,b}(x-b\Delta - yh) dy + 8b^{4} \Delta^{4}\| K_{n} \|_{\R}^{2} \notag \\
&\quad \leq  8h\| g_{\Delta,b} \|_{\R} \int_{\R} K_{n}^{2}(y) dy + 8b^{4} \Delta^{4}\| K_{n} \|_{\R}^{2}  \lesssim \Delta h \int_{\R} \left | \frac{\varphi_{W}(u)}{\varphi_{\Delta}(u/h)} \right |^{2} du + \Delta^{4}  \lesssim \Delta h, \label{eq: variance upper bound}
\end{align} 
so that $\sup_{f \in \mF_{n}} \Ep[f^{2}(Y_{n,1})] \lesssim 1$. 
On the other hand, since $\sup_{f \in \mF_{n}} \| f \|_{\R}  \lesssim 1/\sqrt{\Delta h}$, 
\begin{align*}
&\sup_{f \in \mF_{n}} \Ep[|f(Y_{n,1})|^{3}] \lesssim \sup_{f \in \mF_{n}} \Ep[f^{2}(Y_{n,1})]/\sqrt{\Delta h} \lesssim 1/\sqrt{\Delta h} \quad \text{and}\\
&\sup_{f \in \mF_{n}} \Ep[f^{4}(Y_{n,1})] \lesssim \sup_{f \in \mF_{n}} \Ep[f^{2}(Y_{n,1})]/(\Delta h) \lesssim 1/(\Delta h). 
\end{align*}
Therefore, applying Theorem 2.1 in \cite{ChChKa16} to $\mF_{n} \cup (-\mF_{n})$ with $B(f) \equiv 0, A \lesssim 1, v \lesssim 1, \sigma \sim 1, b \lesssim 1/\sqrt{\Delta h}, \gamma \lesssim 1/\log n$, and sufficiently large $q$ (in the notation used in the cited theorem), yields that there exists a random variable $V_{n}$ with $V_{n} \stackrel{d}{=} \| \mathbb{U}_{n} \|_{\mF_{n}}$ such that 
\begin{equation}
\left | \| \mathbb{G}_{n} \|_{\mF_{n}} - V_{n}\right | = O_{\Pr} \left \{  \frac{(\log n)^{1+1/q}}{n^{1/2-1/q} \sqrt{\Delta h}} + \frac{\log n}{(n \Delta h)^{1/6}} \right \}  = o_{\Pr} \{ (\log n)^{-1/2} \}.  
\label{eq: coupling}
\end{equation}

Now, for $f_{n,x}(y)=y^{2}K_{n}((x-y)/h)/s_{n}(x)$, define 
\[
Z_{n}^{G} (x) = \mathbb{U}_{n}(f_{n,x}), \ x \in I,
\]
and observe that $Z_{n}^{G}$ is a tight Gaussian random variable in $\ell^{\infty}(I)$ with mean zero and the same covariance function as $Z_{n}$. We will derive the conclusion of the theorem from (\ref{eq: coupling}). To this end, the following anti-concentration inequality will play a crucial role: for any $\varepsilon> 0$, 
\begin{equation}
\sup_{z \in \R} \Pr \{ | \| Z_{n}^{G} \|_{I} - z | \leq \varepsilon \} \leq 4 \varepsilon (1+\Ep[ \| Z_{n}^{G} \|_{I}]).
\label{eq: AC}
\end{equation}
See Corollary 2.1 in \cite{ChChKa14b}; see also Theorem 3 in \cite{ChChKa15}. 
From the result (\ref{eq: coupling}), there exits a sequence of constants $\varepsilon_{n} \downarrow 0$ such that $\Pr \{ | \| \mathbb{G}_{n} \|_{\mF_{n}} - V_{n} | > \varepsilon_{n} (\log n)^{-1/2} \} \leq \varepsilon_{n}$. 
Since $\| \mathbb{G}_{n} \|_{\mF_{n}} = \| Z_{n} \|_{I}$ and $V_{n} \stackrel{d}{=} \| Z_{n}^{G} \|_{I}$, we have that 
\begin{align*}
\Pr \{ \| Z_{n} \|_{I} \leq z \} &\leq \Pr \{ \| Z_{n}^{G} \|_{I} \leq z + \varepsilon_{n} (\log n)^{-1/2} \} + \varepsilon_{n} \\
&\leq \Pr \{ \| Z_{n} ^{G} \|_{I} \leq z\} + 4\varepsilon_{n}  (\log n)^{-1/2} (1+\Ep[ \| Z_{n}^{G} \|_{I}]) + \varepsilon_{n}
\end{align*}
for any $z \in \R$. Likewise, we have 
\[
\Pr \{ \| Z_{n} \|_{I} \leq z \} \geq \Pr \{ \| Z_{n} ^{G} \|_{I} \leq z\} - 4\varepsilon_{n}  (\log n)^{-1/2} (1+\Ep[ \| Z_{n}^{G} \|_{I}]) - \varepsilon_{n}
\]
for any $z \in \R$. Now, from the covering number bound (\ref{eq: VC type}) together with the fact that $\Var (f_{n,x}(Y_{n,1}))=1$ for all $x \in I$, Dudley's entropy integral bound \citep[cf.][Corollary 2.2.8]{vaWe96} yields that
\[
\Ep[ \| Z_{n}^{G} \|_{I}] = \Ep[\| \mathbb{U}_{n} \|_{\mF_{n}}] \lesssim \int_{0}^{1} \sqrt{1+\log(1/(\varepsilon \sqrt{\Delta h}))} d\varepsilon \lesssim \sqrt{\log n}.
\]
This completes the proof. 

\qed

\subsection{Proof of Theorem \ref{thm: bootstrap}}
We first prove the following technical lemma.

\begin{lemma}
\label{lem: scaling estimate}
$\| \hat{s}_{n}^{2}/s_{n}^{2} -1 \|_{I} = o_{\Pr}\{ (\log n)^{-1} \}$.
\end{lemma}

\begin{proof}
From Lemma \ref{lem: uniform convergence}, we have that $\| \hat{\varphi}_{\Delta} - \varphi_{\Delta} \|_{[-h^{-1},h^{-1}]} = O_{\Pr}(n^{-1/2} \log n)$, so that it is not difficult to verify that $\| \hat{K}_{n} - K_{n} \|_{\R} = O_{\Pr}(n^{-1/2} \log n)$ and $\| \hat{K}_{n}^{2} - K_{n}^{2} \|_{\R} \leq \| \hat{K}_{n} +K_{n} \|_{\R} \| \hat{K}_{n} - K_{n} \|_{\R} = O_{\Pr}(n^{-1/2} \log n)$. 
Hence 
\begin{align}
&\left \| \frac{1}{n} \sum_{j=1}^{n} Y_{n,j}^{2} \{ \hat{K}_{n}((\cdot-Y_{n,j})/h) - K_{n}((\cdot-Y_{n,j})/h)\} \right \|_{I} \notag \\
&\qquad \leq \underbrace{\left ( \frac{1}{n} \sum_{j=1}^{n} Y_{n,j}^{2} \right )}_{=O_{\Pr}(\Delta)} \| \hat{K}_{n} - K_{n} \|_{\R} =O_{\Pr} (n^{-1/2} \Delta \log n), \label{eq: intermediate}
\end{align}
and likewise
\[
\left \| \frac{1}{n} \sum_{j=1}^{n} Y_{n,j}^{4} \{ \hat{K}^{2}_{n}((\cdot-Y_{n,j})/h) - K_{n}^{2}((\cdot-Y_{n,j})/h)\} \right \|_{I} = O_{\Pr} (n^{-1/2} \Delta \log n).
\]
Since $\| n^{-1}\sum_{j=1}^{n} Y_{n,j}^{2}K_{n}((x-Y_{n,j})/h) \|_{I} \lesssim n^{-1}\sum_{j=1}^{n} Y_{n,j}^{2}  = O_{\Pr}(\Delta)$, we have that 
\[
\hat{s}_{n}^{2}(x)=\underbrace{\frac{1}{n} \sum_{j=1}^{n} Y_{n,j}^{4}K_{n}^{2}((x-Y_{n,j})/h) - \left \{ \frac{1}{n} \sum_{j=1}^{n} Y_{n,j}^{2}K_{n}((x-Y_{n,j})/h) \right \}^{2}}_{=: \tilde{s}_{n}^{2}(x)}+ O_{\Pr}(n^{-1/2} \Delta \log n)
\]
uniformly in $x \in I$. Furthermore, since $\inf_{x \in I}s_{n}^{2}(x) \gtrsim \Delta h$ and 
\[
\frac{n^{-1/2}\Delta \log n}{\Delta h} = n^{-1/2} h^{-1} \log n \lesssim n^{-1/2} \Delta^{-1/2} \log n \ll (\log n)^{-1},
\]
it remains to prove that $\| \tilde{s}_{n}^{2}/s_{n}^{2} - 1 \|_{I} = o_{\Pr}\{ (\log n)^{-1} \}$. To this end, 
since $\| \Ep[ Y_{n,1}^{2}K_{n}((\cdot-Y_{n,1})/h) ]/s_{n}  \|_{I}\lesssim \Delta/\sqrt{\Delta h} = \sqrt{\Delta /h} \lesssim \Delta^{1/4}$, it suffices to prove that 
\begin{align}
&\left \| \frac{1}{n} \sum_{j=1}^{n} \{ f^{2}(Y_{n,j}) - \Ep[f^{2}(Y_{n,1})]\} \right \|_{\mF_{n}} = o_{\Pr} \{ (\log n)^{-1} \}, \ \text{and} \label{eq: variance bound 1} \\
&\left \| \frac{1}{n} \sum_{j=1}^{n}  \{ f(Y_{n,j}) - \Ep[f(Y_{n,1})]\} \right \|_{\mF_{n}} = o_{\Pr} \{ (\log n)^{-1/2} \}. \label{eq: variance bound 2}
\end{align}
To prove (\ref{eq: variance bound 1}), we make use of Corollary 5.1 in \cite{ChChKa14a}. 
Let $\mF_{n}^{2}=\{ f^{2} : f \in \mF_{n} \}$.
Observe that $\sup_{f \in \mF_{n}^{2}} \Ep[ f^{2}(Y_{n,1})] = \sup_{f \in \mF_{n}} \Ep[ f_{n}^{4}(Y_{n,1}) ] \lesssim 1/(\Delta h)$, and 
$\sup_{f \in \mF_{n}^{2}} \| f \|_{\R} = \sup_{f \in \mF_{n}} \| f^{2} \|_{\R} \lesssim 1/(\Delta h)$. 
From the covering number bound (\ref{eq: VC type}) together with Corollary A.1 in \cite{ChChKa14a}, there exist constants $A_{2},v_{2} > 0$ independent of $n$ such that 
\[
\sup_{Q} N(\mF_{n}^{2},\| \cdot \|_{Q,2}, \varepsilon/(\Delta h)) \leq (A_{2}/\varepsilon)^{v_{2}}, \ 0 < \forall \varepsilon \leq 1,
\]
where $\sup_{Q}$ is taken over all finitely discrete distributions on $\R$. 
Therefore, from Corollary 5.1 in \cite{ChChKa14a}, the expectation of the left hand side on (\ref{eq: variance bound 1}) is bounded by
\[
\lesssim \sqrt{\frac{\log n}{n \Delta h}} + \frac{\log n}{n\Delta h} \ll (\log n)^{-1}.
\]
For (\ref{eq: variance bound 2}), from the covering number bound (\ref{eq: VC type}), 
Theorem 2.14.1 in \cite{vaWe96} shows that the expectation of the left hand side on (\ref{eq: variance bound 2}) is bounded by 
\[
\lesssim (n \Delta h)^{-1/2} \ll (\log n)^{-1/2}.
\]
This completes the proof. 
\end{proof}

\begin{proof}[Proof of Theorem \ref{thm: bootstrap}]
We separately prove  the theorem for the multiplier and empirical bootstraps. 

\medskip

\textbf{Multiplier bootstrap case}: We first verify that 
\begin{align}
&\sum_{j=1}^{n} \xi_{j} \left \{ Y_{n,j}^{2}\hat{K}_{n}((x-Y_{n,j})/h) - n^{-1} {\textstyle \sum}_{j'=1}^{n} Y_{n,j'}^{2}\hat{K}_{n}((x-Y_{n,j'})/h) \right \} \notag \\
&\quad =\sum_{j=1}^{n} \xi_{j} \left \{ Y_{n,j}^{2}K_{n}((x-Y_{n,j})/h) - n^{-1} {\textstyle \sum}_{j'=1}^{n} Y_{n,j'}^{2}K_{n}((x-Y_{n,j'})/h) \right \} \notag \\
&\qquad + o_{\Pr}(\sqrt{n\Delta h/\log n})
\label{eq: estimated multiplier}
\end{align}
uniformly in $x \in I$. From (\ref{eq: intermediate}), 
\[
\left \| \left ( \sum_{j=1}^{n} \xi_{j} \right ) \left [ \frac{1}{n} \sum_{j=1}^{n} Y_{n,j}^{2} \{ \hat{K}_{n}((\cdot-Y_{n,j})/h)-K_{n}((\cdot-Y_{n,j})/h) \} \right ] \right \|_{I} = O_{\Pr}(\Delta \log n). 
\]
Next, observe that 
\begin{align*}
&\left \| \sum_{j=1}^{n} \xi_{j} Y_{n,j}^{2} \{ \hat{K}_{n}((\cdot-Y_{n,j})/h)-K_{n}((\cdot-Y_{n,j})/h) \}  \right \|_{I} \\
&\quad \lesssim \int_{-1}^{1} \left | \sum_{j=1}^{n} \xi_{j} Y_{n,j}^{2} e^{iuY_{n,j}/h} \right | \left | \frac{1}{\hat{\varphi}_{\Delta}(u/h)} - \frac{1}{\varphi_{\Delta}(u/h)} \right | du \\
&\quad \leq O_{\Pr} (n^{-1/2} \log n) \int_{-1}^{1} \left | \sum_{j=1}^{n} \xi_{j} Y_{n,j}^{2} e^{iuY_{n,j}/h} \right | du, \ \text{and} \\
&\Ep \left [  \left | \sum_{j=1}^{n} \xi_{j} Y_{n,j}^{2} e^{iuY_{n,j}/h} \right | \right ] \leq \Ep \left [ \sqrt{ \sum_{j=1}^{n} Y_{n,j}^{4}} \right ] \lesssim \sqrt{n \Delta}, 
\end{align*}
which yields that 
\[
\left \| \sum_{j=1}^{n} \xi_{j} Y_{n,j}^{2} \{ \hat{K}_{n}((\cdot-Y_{n,j})/h)-K_{n}((\cdot-Y_{n,j})/h) \}  \right \|_{I} = O_{\Pr} (\sqrt{\Delta} \log n).
\]
Therefore, we have proved (\ref{eq: estimated multiplier}). 

Now, since $\|1/\hat{s}_{n}\|_{I} =O_{\Pr}(1/\sqrt{\Delta h})$ by Lemma \ref{lem: scaling estimate}, we have that 
\begin{align}
&\hat{Z}_{n}^{MB} (x)= \frac{1}{\hat{s}_{n}(x)\sqrt{n}} \sum_{j=1}^{n} \xi_{j}\left \{ Y_{n,j}^{2}K_{n}((x-Y_{n,j})/h) - n^{-1} {\textstyle \sum}_{j'=1}^{n} Y_{n,j'}^{2}K_{n}((x-Y_{n,j'})/h) \right \} \notag \\
&\qquad \qquad \qquad + o_{\Pr} \{ (\log n)^{-1/2} \} \notag \\
&=[ 1 + o_{\Pr}\{ (\log n)^{-1} \} ] \underbrace{\frac{1}{s_{n}(x) \sqrt{n}} \sum_{j=1}^{n} \xi_{j}\left \{ Y_{n,j}^{2}K_{n}((x-Y_{n,j})/h) - n^{-1} {\textstyle \sum}_{j'=1}^{n} Y_{n,j'}^{2}K_{n}((x-Y_{n,j'})/h) \right \}}_{=:Z_{n}^{MB}(x)} \notag \\
&\qquad +  o_{\Pr}\{ (\log n)^{-1/2} \} \label{eq: expansion}
\end{align}
uniformly in $x \in I$. We wish to show that 
\begin{equation}
\sup_{z \in \R} \left | \Pr \{ \| Z_{n}^{MB} \|_{I} \leq z \mid \mathcal{D}_{n} \} - \Pr \{ \| Z_{n}^{G} \|_{I} \leq z \} \right | \stackrel{\Pr}{\to} 0.
\label{eq: MB convergence}
\end{equation}
To this end, we shall apply Theorem 2.2 in \cite{ChChKa16} to $\mF_{n} \cup (-\mF_{n})$. Let 
\[
\mathbb{G}_{n}^{\xi} (f)=\frac{1}{\sqrt{n}} \sum_{j=1}^{n} \xi_{j}\left \{  f(Y_{n,j}) - n^{-1} {\textstyle \sum}_{j'=1}^{n} f(Y_{n,j'})  \right \}, \ f \in \mF_{n}.
\]
Applying Theorem 2.2 in \cite{ChChKa16} to $\mF_{n} \cup (-\mF_{n})$ with $B(f) \equiv 0, A \lesssim 1, v \lesssim 1, \sigma \sim 1, b \lesssim 1/\sqrt{\Delta h}, \gamma \lesssim 1/\log n$, and sufficiently large $q$ (in the notation used in the cited theorem), yields that there exists a random variable $V_{n}^{\xi}$ whose conditional distribution given $\mathcal{D}_{n}$ is identical to the distribution of $\| \mathbb{U}_{n} \|_{\mF_{n}} (= \| Z_{n}^{G} \|_{I})$, i.e., $\Pr \{ V_{n}^{\xi} \leq z \mid \mathcal{D}_{n} \} = \Pr \{ \| Z_{n}^{G} \|_{I} \leq z \}$ for all $z \in \R$ almost surely, and such that 
\[
\left | \| \mathbb{G}_{n}^{\xi} \|_{\mF_{n}} - V_{n}^{\xi} \right | = O_{\Pr} \left \{ \frac{(\log n)^{2+1/q}}{n^{1/2-1/q} \sqrt{\Delta h}} + \frac{(\log n)^{7/4 + 1/q}}{(n\Delta h)^{1/4}}\right \} = o_{\Pr} \{ (\log n)^{-1/2} \},
\]
which implies that there exists a sequence of constants $\varepsilon_{n} \downarrow 0$ such that 
\[
\Pr \{ | \| \mathbb{G}_{n}^{\xi} \|_{\mF_{n}} - V_{n}^{\xi} | > \varepsilon_{n} (\log n)^{-1/2} \mid \mathcal{D}_{n} \} \stackrel{\Pr}{\to} 0.
\]
Since $\| \mathbb{G}_{n}^{\xi} \|_{\mF_{n}} = \| Z_{n}^{MB} \|_{I}$, we have that 
\begin{align*}
\Pr \{ \| Z_{n}^{MB} \|_{I} \leq z \mid \mathcal{D}_{n} \} &\leq \Pr \{ V_{n}^{\xi} \leq z + \varepsilon_{n} (\log n)^{-1/2} \mid \mathcal{D}_{n} \} + o_{\Pr}(1) \\
&= \Pr \{ \| Z_{n}^{G} \|_{I} \leq z + \varepsilon_{n} (\log n)^{-1/2} \} + o_{\Pr} (1)
\end{align*}
uniformly in $z \in \R$. From the anti-concentration inequality (\ref{eq: AC}) together with the bound $\Ep [ \| Z_{n}^{G} \|_{I}] \lesssim \sqrt{\log n}$, we conclude that 
\[
\Pr \{ \| Z_{n}^{MB} \|_{I} \leq z \mid \mathcal{D}_{n} \} \leq \Pr \{ \| Z_{n}^{G} \|_{I} \leq z \} + o_{\Pr}(1)
\]
uniformly in $z \in \R$. Likewise, we have $\Pr \{ \| Z_{n}^{MB} \|_{I} \leq z \mid \mathcal{D}_{n} \} \geq \Pr \{ \| Z_{n}^{G} \|_{I} \leq z \} - o_{\Pr}(1)$ uniformly in $z \in \R$. Hence, we have proved (\ref{eq: MB convergence}). 

From (\ref{eq: MB convergence}) together with the bound $\Ep [ \| Z_{n}^{G} \|_{I}] \lesssim \sqrt{\log n}$, we see that $\| Z_{n}^{MB} \|_{I} = O_{\Pr}(\sqrt{\log n})$. So, from (\ref{eq: expansion}), we have that 
\[
\hat{Z}_{n}^{MB}(x) = Z_{n}^{MB}(x) + o_{\Pr}\{ (\log n)^{-1/2} \}
\]
uniformly in $x \in I$. In view of the proof of (\ref{eq: MB convergence}), we conclude that 
\begin{equation}
\sup_{z \in \R} | \Pr \{ \| \hat{Z}_{n}^{MB} \|_{I} \leq z \mid \mathcal{D}_{n} \} - \Pr \{ \| Z_{n}^{G} \|_{I} \leq z \} | \stackrel{\Pr}{\to} 0. \label{eq: MB convergence 2}
\end{equation}

Now, we wish to show that $\Pr \{ \rho (x) \in \hat{\mathcal{C}}^{MB}_{1-\tau} (x) \ \forall x \in I \} \to 1-\tau$. We begin with noting that 
\[
\rho(x) \in \hat{\mathcal{C}}^{MB}_{1-\tau} \ \forall x \in I \Leftrightarrow \| \sqrt{n}\Delta hx^{2} (\hat{\rho} - \rho)/\hat{s}_{n} \|_{I} \leq \hat{c}_{n}^{MB}(1-\tau).
\]
Recall that from the conclusion of Theorem \ref{thm: Gaussian approximation} together with the bound $\Ep [\| Z_{n}^{G} \|_{I}] \lesssim \sqrt{\log n}$, we have that $\| Z_{n} \|_{I} = O_{\Pr}(\sqrt{\log n})$. 
Observe that 
\begin{align*}
\frac{\sqrt{n}\Delta h x^{2} (\hat{\rho} (x)- \rho(x))}{\hat{s}_{n}(x)} &= \frac{s_{n}(x)}{\hat{s}_{n}(x)} \cdot \frac{\sqrt{n}\Delta x^{2}(\hat{\rho}(x) - \rho(x))}{s_{n}(x)} \\
&= \frac{s_{n}(x)}{\hat{s}_{n}(x)}\left [ Z_{n}(x) + o_{\Pr} \{ (\log n)^{-1/2} \} \right] \quad (\text{from (\ref{eq: linearization})}) \\
&= \left [ 1+o_{\Pr} \{ (\log n)^{-1} \} \right ]\left [ Z_{n}(x) + o_{\Pr} \{(\log n)^{-1/2} \} \right ]\quad (\text{from Lemma \ref{lem: scaling estimate}}) \\
&=Z_{n}(x) + o_{\Pr}\{(\log n)^{-1/2} \} \quad (\text{from $\| Z_{n} \|_{I} = O_{\Pr}(\sqrt{\log n})$}) 
\end{align*}
uniformly in $x \in I$. Hence, using the conclusion of Theorem \ref{thm: Gaussian approximation} together with the anti-concentration inequality (\ref{eq: AC}), we have that 
\begin{equation}
\sup_{z \in \R}| \Pr \{ \| \sqrt{n}\Delta h x^{2} (\hat{\rho} - \rho)/\hat{s}_{n} \|_{I} \leq z \} - \Pr \{ \| Z_{n}^{G} \|_{I} \leq z \} | \to 0.
\label{eq: Gaussian approximation}
\end{equation}
From the result (\ref{eq: MB convergence 2}), using an argument similar to Step 3 in the proof of Theorem 2 in \cite{KaSa16}, we can find a sequence of constants $\varepsilon_{n}' \downarrow 0$ such that 
\begin{equation}
c_{n}^{G}(1-\tau - \varepsilon_{n}') \leq \hat{c}_{n}^{MB}(1-\tau) \leq c_{n}^{G}(1-\tau + \varepsilon_{n}')
\label{eq: bound on critical value}
\end{equation}
with probability approaching one. Therefore, 
\begin{align*}
&\Pr \{ \| \sqrt{n}\Delta h x^{2} (\hat{\rho} - \rho)/\hat{s}_{n} \|_{I} \leq \hat{c}_{n}^{MB}(1-\tau) \}  \\
& \leq \Pr \{  \| \sqrt{n}\Delta hx^{2} (\hat{\rho} - \rho)/\hat{s}_{n} \|_{I} \leq c_{n}^{G}(1-\tau + \varepsilon_{n}') \} + o(1) \quad (\text{from (\ref{eq: bound on critical value})}) \\
&= \underbrace{\Pr \{ \| Z_{n}^{G} \|_{I} \leq c_{n}(1-\tau + \varepsilon_{n}') \}}_{=1-\tau+\varepsilon_{n}'} + o(1) \quad (\text{from (\ref{eq: Gaussian approximation})}) \\
&=1-\tau+o(1).
\end{align*}
Note that from the anti-concentration inequality, the distribution function of $\| Z_{n}^{G} \|_{I}$ is continuous, so that the equality $\Pr \{ \| Z_{n}^{G} \|_{I} \leq c_{n}^{G}(1-\tau + \varepsilon_{n}') \} = 1-\tau + \varepsilon_{n}'$ holds. Likewise, we have $\Pr \{ \| \sqrt{n}\Delta hx^{2} (\hat{\rho} - \rho)/\hat{s}_{n} \|_{I} \leq \hat{c}_{n}^{MB}(1-\tau) \} \geq 1-\tau - o(1)$.

Finally, the Borell-Sudakov-Tsirelson inequality \citep[][Lemma A.2.2]{vaWe96} yields that 
\[
c_{n}^{G}(1-\tau+\varepsilon_{n}') \lesssim \Ep[\| Z_{n}^{G} \|_{I}] + \sqrt{\log (1/(\tau-\varepsilon_{n}'))} \lesssim \sqrt{\log n},
\]
which implies that $\hat{c}_{n}^{MB}(1-\tau) = O_{\Pr}(\sqrt{\log n})$ from (\ref{eq: bound on critical value}). Therefore, the supremum width of the band $\hat{\mathcal{C}}_{1-\tau}^{MB}$ is 
\[
2 \sup_{x \in I} \frac{\hat{s}_{n}(x)}{x^{2}\sqrt{n} \Delta h} \hat{c}_{n}^{MB}(1-\tau)  \lesssim \{ 1+o_{\Pr}(1) \} \frac{\sup_{x \in I} s_{n}(x)}{\sqrt{n}\Delta h} \hat{c}_{n}^{MB}(1-\tau) = O_{\Pr}\{ (n\Delta h)^{-1/2} \sqrt{\log n}\}, 
\]
where the bound $\sup_{x \in I} s_{n}(x) \lesssim \sqrt{\Delta h}$ follows from (\ref{eq: variance upper bound}). 
This completes the proof for the multiplier bootstrap case.

\medskip

\textbf{Empirical bootstrap case}: 
Note that the bootstrap process $\hat{Z}_{n}^{EB}$ can be expressed as 
\[
\hat{Z}_{n}^{EB}(x) = \frac{1}{\hat{s}_{n}(x) \sqrt{n}} \sum_{j=1}^{n} (M_{n,j} - 1) Y_{n,j}^{2} \hat{K}_{n}((x-Y_{n,j})/h), \ x \in I
\]
where each $M_{n,j}$ is the number of times that $Y_{n,j}$ is ``redrawn'' in the bootstrap sample, and the vector $(M_{n,1},\dots,M_{n,n})$ is multinomially distributed with parameters $n$ and (probabilities) $(1/n,\dots,1/n)$ independent of the data $\mathcal{D}_{n}$ \citep[cf.][Section 3.6]{vaWe96}. Given this expression, the proof for the empirical bootstrap case is almost identical to that for the multiplier bootstrap case, where we use Theorem 2.3 in \cite{ChChKa16} instead of their Theorem 2.2. We omit the detail for brevity. 
\end{proof}

\section{Additional proofs}

\subsection{Proof of Lemma \ref{lem: TRV}}
\label{sec: appendix c}

The L\'evy process $L$ has the following decomposition (L\'{e}vy-It\^{o} decomposition)
\begin{align*}
L_{t} =  \gamma t + \sigma B_{t} + \int_{0}^{t} \int_{\R} x1_{[-1,1]}(x) (\mu^{L} - \nu^{L})(ds, dx) + \int_{0}^{t} \int_{\R} x1_{[-1,1]^{c}}(x) \mu^{L}(ds, dx),
\end{align*}
where $B = (B_{t})_{t \geq 0}$ is a standard Brownian motion and $\mu^{L}$ is a Poisson random measure on $[0,\infty) \times \R$, independent of $B$, with intensity measure $\nu^{L}(dt, dx) = dt \nu(dx)$. Depending on the value of $\alpha \in (0,2)$, we set  
\begin{align*}
L^{(1)} &= L - L^{(2)},\  
L^{(2)}_{t} = 
\begin{cases}
 \int_{0}^{t} \int_{\R}x\mu^{L}(ds,dx) & \text{ if $\alpha \in (0,1)$} \\
 \int_{0}^{t} \int_{\R}x(\mu^{L}  - \nu^{L})(ds,dx) & \text{ if $\alpha \in [1,2)$}
\end{cases}
.
\end{align*}
Let $L^{(1)}_{n,j} = L^{(1)}_{j\Delta} - L^{(1)}_{(j-1)\Delta}$. Consider the following decomposition for $\hat{\sigma}^{2}_{TRV} - \sigma^2$: 
\begin{align*}
\hat{\sigma}^{2}_{TRV} - \sigma^2 &= \underbrace{{1 \over n\Delta}\sum_{j=1}^{n}(L^{(1)}_{n,j})^{2} - \sigma^2}_{=:A_{1,n}} -\underbrace{{1 \over n\Delta}\sum_{j=1}^{n}(L^{(1)}_{n,j})^{2}1_{\{|L^{(1)}_{n,j}| > \alpha_{0}\Delta^{\theta_{0}}\}}}_{=: A_{2,n}} \\
&\quad + \underbrace{\hat{\sigma}^{2}_{TRV} - {1 \over n\Delta}\sum_{j=1}^{n}(L^{(1)}_{n,j})^{2}1_{\{|L^{(1)}_{n,j}| \leq \alpha_{0}\Delta^{\theta_{0}}\}}}_{=:A_{3,n}}.
\end{align*}
We shall evaluate $A_{k,n}$ for $k=1,2,3$. Since $L^{(1)}_{n,j} \sim N(\Delta \gamma_{0}, \Delta \sigma^{2})$ with
\[
\gamma_{0} = 
\begin{cases}
\gamma - \int_{[-1,1]}x \nu(dx) & \text{if $\alpha \in (0,1)$} \\
\gamma + \int_{[-1,1]^{c}}x\nu(dx) & \text{if $\alpha \in [1,2)$} 
\end{cases}
,
\]
we have that
\begin{align*}
\Ep[(L^{(1)}_{n,j})^{2}] &= \Delta^2\gamma^{2}_{0} + \Delta \sigma^2,\ \mathrm{E}[(L^{(1)}_{n,j})^{4}] = \Delta^4\gamma^{4}_{0} + 6\Delta^3 \gamma^{2}_{0} \sigma^2 + 3\Delta^2 \sigma^4,
\end{align*}
which yields that
\[
A_{1,n} =  \underbrace{{1 \over n\Delta}\sum_{j=1}^{n}\left \{ (L^{(1)}_{n,j})^{2} - \mathrm{E}[(L^{(1)}_{n,j})^{2}]\right\}}_{=O_{\Pr}(n^{-1/2})}+ \underbrace{{1 \over n\Delta}\sum_{j=1}^{n}\mathrm{E}[(L^{(1)}_{n,j})^{2}] - \sigma^2 }_{=O(\Delta)}= O_{\mathrm{P}}(n^{-1/2} + \Delta). 
\]
Furthermore, for any $q >0$, we have that $\Ep[|L^{(1)}_{n,j}|^{q}] \lesssim \Delta^{q/2}$, which
yields that 
\[
\Ep \left [ (L^{(1)}_{n,j})^{2}1_{\{|L^{(1)}_{n,j}|>\alpha_{0}\Delta^{\theta_{0}}\}} \right ] \lesssim \Delta^{1 + q(1-2\theta_{0})/2}.
\]
Taking $q = 2/(1-2\theta_{0})$, we have that $\Ep[|A_{2,n}|] \lesssim \Delta$. Finally, applying Lemma 13.2.6 in \cite{JP(2012)} with $k = 1, F(x) = x^{2}, s'=2, m=s=p'=1$, and $\theta = 0$ in the notation used in the lemma, we conclude  that $\Ep[|A_{3,n}|] \lesssim \Delta^{(2-\alpha)\theta_{0}}$. These  estimates yield the desired result. \qed

\subsection{Proof of Proposition \ref{prop: condition (ii)}}

%
Since $L_{t} - bt$ is a also L\'{e}vy process with the same L\'{e}vy density $\rho$ as $L_{t}$, without loss of generality, we may assume $b = 0$, and we shall verify Condition (ii) for $P_{\Delta,b} = P_{\Delta}$. 
In this proof,
slightly abusing notation, we shall use the same symbol for a measure and its Lebesgue density if the latter exists. Let $\| \cdot \|_{L^{1}}$ denote the $L^{1}$-norm with respect to the Lebesgue measure. 
We first note that since $yP_{\Delta}$ is absolutely continuous, for any integer $\ell > 1$, $y^{\ell}P_{\Delta}$ is also absolutely continuous with density $(y^{\ell}P_{\Delta})(y) = y^{\ell-1}\left (  (y P_{\Delta})(y) \right )$ (this means that the density of the signed measure $y^{\ell} P_{\Delta}(dy)$ is given by the multiple of $y^{\ell-1}$ with the Lebesgue density of the signed measure $yP_{\Delta}(dy)$). 
Furthermore, boundedness of $x^{4}\rho$ ensures that $\int_{[-1,1]^{c}} x^{2}\rho(x) dx < \infty$, which in turn ensures that $P_{\Delta}$ has finite second moment. In what follows, we will freely use some basic results on convolutions and Fourier transforms of finite signed Borel measures on $\R$; cf. \citet[][Section 8.6]{Fo99}.

\medskip

\textbf{Verification of $\| y^{4} P_{\Delta} \|_{\R} \lesssim \Delta$}: The proof is similar to that of Proposition 14 in \cite{NiReSoTr16}. 
We first note that $\varphi_{\Delta} = \varphi_{\Delta/2}^{2}$ by infinite divisibility. Using this property, we have that $\varphi_{\Delta}'' = \{\Delta \psi'' + (\Delta \psi'')^{2}\}\varphi_{\Delta} = \Delta \psi'' \varphi_{\Delta} +4\{ (\Delta/2) \psi' \varphi_{\Delta/2}\}^{2} = \Delta \psi'' \varphi_{\Delta} + 4 \{( \varphi_{\Delta/2})'\}^{2}$, i.e., 
\[
\int_{\R} e^{iuy} y^{2} P_{\Delta}(dy) = \Delta \underbrace{\left (\sigma^{2} + \int_{\R} e^{iuy} y^{2} \nu(dy) \right)}_{=\int_{\R} e^{ity} (\sigma^{2} \delta_{0} + y^{2}\nu)(dy)} \int_{\R} e^{iuy} P_{\Delta}(dy) + 4 \left ( \int_{\R} e^{iuy} y P_{\Delta/2}(dy) \right )^{2} .
\]
Applying the Fourier inversion, we have that 
\begin{equation}
y^{2}P_{\Delta} = \Delta \nu_{\sigma}*P_{\Delta} + 4 (yP_{\Delta/2})*(yP_{\Delta/2}),
\label{eq: second moment}
\end{equation}
where $\nu_{\sigma} = \sigma^{2}\delta_{0}+y^{2}\nu$ and $*$ denotes the convolution. 
Using the rule $y(P*Q) = (yP)*Q + P*(yQ)$ for finite signed Borel measures $P,Q$ on $\R$ such that $\int_{\R}|y| |P|(dy) < \infty$ and $\int_{\R} |y| |Q| (dy) < \infty$, we have that 
\begin{align*}
y^{3}P_{\Delta} &= \Delta \left \{ (y^{3}\nu) * P_{\Delta} + \nu_{\sigma} * (yP_{\Delta}) \right \}  + 8 (yP_{\Delta/2})*(y^{2}P_{\Delta/2}), \ \text{and} \\
y^{4}P_{\Delta} &= \Delta \left \{ (y^{4}\nu)*P_{\Delta} + 2(y^{3}\nu)*(yP_{\Delta}) +  \nu_{\sigma}*(y^{2}P_{\Delta}) \right \} \\
&\quad + 8 (y^{3}P_{\Delta/2})*(yP_{\Delta/2}) + 8 (y^{2}P_{\Delta/2}) * (y^{2}P_{\Delta/2}).
\end{align*}
Since $\| y^{3} \nu \|_{\R} < \infty, P_{\Delta}(\R) = 1, \| yP_{\Delta} \|_{\R} \lesssim \log (1/\Delta)$ (by assumption), $\nu_{\sigma} (\R) < \infty$, and $\| y^{2}P_{\Delta} \|_{L^{1}} = \Ep[ L_{\Delta}^{2} ] \lesssim \Delta$, we have that 
\[
\| y^{3} P_{\Delta} \|_{\R} \lesssim \Delta \left \{ \| y^{3}\nu \|_{\R} P_{\Delta}(\R) + \| yP_{\Delta} \|_{\R} \nu_{\sigma}(\R) \right \} + \| yP_{\Delta} \|_{\R} \| y^{2}P_{\Delta} \|_{L^{1}} \lesssim \Delta \log (1/\Delta).
\]

Now, if  $\sigma > 0$, then we have that $\| y P_{\Delta} \|_{\R} \lesssim 1$ (cf. Remark \ref{rem: condition (ii)}), so that $\| y^{2} P_{\Delta} \|_{\R} \lesssim \| y P_{\Delta} \|_{\R} \vee \| y^{3} P_{\Delta} \|_{\R} \lesssim 1$. 
On the other hand, if $\sigma = 0$, then from (\ref{eq: second moment}), we have that 
\[
\| y^{2} P_{\Delta} \|_{\R} \lesssim \Delta \| y^{2} \nu \|_{\R} P_{\Delta}(\R) + \| y P_{\Delta/2} \|_{\R} \| y P_{\Delta/2} \|_{L^{1}} \lesssim \Delta + \Delta^{1/2} \log (1/\Delta) \lesssim 1,
\]
where we have used that $\| y P_{\Delta/2} \|_{L^{1}} = \Ep[|L_{\Delta/2}|] \leq \sqrt{\Ep[ L_{\Delta/2}^{2} ]} \lesssim \Delta^{1/2}$. In either case, we have that $\| y^{2} P_{\Delta} \|_{\R} \lesssim 1$.

For $y^{4} P_{\Delta}$, we have that
\begin{align*}
\| y^{4} P_{\Delta} \|_{\R} &\lesssim \Delta \left \{ \| y^{4} \nu \|_{\R} P_{\Delta}(\R) + \| y^{3}\nu \|_{\R} \| y P_{\Delta} \|_{L^{1}} + \| y^{2}P_{\Delta} \|_{\R}\nu_{\sigma}(\R) \right \} \\
&\quad + \| y^{3} P_{\Delta/2} \|_{\R} \| yP_{\Delta/2} \|_{L^{1}} + \| y^{2}P_{\Delta} \|_{\R} \| y^{2}P_{\Delta} \|_{L^{1}} \\
&\lesssim \Delta (1+\Delta^{1/2}) + \Delta^{3/2} \log (1/\Delta) + \Delta \lesssim \Delta.  
\end{align*}

\medskip

\textbf{Verification of $\inf_{y \in I^{\varepsilon_{0}}} (y^{4}P_{\Delta})(y) \gtrsim \Delta$ for some $\varepsilon_{0} > 0$}. We divide the proof into two steps.

\medskip

\textbf{Step 1}. We first consider the case where $\sigma=0, \| x \nu \|_{\R} < \infty$, and $\gamma_{c} - \int_{\R} x \nu(dx) = 0$ ($\| x \nu \|_{\R} < \infty$ ensures that $\int_{\R} |x| \nu (dx) < \infty$). In this case, the characteristic exponent  $\psi(u)$ is 
\[
\psi (u) = \int_{\R} (e^{iux}-1) \nu (dx).
\]
Observe that $\psi'(u) = i \int_{\R} e^{iux} x \nu (dx)$, and applying the Fourier inversion to $i^{-1}\varphi_{\Delta}' = \Delta i^{-1}\psi'\varphi_{\Delta}$, we have that $yP_{\Delta} = \Delta(y\nu) * P_{\Delta}$, so that $\| y P_{\Delta} \|_{\R} \leq  \Delta \| y \nu \|_{\R} P_{\Delta}(\R) \lesssim \Delta$. From (\ref{eq: second moment}), 
\[
y^{2}P_{\Delta} =\Delta (y^{2}\nu)*P_{\Delta} + 4 (yP_{\Delta/2})*(yP_{\Delta/2}),
\]
and $\| (yP_{\Delta/2})*(yP_{\Delta/2}) \|_{\R} \leq \| yP_{\Delta/2} \|_{\R} \| yP_{\Delta/2} \|_{L^{1}} \ll \Delta$. Since $P_{\Delta} ([-\varepsilon,\varepsilon]^{c}) = \Pr(|L_{\Delta}| > \varepsilon) \leq \varepsilon^{-2}\Ep[L_{\Delta}^{2}] \lesssim \Delta$ for any $\varepsilon > 0$ by Markov's inequality, we have that $P_{\Delta}([-\varepsilon,\varepsilon]) = 1-O(\Delta)$. 
So, 
\[
\inf_{y \in I^{\varepsilon_{1}/2}} ((y^{2}\nu)*P_{\Delta}) (y) \geq P_{\Delta}([-\varepsilon_{1}/2,\varepsilon_{1}/2]) \inf_{y \in I^{\varepsilon_{1}}}
 (y^{2}\nu)(y) \gtrsim 1,
\]
 which shows that $\inf_{y \in I^{\varepsilon_{1}/2}} (y^{2}P_{\Delta})(y) \gtrsim \Delta$.
This leads to the desired result with $\varepsilon_{0}=\varepsilon_{1}/2$. 

\medskip

\textbf{Step 2}. Next, we consider the general case. Decompose the L\'{e}vy measure $\nu$ into $\nu = \nu 1_{[-\varepsilon_{1}/4,\varepsilon/4]} + \nu 1_{[-\varepsilon_{1}/4,\varepsilon_{1}/4]^{c}} =: \nu_{1} + \nu_{2}$, and observe that $\nu_{2}$ is a finite, non-zero measure with $\| x \nu_{2} \|_{\R} < \infty$. Then the characteristic exponent $\psi (u)$ can be decomposed as 
\[
\psi (u) = \underbrace{-\frac{\sigma^{2}u^{2}}{2} + iu \gamma_{1} + \int_{\R} (e^{iux} - 1 - iux) \nu_{1}(dx)}_{=:\psi_{1}(u)}+ \underbrace{\int_{\R} (e^{iux} - 1) \nu_{2}(dx)}_{=:\psi_{2}(u)},
\]
where $\gamma_{1} := \gamma_{c} - \int_{\R} x \nu_{2}(dx)$. From this decomposition, we have 
\[
L \stackrel{d}{=} M + N = (M_{t}+N_{t})_{t \geq 0},
\]
where  $M = (M_{t})_{t \geq 0}$ is  a L\'{e}vy process with L\'{e}vy measure $\nu_{1}$, and  $N =(N_{t})_{t \geq 0}$ is a compound Poisson process with jump intensity $\nu_{2}(\R)$ and jump size distribution $\nu_{2}/\nu_{2}(\R)$ independent of $M$. 
Let $Q_{\Delta}, R_{\Delta}$ denote the distributions of $M_{\Delta},N_{\Delta}$, respectively, so that $P_{\Delta} = Q_{\Delta}*R_{\Delta}$. 
Since $R_{\Delta}$ is a compound Poisson distribution with absolutely continuous jump distribution, we obtain the decomposition $R_{\Delta} = R_{\Delta} (\{ 0 \}) \delta_{0} + R_{\Delta}^{ac}$ where $R_{\Delta}^{ac}$ is absolutely continuous, so that $y^{2}P_{\Delta} = R_{\Delta}(\{ 0 \}) y^{2}Q_{\Delta}+ y^{2}(R_{\Delta}^{ac}*Q_{\Delta})$. Since both $y^{2}P_{\Delta}$ and $R_{\Delta}^{ac}*Q_{\Delta}$ are absolutely continuous, so is $y^{2}Q_{\Delta}$, and 
\[
(y^{2}P_{\Delta})(y) \geq y^{2} (R_{\Delta}^{ac}*Q_{\Delta})(y).
\]
Now, since $y^{2}R_{\Delta} = y^{2}R_{\Delta}^{ac}$ and from the result of Step 1, we have that $\inf_{y \in I^{\varepsilon_{1}/4}} (y^{2}R_{\Delta}^{ac}) (y) \gtrsim \Delta$ and so $\inf_{y \in I^{\varepsilon_{1}/4}} R_{\Delta}^{ac}(y) \gtrsim \Delta$. Furthermore, since $Q_{\Delta}([-\varepsilon_{1}/8,\varepsilon_{1}/8]^{c}) = \Pr (|M_{\Delta}| > \varepsilon_{1}/8) \lesssim \Ep[ M_{\Delta}^{2} ]/\varepsilon_{1}^{2} \lesssim \Delta$ by Markov's inequality, we have that 
\[
\inf_{y \in I^{\varepsilon_{1}/8}} (R_{\Delta}^{ac}*Q_{\Delta})(y) \geq Q_{\Delta}([-\varepsilon_{1}/8,\varepsilon_{1}/8]) \inf_{y \in I^{\varepsilon_{1}/4}} R_{\Delta}^{ac}(y) \gtrsim \Delta,
\]
which implies that $\inf_{y \in I^{\varepsilon_{1}/8}} (y^{2}P_{\Delta})(y)  \gtrsim \Delta$. This leads to the desired result with $\varepsilon_{0}=\varepsilon_{1}/8$. \qed

\subsection{Proof of Lemma \ref{lem: infinite variation}}

The proof is a modification of that of \cite{NiReSoTr16}, Proposition 16 Case (iv). We will obey the notational convention used in the proof of Proposition \ref{prop: condition (ii)}. 
We first note that, under the assumption of the lemma, the L\'{e}vy measure is infinite, since for $\varepsilon \in (0,1)$, 
\[
\int_{|x| \leq 1} \rho(x) dx \geq \int_{|x| \leq \varepsilon} \frac{x^{2}}{\varepsilon^{2}} \rho(x) dx \geq \frac{c}{\varepsilon},
\]
and taking $\varepsilon \downarrow 0$ shows that the far left hand side is infinite. 
Then, Theorem 27.7 in \cite{Sa99} yields that $P_{\Delta}$ is absolutely continuous with respect to the Lebesgue measure, and by the Fourier inversion, we have that $\| y P_{\Delta} \|_{\R} \lesssim \| \varphi_{\Delta}' \|_{L^{1}}$. Observe that 
\[
|\varphi_{\Delta}'(u)| \leq \Delta\left \{  |\gamma_{c}| +\left | \int_{\R}(e^{iux} - 1) x \rho (x) dx \right | \right \} e^{-\Delta \int_{\R} (1-\cos (ux)) \rho(x) dx }. 
\]
Since $| e^{iux} - 1 | \leq |ux| \wedge 2$, we have that 
\[
\left | \int_{\R}(e^{iux} - 1) x \rho (x) dx \right | \leq |u| \int_{|x| \leq 1/|u|}  x^{2} \rho(x) dx + 2\int_{|x| > 1/|u|} |x| \rho (x) dx. 
\]
Furthermore, since there exists a small constant $c' > 0$ such that $1-\cos (x) \geq c' x^{2}$ for all $|x| \leq 1$, we have that 
\[
\int_{\R} (1-\cos (ux)) \rho(x) dx \geq c'  u^{2} \int_{|x| \leq 1/|u|}x^{2} \rho(x) dx.
\]
From these estimates, it is seen that $\| \varphi_{\Delta}'1_{[-1,1]} \|_{L^{1}} \lesssim \Delta$. On the other hand, since 
\[
c \leq |u| \int_{|x| \leq 1/|u|}x^{2} \rho(x) dx \leq C \quad \text{and} \quad \int_{|x| > 1/|u|} |x| \rho (x) dx \leq C (1+\log |u|)
\]
for $|u| > 1$, 
we have that 
\[
\| \varphi_{\Delta}' 1_{[-1,1]^{c}} \|_{L^{1}} \lesssim \Delta \int_{[-1,1]^{c}}  (1+\log |u|)  e^{-c'' \Delta |u|} du \lesssim \log (1/\Delta),
\]
where $c'' = c c'$. 
This completes the proof. 
\qed

\subsection{Proof of Lemma \ref{lem: NIG bias}}
Without loss of generality, we may assume $\alpha=1$. Observe that 
\[
e^{\beta x}\int_{0}^{\infty} e^{-t - \frac{x^{2}}{4t}} dt = \int_{0}^{\infty} e^{-t} e^{\beta x - \frac{x^{2}}{4t}} dt = \int_{0}^{\infty} e^{-t (1-\beta^{2})} e^{-\frac{1}{4t} (x-2\beta t)^{2}} dt. 
\]
Pick any $x,y \in \R$, and suppose that $(x-2\beta t)^{2} \leq (y-2\beta t)^{2}$. Then 
\begin{align*}
&\left | e^{-\frac{1}{4t} (x-2\beta t)^{2}} - e^{-\frac{1}{4t} (y-2\beta t)^{2}} \right | \leq 1 - e^{-\frac{1}{4t} \{ (y-2\beta t)^{2} - (x-2\beta t)^{2} \}}
\leq \left ( 1 - e^{-\frac{1}{4t} \{ (y-2\beta t)^{2} - (x-2\beta t)^{2} \}} \right )^{r} \\
&\quad \leq \frac{1}{4^{r}t^{r}}  \{ (y-2\beta t)^{2} - (x-2\beta t)^{2} \}^{r} 
= \frac{1}{4^{r}t^{r}} |y-x|^{r} \underbrace{|y+x-4\beta t|^{r}}_{\leq |y+x|^{r} + 4^{r} \beta^{r} t^{r}} \\
&\quad \leq \frac{1}{4^{r} t^{r}} |y-x|^{r}|y+x|^{r} + \beta^{r} |y-x|^{r}.  
\end{align*}
By symmetry, this inequality holds for any $x,y \in \R$. 

Now, by a change of variables, we have that $(\rho_{\sharp} * (h^{-1}W(\cdot/h)))(x)- \rho_{\sharp}(x) = \int_{\R} \{ \rho_{\sharp}(x-yh) - \rho_{\sharp}(x) \}W(y) dy$.
Observe that, since $t^{-r}$ is integrable around the origin, 
\begin{align*}
| \rho_{\sharp}(x-yh) - \rho_{\sharp}(x) | &\leq \frac{\delta}{\pi} \int_{0}^{\infty} e^{-t(1-\beta^{2})}\underbrace{\left ( \frac{1}{4^{r} t^{r}} |yh|^{r}|2x - yh|^{r} + \beta^{r} |yh|^{r} \right )}_{\leq 4^{-r}t^{-r} |yh|^{2r} + (2^{-r} t^{-r} |x|^{r} + \beta^{r})|yh|^{r}} dt \\
&\lesssim h^{2r} |y|^{2r} + (1+|x|^{r}) h^{r} |y|^{r}
\end{align*}
up to a constant that depends only on $\beta,\delta,r$. Therefore, we conclude that 
\[
\| \rho_{\sharp} * (h^{-1}W(\cdot/h)) - \rho_{\sharp} \|_{I} \lesssim h^{2r} \int_{\R} |y|^{2r} |W(y)| dy + h^{r} \int_{\R} |y|^{r}|W(y)| dy \lesssim h^{r}. 
\]
This completes the proof. \qed

\section{Convergence rates of $\hat{\rho}$ under weighted sup-norm on $\R$}
\label{sec: uniform convergence}

In this section, we study convergence rates of the spectral estimator $\hat{\rho}$ under the weighted sup-norm $\| f \|_{w,\infty} = \sup_{x \in \R} |x^{2}f(x)|$. For this purpose, we work with conditions similar those in Assumption \ref{as: assumption 1}. 

\begin{proposition}
Suppose that $\int_{\R} |x|^{2q} \rho (x) dx < \infty$ for some $q \geq 2$; there exists $b \in \R$ such that the measure $y^{4}P_{\Delta,b}(dy)$ has Lebesgue density $g_{\Delta,b}$ such that  $\| g_{\Delta,b} \|_{\R} \lesssim \Delta$; for some $r > 0$,  $x^{2} \rho$ is $p$-times differentiable, and $(x^{2}\rho)^{(p)}$ is $(r-p)$-H\"{o}lder continuous, where $p$ is the integer such that $p < r \leq p+1$; let $W: \R \to \R$ be an integrable function that satisfies (\ref{eq: kernel}). Furthermore, suppose that $h \gtrsim \Delta^{1/2}$ and $\log \Delta^{-1} \lesssim \log n$. Then 
\[
\| \hat{\rho} - \rho \|_{w,\infty} = O_{\Pr} \left ( \chi_{n} +h^{r} + h^{-1}| \hat{\sigma}^{2} - \sigma^{2} | \right ),
\]
where 
\[
\chi_{n} = n^{-1/2} h^{-1}  \log n  +  n^{-1}\Delta^{-1/2}h^{-1} (\log n)^{2} + (n\Delta h)^{-1/2} (\log n)^{1/2} + (n\Delta)^{-1+1/q} h^{-1} \log n. 
\]
If in addition $\Delta \lesssim (\log n)^{-2}$ and $(n \Delta)^{-1/2+1/q} h^{-1/2} (\log n)^{1/2} \lesssim 1$, then 
\[
\chi_{n} \lesssim (n\Delta h)^{-1/2} (\log n)^{1/2}.
\]
\end{proposition}

\begin{proof}
Recall from the decomposition (\ref{eq: decomposition}) that 
\begin{align*}
x^{2} (\hat{\rho}(x) - \rho(x)) &= [(y^{2}\rho) * (h^{-1}W(\cdot/h))] (x) - x^{2}\rho(x) \\
&\quad + \frac{1}{2\pi} \int_{\R} e^{-iux} (\psi''(u) - \hat{\psi}''(u)) \varphi_{W}(uh) du + (\sigma^{2} - \hat{\sigma}^{2}) h^{-1}W(x/h). 
\end{align*}
Lemma \ref{lem: bias} yields that $\|  [(y^{2}\rho) * (h^{-1}W(\cdot/h))] - x^{2}\rho \|_{\R} \lesssim h^{r}$. Furthermore, the expansion (\ref{eq: linearization}) holds under the assumption of the proposition, and therefore, 
\begin{align*}
\frac{1}{2\pi} \int_{\R} e^{-iux} (\psi''(u) - \hat{\psi}''(u)) \varphi_{W}(uh) du &= \underbrace{\frac{-1}{2\pi \Delta} \int_{\R} e^{-iux} \frac{(\hat{\varphi}_{\Delta} - \varphi_{\Delta})''}{\varphi_{\Delta}} (u) \varphi_{W}(uh) du}_{=\frac{1}{n \Delta h} \sum_{j=1}^{n} \{ Y_{n,j}^{2}K_{n}((x-Y_{n,j})/h) - \Ep[Y_{n,1}^{2}K_{n}((x-Y_{n,1})/h) ] \}} \\
&\quad +O_{\Pr} \{  n^{-1/2} h^{-1}  \log n  +  n^{-1}\Delta^{-1/2}h^{-1} (\log n)^{2} \}
\end{align*}
uniformly in $x \in \R$. It remains to prove that 
\begin{align}
&\left \| \frac{1}{n \Delta h} \sum_{j=1}^{n} \{ Y_{n,j}^{2}K_{n}((\cdot-Y_{n,j})/h) - \Ep[Y_{n,1}^{2}K_{n}((\cdot-Y_{n,1})/h) ] \} \right \|_{\R} \notag \\
&\quad = O_{\Pr} \{ (n\Delta h)^{-1/2} (\log n)^{1/2} + (n\Delta)^{-1+1/q} h^{-1} \log n \}. \label{eq: stochastic bound}
\end{align}
To this end, we shall apply Corollary 5.1 in \cite{ChChKa14a} to the function class
\[
\breve{\mF}_{n} = \{ y \mapsto y^{2} K_{n}((x-y)/h) : x \in \R \}.
\]
Under the present assumption, we still have that $\| K_{n} \|_{\R} \lesssim 1$, and choose a constant $D_{3} > 0$ independent of  $n$ such that $\| K_{n} \|_{\R} \leq D_{3}$. Let $\breve{F}(y) = D_{3} y^{2}$, which is an envelope function for $\breve{\mF}_{n}$. 
From the proof of Lemma  (\ref{lem: BV}), it is seen that $K_{n}$ is of bounded variation with $\mathsf{TV}(K_{n}) \lesssim 1$, so that by Lemma \ref{lem: BV} together with a simple covering number calculation, we have that for some constants $A_{3},v_{3} > 0$ independent of $n$, 
\[
\sup_{Q} N(\breve{\mF}_{n},\| \cdot \|_{Q,2},\varepsilon \| \breve{F} \|_{Q,2}) \leq (A_{3}/\varepsilon)^{v_{3}}, \ 0 < \forall \varepsilon \leq 1.
\]
Observe that $\sup_{x \in \R}\Ep[ Y_{n,1}^{4} K_{n}^{2}((x-Y_{n,1})/h) ] \lesssim \Delta h$ (cf. (\ref{eq: variance upper bound})), and 
\[
\Ep[ \max_{1 \leq j \leq n} Y_{n,j}^{4} ] \leq (\Ep[  \max_{1 \leq j \leq n} |Y_{n,j}|^{2q} ])^{2/q} \leq \left ( \sum_{j=1}^{n} \Ep[|Y_{n,j}|^{2q}] \right )^{2/q} \lesssim (n\Delta)^{2/q},
\]
where we have used that $\Ep[ |Y_{n,1}|^{2q} ] \lesssim \Delta$, which follows from applying Theorem 1.1 in \cite{Fi08} with $f(x) = |x|^{2q}$. 
Therefore, applying Corollary 5.1 in \cite{ChChKa14a} to $\breve{\mF}_{n}$, we conclude that 
\[
\Ep \left [ \left \| \frac{1}{n\Delta h} \sum_{j=1}^{n}\{ f(Y_{n,j}) - \Ep[ f(Y_{n,1}) ] \} \right \|_{\breve{\mF}_{n}} \right ] \lesssim \sqrt{\frac{\log n}{n \Delta h}} + \frac{\log n}{(n \Delta)^{1-1/q}h}, 
\]
which leads to (\ref{eq: stochastic bound}). 

The last assertion is trivial, and the proof is completed. 
\end{proof}

\end{document}